\documentclass[11pt,leqno]{amsart}

\usepackage{amsmath,amsfonts,amsthm,amscd,amssymb,graphicx, url, color,cite}
\usepackage[pagebackref=true]{hyperref}
\usepackage{txfonts} 
\usepackage[left=1 in,top=1 in,right=1 in,bottom=1 in]{geometry}
\usepackage[titletoc,title]{appendix}
\numberwithin{equation}{section}

\usepackage{color}

\newtheorem{theorem}{Theorem}[section]

\newtheorem{lemma}[theorem]{Lemma}

\newtheorem{proposition}[theorem]{Proposition}

\newtheorem{remark}[theorem]{Remark}
\newtheorem{definition}[theorem]{Definition}

\def\d{\dagger}

\def\sgn{\rm sgn}

\newcommand{\R}{\mathbb{R}}

\newcommand{\e}{\varepsilon} 
\newcommand{\F}{\mathbf{F}}

\def\d{\text{d}}

\newcommand{\ba}{\mathbf{a}}

\newcommand{\K}{{\mathbb K}}

\newcommand{\beq}{\begin{equation}}
\newcommand{\eeq}{\end{equation}}
\renewcommand{\div}{\mbox{div}\,}
\newcommand{\A}{\mathbf{A}}

\newcommand{\beqs}{\begin{equation*}}
\newcommand{\eeqs}{\end{equation*}}

\begin{document}
\title[Calder\'on-Zygmund estimates  
for parabolic equations]{ Interior Calder\'on-Zygmund estimates  for solutions to general parabolic equations of $p$-Laplacian type}

\author{Truyen Nguyen}
\address{Department of Mathematics, The University of Akron, Akron, OH 44325, USA}
\email{tnguyen@uakron.edu}
\thanks{The research of the  author is supported in part
by a grant  from the Simons Foundation (\# 318995)} 


\maketitle

\begin{abstract} 
We study  general parabolic equations of the form $u_t  =  \div \A(x,t, u,D u)  +\div(|\mathbf{F}|^{p-2} \mathbf{F})+ f$
whose principal part depends on the solution itself. The vector field $\A$ is assumed to have small mean oscillation in $x$,   
measurable in $t$, Lipschitz continuous in $u$, and its growth in $Du$ is like the $p$-Laplace operator.
We establish interior Calder\'on-Zygmund  estimates for  locally bounded weak solutions to 
the equations when $p>2n/(n+2)$. This is achieved by employing a perturbation method together with developing a 
two-parameter technique and a new compactness argument.  We also make
crucial use  of the intrinsic geometry method by DiBenedetto \cite{D2} and the maximal function free approach by Acerbi and Mingione \cite{AM}.
\end{abstract}
\section{Introduction}
Let  $n\geq 2$ and $Q_6 = B_6(0) \times (-36,36)\subset \R^n\times \R$
be the standard parabolic cylinder centered at the 
origin. The  primary purpose of this paper is to
investigate interior spatial gradient estimates of Calder\'on-Zygmund type for  weak solutions to  quasilinear parabolic equations of the form
\begin{equation}\label{mainPDE}
u_t  =  \div \A(z, u,D u)  +\div(|\mathbf{F}|^{p-2} \mathbf{F})+ f \quad \text{in}\quad Q_6
\end{equation} 
with $z=(x,t)\in Q_6$,  $\mathbf{F}: Q_6 
\to \R^n$,  and  $f: Q_6 
\to \R$.
Let  $\K\subset \R$ be an open interval and consider general vector field
\[
\A = \A(z,u,\xi) : Q_6\times \overline\K\times \R^n \longrightarrow \R^n
\]
which is  a Carath\'eodory map, that is, $\A(z,u,\xi)$ is measurable in $z$ for every $(u,\xi)\in \overline \K \times \R^n$ and continuous 
in $(u,\xi)$ for a.e. $z\in Q_6$. We assume that there exist constants
$\Lambda>0$ and  $1< p<\infty$ such that $\A$  satisfies the following  structural conditions for a.e. $z\in Q_6$, all $ u\in \overline\K$, and all $\xi,\eta\in\R^n$:
\begin{align}
&\big\langle  \A(z,u,\xi) -\A(z,u,\eta), \xi-\eta\big\rangle \geq 
\left \{
\begin{array}{lcll}
  \Lambda^{-1} |\xi-\eta|^p &\text{if}\quad p\geq 2,\\
\Lambda^{-1} \big(1+ |\xi| +|\eta|)^{p-2} |\xi-\eta|^2 &\quad\,\,\text{ if}\quad 1 <p<2,
\end{array}\right. 
\label{structural-reference-1}\\ 
& |\A(z,u,\xi)|  \leq \Lambda (1 + |\xi|^{p-1}),\label{structural-reference-2}\\
&|\A(z,u_1,\xi)-\A(z,u_2,\xi)|  \leq \Lambda |u_1 -u_2|\,\, 
\big(1+ |\xi|^{p-1}\big) \qquad\qquad\quad\qquad \,\,\,\,\forall u_1, u_2\in \overline\K.\label{structural-reference-3}
\end{align}

 The class of equations of the form \eqref{mainPDE} with $\A$ satisfying \eqref{structural-reference-1}--\eqref{structural-reference-3}
 contains the well-known parabolic $p$-Laplace equations. More generally, it includes those of the form 
 \begin{equation}\label{PDE}
u_t  =  \div \big(\ba(x,t)|D u|^{p-2} Du \big)  +\div(|\mathbf{F}|^{p-2} \mathbf{F}) \quad \text{in}\quad Q_6.
\end{equation} 
The  regularity theory for weak solutions of \eqref{PDE} 
is well developed \cite{AM,Bo,BOR,BW1,D2,DF1,DF2,KiL,KuM1,KuM2,La,Mi1,Mi2}. In particular, Calder\'on-Zygmund-type estimates for
\eqref{PDE} were derived in \cite{AM,Mi2} exploiting the essential fact that the equation is  invariant  with respect to the so-called intrinsic geometry \cite{D2}. These generalize previous   results obtained
 for elliptic equations of $p$-Laplacian type \cite{CP,DM,I,KZ}.  
However, there is a great difficulty in studying \eqref{PDE} compared to
its elliptic 
counterpart since it
scales differently in space and time and as a result 
there is no natural maximal function associated to \eqref{PDE} when $p\neq 2$.
To handle this problem, a new  and important maximal function free approach
was developed by Acerbi and Mingione \cite{AM}.
These other key ingredients used in  \cite{AM}  are  the localization method introduced by Kinnunen and Lewis \cite{KiL} and the celebrated $L^\infty$ estimates due to 
DiBenedetto and Friedman \cite{DF1} 
for spatial gradients of solutions to the frozen homogeneous equations. The result and method  in \cite{AM} were
extended further  in recent articles \cite{Bo,BOR} to cover equations of the form $u_t  =  \div \A(x,t,Du) 
+\div(|\mathbf{F}|^{p-2} \mathbf{F}) + f$.

The aim of this paper is to address Calder\'on-Zygmund-type estimates for a new class of parabolic equations whose  principal parts
are allowed to depend on the $u$ variable. We study  general parabolic equations of the form \eqref{mainPDE} which includes equations
describing $p$-harmonic flows. It is worth pointing out that this class of equations is not invariant with respect to the intrinsic geometry 
due to the dependence of $\A$ on $u$. Nevertheless, we are able to establish the following  main result about $L^q$ estimates for $Du$. Hereafter,
 we denote
 $ Q_{\bar z}(r,\theta) := B_r(\bar x)\times (\bar t-\theta, \bar t +\theta)$
 for $\bar z=(\bar x,\bar t)$. Also for a ball $B\subset \R^n$, $\A_{B}(t,u,\xi) :=\fint_{B} \A(x,t,u,\xi)\, dx$ is the average of $\A$ with respect to the $x$ variable. 
\begin{theorem}\label{thm:main2}  Let $p> 2n/(n+2)$ and  $\A : Q_6\times \overline{\K}\times  \R^n \longrightarrow \R^n$
be  a Carath\'eodory map such that $\xi\mapsto \A(z,u,\xi)$ is differentiable on $\R^n\setminus \{0\}$ for a.e. 
$z\in Q_6$ and all $u\in \overline{\K}$.
 Assume that  $\A(\cdot,\cdot,0)=0$ and  $\A$ satisfies the following  conditions for a.e. $z\in Q_6$ and  all $(u,\xi)\in 
\overline{\K}\times (\R^n\setminus \{0\})$:
\begin{equation}\label{strengthen-stru}
\left \{
\begin{array}{lcll}
\langle \partial_\xi \A(z,u,\xi) \eta, \eta\rangle 
&\geq& \Lambda^{-1} (\mu^2+|\xi|^2)^{\frac{p-2}{2}} |\eta|^2 \, \, \qquad \forall \eta\in \R^n, \\
\big|\partial_\xi \A(z,u,\xi)\big|  &\leq& \Lambda (\mu^2 + |\xi|^2)^{\frac{p-2}{2}},\\
  |\A(z,u_1,\xi)-\A(z,u_2,\xi)| &\leq &   \Lambda |u_1 - u_2| \, (\mu^2+ |\xi|^2)^{\frac{p-1}{2}}\quad \forall u_1, \, u_2\in \overline{\K}
\end{array}\right.
\end{equation}
for some constants $\Lambda>0$ and $\mu\in [0,1]$. 
Then for any 
$M_0\in (0,\infty)$ and $q>1$, there exists $\delta >0$ depending only 
on $p$, $q$, $n$, $M_0$, $\Lambda$, and $\K$ 
such that: if
\begin{equation}\label{BMO-osc}
 \sup_{\bar z=(\bar x, \bar t)\in Q_3,\,  Q_{\bar z}(r, \theta)\subset Q_6} \fint_{ Q_{\bar z}(r,\theta)} \Big[
\sup_{u\in \overline\K}\sup_{ \xi\in\R^n}\frac{|\A(x,t, u,\xi) -  \A_{B_r(\bar x)}(t,u,\xi)|}{1+ |\xi|^{p-1}}
\Big] \, dx dt\leq \delta^p  
\end{equation}
and $u$ is a weak solution to \eqref{mainPDE} with $\|u\|_{L^\infty(Q_4)}\leq M_0$, we have
\begin{align*}
\int_{Q_3} |D u|^{pq} \, dz
\leq C&\left\{1+ \Big[\int_{Q_6} (|Du|^p  + |\F|^p)\, dz +\Big(\int_{Q_6} |f|^{\frac{p(n+2)}{p(n+2)-n}}\, dz\Big)^{\hat p} \Big]^{1+ d(q-1)} 
\right.\\
&\left. \qquad\qquad\qquad\qquad +  \int_{Q_6}|\F|^{pq} \, dz +\Big(\int_{Q_6} |f|^{\frac{pq(n+2)}{p(n+2)-n}}\, dz\Big)^{\hat p} \right\}.
\end{align*}
Here  $C>0$ is a constant depending only on  $p$, $q$, $n$, $M_0$, $\Lambda$, $\K$, and 
 $d\geq 1$ and $\hat p>1$ are the numbers given by
\begin{equation}\label{de:d}
d := \left \{
\begin{array}{lcll}
 \frac{p}{2} &\text{if}\quad p\geq 2, \\
\frac{2p}{(n+2)p - 2n} &\qquad  \text{ if}\quad \frac{2n}{n+2}<p<2
\end{array}\right.
\qquad
\mbox{and}\qquad  \hat p := \frac{p(n+2)-n}{p(n+1)-n}.
\end{equation}
\end{theorem}

\bigskip
This result generalizes the gradient estimates obtained in \cite[Theorem~1.6]{HNP1} for the case $\K=[0,1]$ and
 $\A(z, u,  Du) = (1+\alpha u)\ba(z)\, D  u$ with $\alpha>0$ being a constant.
In Theorem~\ref{thm:main2}, $\A$ is only assumed to be measurable in the time variable.
As \eqref{BMO-osc} is automatically satisfied when $x\mapsto \A(x,t,u,\xi)$  is of vanishing mean oscillation, condition
\eqref{BMO-osc}  allows $\A$ to be discontinuous in $x$. It is also well known that
some smallness condition in $x$ for $\A$ is necessary even in the linear case. On the other hand,
 we show under merely structural 
conditions \eqref{structural-reference-1}--\eqref{structural-reference-2} for $\A$  that spatial gradients of weak solutions to  \eqref{mainPDE} enjoy  the higher integrability in the sense of Elcrat and Meyers \cite{ME}
(see  Theorem~\ref{thm:higherint}). 


We prove Theorem~\ref{thm:main2}   by using a perturbation argument 
together with  the intrinsic geometry method \cite{D2}. But as \eqref{mainPDE} is not invariant with respect to this intrinsic geometry, we are led to deal with a rescaling equation which depends  on two parameters (see equation \eqref{eq-2}). Then by
employing  the localization method   \cite{KiL} and the  
maximal function free approach\cite{AM}, we demonstrate  that 
 $L^q$ estimates for $D u$ can be derived as long as  gradients of solutions to the two-parameter equation can be approximated by  bounded gradients in a fashion that is independent of the parameters (Theorem~\ref{thm:conditionalLq}).
 The remaining and key part is to prove that there exists such Lipschitz approximation property. We achieve this through a delicate compactness argument involving two scaling parameters, and by using an important
 gradient bound in \cite{KuM1} which generalizes the fundamental $L^\infty$ gradient estimate by DiBenedetto and Friedman  \cite{DF1}.
 The compactness procedure consists of two main steps and is employed to compare gradients of solutions of our two-parameter equation to those of the corresponding frozen equation. In the first step, we reduce the problem to the homogeneous case 
(Lemma~\ref{lm:compare-solution-1}). We then handle the  homogeneous equation in the second step 
(Lemma~\ref{lm:compare-solution-2}) by making use of the higher integrability stated in Theorem~\ref{thm:higherint}.
It is crucial for our purpose  that the constants $\delta$ in these two lemmas can be chosen to be independent of the  parameters. This two-parameter technique was introduced in our recent paper \cite{HNP1} where parabolic equations whose principal parts are linear in the gradient variable were considered. The technique was further extended in
\cite{NP} to deal with quasilinear elliptic equations of $p$-Laplacian type. However, the arguments 
in \cite{HNP1,NP} do not work for the equations under consideration since \eqref{mainPDE} is degenerate/singular and it scales 
differently in time and space. We overcome this by a different approach in Section~\ref{sec:AppGra} which  exploits the nature of 
evolutionary equations and Gr\"onwall type inequality. Another nice feature of this approach is that it 
works well with highly nonlinear equations and allows us to completely 
avoid using the Minty-Browder's technique as  in \cite{NP}.
As a consequence we do not need to impose any condition 
on $\A$ in the time variable except its  measurability.

 The organization of the paper is as follows. We present some basic properties of our equations  in  Section~\ref{sec:pre} where we also state Theorem~\ref{thm:higherint} about  the higher integrability of gradients. In
 Section~\ref{sec:general}, we prove Proposition~\ref{L^lestimatewithoutPDE} about  $L^l$ estimates for a general function under 
 a decay assumption for its distribution function.  In Section~\ref{sec:conditionalLq}, we formulate a
 Lipschitz approximation property and show in Theorem~\ref{thm:conditionalLq} that this property implies  $L^q$ estimates for $D u$ 
 for any $q>p$. We then verify the  Lipschitz approximation property  in Section~\ref{sec:AppGra}
 by developing a compactness argument involving 
 scaling parameters. The proof of  our main result (Theorem~\ref{thm:main2})
 is given at the end of Subsection~\ref{sub:compactness} by combining the mentioned ingredients and employing an important gradient bound from \cite{KuM1}.  Section~\ref{sec:proof-higher} is devoted to  the the proof of Theorem~\ref{thm:higherint}  about the self improving property of gradients.
  \section{Preliminary results and higher integrability}\label{sec:pre}
In this section we derive  some  elementary estimates which will be used later. We begin with a direct consequence of
 structural condition \eqref{structural-reference-1} when  $1<p<2$.
\begin{lemma}\label{simple-est} Let $1<p<2$ and 
assume that  $\A$ satisfies  \eqref{structural-reference-1}. 
 Then there exists $C_p>0$  depending only on $p$ such that:
 for any $\tau>0$,  we have for a.e. $z\in Q_6$ that
 \begin{align*}
  \tau^{\frac{2}{p} -1} (1 -\tau )|\xi -\eta|^p
\leq \tau^{\frac{2}{p}}  (1+|\xi |^p) + C_p \Lambda   \langle \A(z, u, \xi) -\A(z,u,\eta),
  \xi - \eta \rangle\quad \forall \xi,\eta\in \R^n.
   \end{align*}
\end{lemma}
\begin{proof} Let $\xi,\eta\in\R^n$. Since $|\xi|+|\eta|\leq 2|\xi| +|\xi-\eta|$ and $1<p<2$, we have from \eqref{structural-reference-1} that
\begin{equation}\label{equiv-structural-1}
\langle  \A(z,u,\xi) -\A(z,u,\eta), \xi-\eta\rangle \geq \Lambda^{-1} 2^{p-2} \big(1+ |\xi|+ |\xi-\eta|)^{p-2} |\xi-\eta|^2.
\end{equation}
Using Young's inequality,  we obtain
 \begin{align*}
 &|\xi -\eta|^p  =\big( 1+ |\xi| + |\xi-\eta| \big)^{\frac{p(2-p)}{2}} \big( 1+ |\xi| + |\xi-\eta|\big)^{\frac{p(p-2)}{2}}  
 |\xi -\eta|^p  \\
 &\leq \tau \, 3^{-p} \big( 1+  |\xi| + |\xi-\eta|\big)^p 
 + C_p \tau^{\frac{p-2}{p}} \big(1+  |\xi| + |\xi-\eta|\big)^{p-2}  |\xi-\eta|^2.  
   \end{align*}
   This together with \eqref{equiv-structural-1} yields the conclusion of  the lemma.
\end{proof}
Let us next  introduce some notations that will be used throughout the paper.
For $\bar z=(\bar x,\bar t)$ and $r, \, \theta>0$, we define
 $ Q_{\bar z}(r,\theta) := B_r(\bar x)\times (\bar t-\theta, \bar t +\theta)$
 and 
\begin{equation}\label
{intrinsiccylinder}
Q_{r}^\lambda(\bar z) := \left \{
\begin{array}{lcll}
 B_r(\bar x) \times  (\bar t-\lambda^{2-p} r^2, \bar t+ \lambda^{2-p} r^2) &\text{if}\quad p\geq 2, \\
B_{\lambda^{\frac{p-2}{2}} r}(\bar x) \times  (\bar t- r^2, \bar t+ r^2)&\qquad \text{if}\quad 1<p<2.
\end{array}\right.
\end{equation}
Also,  $Q_{r}(\bar z) := B_r(\bar x) \times  (\bar t- r^2, \bar t+  r^2)$. 
  For simplicity, we will always  write $B_r$ for $B_r(0)$ and $Q_r$ for $Q_{r}(0)$. The 
  cylinders $Q_{r}(\bar z)$ and $Q_{r}^\lambda(\bar z)$ shall be called standard parabolic cylinder 
  and {\it intrinsic cylinder}, respectively. In addition, $\partial_p Q_r$ denotes the standard parabolic
  boundary of $Q_r$.
 \subsection{Weak solutions}
\begin{definition}[weak solutions] Let $\A$ satisfy \eqref{structural-reference-1}--\eqref{structural-reference-2}. 
Assume that $\alpha>0$, $Q_{\bar z}( r, \theta)\subset Q_6$, $f\in L^1(Q_{\bar z}( r, \theta))$, and  
$\F\in L^p(Q_{\bar z}( r, \theta))$. 
 A map
 \[
  u\in C(\bar t -\theta , \bar t +\theta; L^2(B_r))\cap L^p(\bar t - \theta, \bar t +\theta; W^{1,p}(B_r)) 
 \]
 is called a weak solution to   equation
 \begin{equation}\label{generalPDE}
u_t  =  \div \A(z, \alpha u,D u)   +\div(|\mathbf{F}|^{p-2} \mathbf{F}) +f \quad \text{in}\quad Q_{\bar z}( r, \theta)
\end{equation} 
 if $\, u(z) \in \frac1\alpha \, \overline{\K}\,$ for a.e. $z\in Q_{\bar z}( r, \theta)$ and 
 \begin{equation}\label{weak-1}
  \int_{Q_{\bar z}( r, \theta)} u \varphi_t\,dz = \int_{Q_{\bar z}( r, \theta)} \langle \A(z, \alpha u,D u) +  |\mathbf{F}|^{p-2}
  \mathbf{F}, D \varphi\rangle\,dz
  -\int_{Q_{\bar z}( r, \theta)} f \varphi\,dz \quad \forall \varphi\in C_0^\infty(Q_{\bar z}( r, \theta)).
 \end{equation}
\end{definition}
Weak solutions to  \eqref{generalPDE}
 possess a modest degree of regularity in the time variable.
In order to work with test functions involving the solution itself, it is therefore convenient to 
adopt the formulation  in terms of the so-called Steklov averages.
For $g\in L^1(Q_{\bar z}( r, \theta))$ and $0<h<\bar t +\theta$, we define the Steklov average $[g]_h$ of $g$ by 
\begin{equation}\label{Steklov-ave}
[g]_h(x,t) := \left \{
\begin{array}{lcll}
 \frac{1}{h} \int_t^{t+h} g(x,s)\, ds &\qquad\text{ for}\quad t\in (\bar t -\theta, \bar t +\theta-h], \\
0 &\text{for}\quad t> \bar t + \theta -h.
\end{array}\right.
\end{equation}
Then if $f\in L^l(Q_{\bar z}( r, \theta))$ for some $l>\frac{pn}{p(n+1)-n}$, we have that
\eqref{weak-1} is equivalent to 
\begin{equation}\label{weak-2}
  \int_{B_r(\bar x)\times \{t\}} \partial_t [u]_h \phi\,dx = - \int_{B_r(\bar x)\times \{t\}} \langle 
  [\A(z, \alpha u,D u)]_h +[|\mathbf{F}|^{p-2} \mathbf{F}]_h, D \phi\rangle\,dx
 + \int_{B_r(\bar x)\times \{t\}} [f]_h \phi\,dx 
 \end{equation}
 for all $0<t<\bar t +\theta -h$ and all $\phi \in W_0^{1,p}(B_r(\bar x))$.
 
\subsection{Scaling properties} The following result displays scaling properties of our equation.
\begin{lemma}\label{lm:scaling}
Assume that $p>1$. Suppose  $u$ is a weak solution to  equation
\begin{equation}\label{eq-1}
u_t  =  \div \A(x,t, u,D u)  +\div(|\F|^{p-2} \F) + f\quad \text{in}\quad Q_{\theta r}^\lambda(\bar z).
\end{equation}
For $p\geq 2$, we define
\begin{equation*}
 \left \{
\begin{array}{lcll}
 \tilde u(x,t) := \frac{u(\bar x + \theta x, \, \bar t + \lambda^{2-p} \theta^2 t)}{\theta \lambda},
 \quad \tilde \A(x,t, u, \xi) := \frac{\A(\bar x + \theta x, \, \bar t + \lambda^{2-p} \theta^2 t, \, u, \, \lambda \xi)}{\lambda^{p-1}},
 \\
 \tilde \F(x,t) := \frac{\F(\bar x + \theta x, \, \bar t + \lambda^{2-p} \theta^2 t)}{ \lambda},
\quad \tilde f(x,t) := \theta\, \frac{f(\bar x + \theta x, \, \bar t + \lambda^{2-p} \theta^2 t)}{\lambda^{p-1}}.
 \end{array}\right.
\end{equation*}
For $p< 2$, we define
\begin{equation*}
 \left \{
\begin{array}{lcll}
 \tilde u(x,t) := \frac{u(\bar x +  \theta \lambda^{\frac{p-2}{2}}  x, \, \bar t + \theta^2 t)}{\theta \lambda^{\frac{p}{2}}}, 
 \quad
 \tilde \A(x,t, u, \xi) := \frac{\A(\bar x +  \theta \lambda^{\frac{p-2}{2}}  x, \, \bar t + \theta^2 t, \, u, \, \lambda \xi)}{\lambda^{p-1}},\\
  \tilde \F(x,t) := \frac{\F(\bar x +  \theta \lambda^{\frac{p-2}{2}}  x, \, \bar t + \theta^2 t)}{ \lambda},\quad \tilde f(x,t) := \theta\, \frac{f(\bar x +  \theta \lambda^{\frac{p-2}{2}}  x, \, \bar t + \theta^2 t)}{ \lambda^{\frac{p}{2}}}.
 \end{array}\right.
\end{equation*}
Then $\tilde u$ is a weak solution to  equation
\begin{equation}\label{eq-2}
\tilde u_t  = \div \tilde \A(z,\theta \hat\lambda \tilde u,D \tilde u)
+\div(|\tilde \F|^{p-2} \tilde \F)  + 
\tilde f \quad 
\text{in}\quad Q_{r},
\end{equation}
where $\hat \lambda=\lambda$ if $p\geq 2$ and $\hat \lambda=\lambda^{\frac{p}{2}}$ if $1<p<2$.
\end{lemma}
\begin{proof}
 This can be easily checked by writing out the weak formulations  and making an appropriate change of variables. Let us consider only the 
 case $p<2$. For $\varphi\in
  C_0^\infty(Q_{\theta r}^\lambda(\bar z))$, let
 \[
  \tilde \varphi (x,t) := \varphi(\bar x +  \theta \lambda^{\frac{p-2}{2}}  x, \, \bar t + \theta^2 t).
 \]
Then since $\tilde \varphi_t = \theta^2 \varphi_s$, $D \tilde \varphi =  \theta \lambda^{\frac{p-2}{2}} D_y \varphi$, and
$D \tilde u =\frac{1}{\lambda}D_y u$, we see that
the integral
\begin{align*}
 \int_{Q_r} (\tilde u \tilde\varphi_t)(x,t) \, dx dt = \int_{Q_r}\langle \tilde \A(x,t,\theta \lambda^{\frac{p}{2}} \tilde u,D \tilde u) 
 +|\tilde \F|^{p-2} \tilde \F,
D \tilde \varphi(x,t)\rangle \, dx dt
 - \int_{Q_r} (\tilde f  \tilde \varphi)(x,t)\, dx dt
\end{align*}
is equivalent to
 \begin{align*}
  \int_{Q_{\theta r}^\lambda(\bar z)}  (u \varphi_s)(y,s) \, dy ds 
 =  \int_{Q_{\theta r}^\lambda(\bar z)}  \langle \A(y,s, u,D_y u) +| \F|^{p-2} \F, D_y \varphi(y,s)\rangle
 \, dy ds
 - \int_{Q_{\theta r}^\lambda(\bar z)}  (f \varphi)(y,s)\, dy ds.
 \end{align*}
Therefore, we infer that $u$ is a weak solution of \eqref{eq-1} if and only if  $\tilde u$ is a weak solution of \eqref{eq-2}.
\end{proof}

\subsection{Energy estimates}
We see from Lemma~\ref{lm:scaling} that our equations are not invariant with respect to the intrinsic geometry. 
This forces us to deal with equation \eqref{eq-2} involving two parameters. For simplicity, we set $\alpha
=\theta \hat \lambda$ and consider  equation
\begin{equation}\label{eq:uQ4}
 u_t =\div \A(z, \alpha u,D  u) + \div(|\F|^{p-2}\F)+ f \quad\mbox{in}
\quad Q_4.
\end{equation}
We now derive some elementary
energy estimates for \eqref{eq:uQ4}. Hereafter, $d\geq 1$ and $\hat p>1$ denote the constants 
given by \eqref{de:d}. We also use throughout the paper that
\begin{equation*}\label{pp}
\bar p = \frac{p(n+2)}{n} \,\, \quad \mbox{ and }\quad \,\, \bar p' =  \frac{p(n+2)}{p(n+2)-n}.
\end{equation*}
Notice that $\bar p'$ is the conjugate of $\bar p$, i.e. $\frac{1}{\bar p} +\frac{1}{\bar p'}=1$. 
\begin{lemma}\label{energy-estimate} 
Assume that $\alpha>0$,  $\A$ satisfies \eqref{structural-reference-1}--\eqref{structural-reference-2} in $Q_4$, and $\A(\cdot,\cdot,0)=0$. 
Suppose  $u$ is a  weak solution of  \eqref{eq:uQ4}.
There exists a constant $C>0$ depending only on $p$, $n$, and $\Lambda$ such that
\begin{enumerate}
 \item[(i)] If $p\geq 2$, then
\begin{align*}
\sup_{s\in (-9,9)} &\int_{B_3} u^2(x,s)\, dx
+ \int_{Q_3}  |D u|^p \, dz
\leq
 C \Bigg[ \int_{Q_4} \big(|u|  + u^2  + |u|^p 
  + |\F|^p \big)\,dz +\Big( \int_{Q_4}|f|^{\bar p'}dz \Big)^{\hat p}\Bigg].
\end{align*}
\item[(ii)] If $1<p<2$, then
\begin{align*}
&\sup_{s\in (-9,9)} \int_{B_3} u^2(x,s)\, dx
+\sigma \int_{Q_3}  |D u|^p \, dz\leq C \int_{Q_4}  ( |u|  + u^2 )\,dz\nonumber\\
&+
 C\left[\sigma^{\frac{2}{2-p}} \int_{Q_4} (1+ | D u|^p) \, dz
 + \sigma^{1-p}  \int_{Q_4}  |u|^p \, dz
 +\sigma^{\frac{-1}{p-1}}\int_{Q_4}  |\F|^p  \, dz +\sigma^{\frac{-(n+p)}{p(n+1)-n}}\Big( \int_{Q_4}|f|^{\bar p'}dz \Big)^{\hat p}\right]
\end{align*}
for every $\sigma>0$ small.
\end{enumerate}
\end{lemma}
\begin{proof}
Let $\varphi \in C^\infty_0(Q_4)$ be the standard nonnegative cut-off function which is $1$ on  $Q_3$.
Using  $\phi(x) =\varphi(x,t)^p [u]_h(x,t)$ in the  weak formulation  \eqref{weak-2} and then integrating in 
$t$  we get after letting $h\to 0^+$ that 
\begin{align*}
\frac12\int_{B_4} (u^2 \varphi^p)(x,s) \, dx 
&- \frac{p}{2}  \int_{-16}^s\int_{B_4}  u^2 \varphi^{p-1} \varphi_t\, dz\\
&= -
\int_{-16}^s\int_{B_4}  \langle \A(z, \alpha u,D  u) + |\F|^{p-2} \F , D (\varphi^p u) \rangle \, dz
+\int_{-16}^s\int_{B_4} f \varphi^p u \, dz
\end{align*}
for each $s\in (-16,16)$. 
Let $K_s := B_4\times (-16, s)$.
Then it follows from the above identity, the 
assumption $\A(z, \alpha u,0)=0$, and   \eqref{structural-reference-2} that
\begin{align}\label{ini-energy}
&\frac12\int_{B_4} (u^2 \varphi^p)(x,s) \, dx 
- \frac{p}{2}  \int_{K_s}  u^2 \varphi^{p-1} \varphi_t\, dz +\int_{K_s}  \langle 
\A(z, \alpha u,D  u)-\A(z, \alpha u,0) , D u \rangle \varphi^p \, dz \nonumber\\
&=-p\int_{K_s} \langle \A(z, \alpha u,D  u) +|\F|^{p-2}\F, D \varphi  \rangle u \varphi^{p-1} \, dz
-  \int_{K_s}  |\F|^{p-2} \langle   \F , Du\rangle \varphi^p\, dz+\int_{K_s} f \varphi^p u \, dz\nonumber\\
&\leq p \int_{K_s} \Big[\Lambda \big(1+| D  u|^{p-1}\big) +|\F|^{p-1}\Big]\varphi^{p-1} |u| |D \varphi|        \, dz
 +\int_{K_s}  |\F|^{p-1}  | D u| \varphi^p\, dz+\int_{K_s} |f| |u \varphi^p|  \, dz .
\end{align}
 We next use H\"older's inequality,  the parabolic embedding  (see \cite[Proposition~3.1, page~7]{D2}), and Young's inequality  to get
\begin{align}\label{f-sobolev}
\int_{K_s}   |f|  |u \varphi^p|\, dz
&\leq \|u \varphi^p\|_{L^{\bar p}(K_s)}
\|f\|_{L^{\bar p'}(K_s)}\nonumber\\
&\leq 
C(n,p)\Big(\int_{K_s}  |D (u \varphi^p)|^p \, dz\Big)^{\frac{1}{\bar p}}
\Big(\sup_{t\in (-16,s)}\int_{B_4} (u \varphi^p)^2(x,t)\, dx \Big)^{\frac{p}{n \bar p}}
\,\, \|f\|_{L^{\bar p'}(K_s)}\nonumber\\
&\leq \e \int_{K_s}  |D (u \varphi^p)|^p \, dz +\e \sup_{t\in (-16,s)}\int_{B_4}(u \varphi^p)^2(x,t)\, dx
+ \frac{C(n,p)}{\e^{\frac{p+n}{p(n+1)-n}}} \|f\|_{L^{\bar p'}(K_s)}^{\frac{p(n+2)}{p(n+1)-n}}\quad \forall \e>0.
\end{align}
Applying Young's inequality to the first two integrals on the right hand side of \eqref{ini-energy} and using 
\eqref{f-sobolev} with a suitable choice of $\e$, we then obtain
\begin{align}\label{pre-energy}
\frac12\int_{B_4} (u^2 \varphi^p)(x,s) \, dx 
 &+\int_{K_s}  \langle 
\A(z, \alpha u,D  u)-\A(z, \alpha u,0) , D u \rangle \varphi^p \, dz -\sigma \int_{K_s} | D  u|^p\varphi^p\, dz\nonumber\\
 &\leq \frac14 \sup_{t\in (-16,16)}\int_{B_4} (u^2 \varphi^p)(x,t)\, dx + C \int_{Q_4}  (|u|+ u^2)  \, dz
 + C \sigma^{1-p} \int_{Q_4}  |u|^p \, dz \nonumber\\ 
 &\quad  +C\, \sigma^{\frac{-1}{p-1}}\int_{Q_4}  |\F|^p\, dz
 +C \sigma^{\frac{-(n+p)}{p(n+1)-n}}\Big( \int_{Q_4}|f|^{\bar p'}dz \Big)^{\hat p} \qquad \qquad \forall \sigma\in (0,1).
\end{align}
 If $p\geq 2$, then by using structural condition 
\eqref{structural-reference-1} and choosing $\sigma$ sufficiently small we arrive at
\begin{align}\label{uphi-energy}
&\frac12\int_{B_4} (u^2 \varphi^p)(x,s) \, dx
+\frac{1}{2\Lambda}\int_{K_s}  | D u|^p \varphi^p \, dz\nonumber\\
&\leq   \frac14 \sup_{t\in (-16,16)}\int_{B_4} (u^2 \varphi^p)(x,t)\, dx 
+C\Bigg[\int_{Q_4}  \big( |u|+u^2 +|u|^p +|\F|^p\big) \, dz
 +\Big( \int_{Q_4}|f|^{\bar p'}dz \Big)^{\hat p}\Bigg]
\end{align}
for each $s\in (-16,16)$. This immediately  gives
\begin{align*}
\sup_{s\in (-16,16)}\int_{B_4} (u^2 \varphi^p)(x,s) \, dx
\leq 
C\Bigg[\int_{Q_4}  \big( |u|+u^2 +|u|^p +|\F|^p\big) \, dz
 +\Big( \int_{Q_4}|f|^{\bar p'}dz \Big)^{\hat p}\Bigg].
\end{align*}
Combing this with \eqref{uphi-energy} and using the fact
$\varphi=1$ on $Q_3$, we obtain (i). In the case $1<p<2$, it follows  from  Lemma~\ref{simple-est} that
\[
 c \tau^{\frac{2-p}{p}} |Du|^p - 2c \tau^{\frac2p}(1+|Du|^p)\leq 
\langle \A(z, \alpha u,D  u)-\A(z, \alpha u,0) , D u  \rangle \qquad \forall \tau\in (0,1/2).
\]
This together with \eqref{pre-energy} gives
\begin{align*}
&\frac12\int_{B_4} (u^2 \varphi^p)(x,s) \, dx
+(c \tau^{\frac{2-p}{p}}-\sigma) \int_{K_s}  | D u|^p \varphi^p\, dz\\
&\leq \frac14 \sup_{t\in (-16,16)}\int_{B_4} (u^2 \varphi^p)(x,t)\, dx + C \int_{Q_4}  (|u|+ u^2)  \, dz
+2c\tau^{\frac2p} \int_{Q_4} (1+ | D u|^p)\, dz\nonumber\\
& \quad + C \sigma^{1-p} \int_{Q_4}  |u|^p \, dz 
+C\, \sigma^{\frac{-1}{p-1}}\int_{Q_4}  |\F|^p\, dz
 +C \sigma^{\frac{-(n+p)}{p(n+1)-n}}\Big( \int_{Q_4}|f|^{\bar p'}dz \Big)^{\hat p}
 \quad \mbox{for  }s\in (-16,16).
\end{align*}
 By taking $\tau$ such that $c \tau^{\frac{2-p}{p}}=2\sigma$, we infer as in the case $p\geq 2$  that (ii) holds.
\end{proof}
The next  lemma allows us to estimate the difference between gradients of  solutions originating from different equations. 
\begin{lemma}\label{gradient-est-II} 
Assume that $\alpha>0$, and $\A$, $\hat \A$ satisfy \eqref{structural-reference-1}--\eqref{structural-reference-2}
in $Q_4$.
Suppose $u$ is a  weak solution of \eqref{eq:uQ4}
and  $v$ is a weak solution of  
\begin{equation*}
\left \{
\begin{array}{lcll}
 v_t &=&\div  \hat\A(z, \alpha v,D  v)  \quad &\text{in}\quad Q_4, \\
v & =& u\quad &\text{on}\quad \partial_p Q_4.
\end{array}\right.
\end{equation*}
Then there exists $C>0$ depending only on $n$, $p$, and $\Lambda$ such that
\begin{align*}
\sup_{s\in (-16,16)}\int_{B_4} | u -v|^2\, dx
&+ \int_{Q_4}  |D  u - D  v|^p \, dz\leq 
C\Bigg[ \int_{Q_4} \big(1+  | Du|^p +|\F|^p\big) \, dz +\Big( \int_{Q_4}|f|^{\bar p'}dz \Big)^{\hat p}\Bigg].
\end{align*}
\end{lemma}
\begin{proof}
Let 
$h =  u -  v$. Then  $h$ is a weak solution of 
\[
h_t = \div \big[ \hat\A(z, \alpha v,D u) -
 \hat\A(z,  \alpha v,D v)  \big]+\div \big[\A(z,\alpha u,D  u) -
 \hat\A(z,\alpha v,D  u) \big] + \div(|\F|^{p-2}\F) + f
 \mbox{ in } Q_4,
\]
with $h=0$ on $\partial_p Q_4$. Multiplying the above equation by $h$ and integrating by parts we obtain for each $s\in (-16,16)$ that
\begin{align*}
&\int_{B_4} \frac{h^2(x,s)}{2} \, dx
+ \int_{-16}^s\int_{B_4} \langle \hat\A(z,\alpha  v,D  u) -
 \hat\A(z,\alpha  v,D  v) , D h\rangle \, dz\\
&= -
\int_{-16}^s\int_{B_4}  \langle \A(z,\alpha  u,D  u) -
\hat\A(z,\alpha v,D  u) + |\F|^{p-2}\F, D h\rangle \, dz
 +  \int_{-16}^s\int_{B_4} f h\, dz.
\end{align*} 
We deduce from this, structural conditions \eqref{structural-reference-1}--\eqref{structural-reference-2},
and  Lemma~\ref{simple-est} with $\tau=1/2$ that
\begin{align*}
\frac{1}{2}\int_{B_4} h^2(x,s) \, dx
&+ c(\Lambda, p)\int_{K_s}  | Dh|^p \, dz- c(\Lambda, p) \int_{K_s}  (1+| D u|^p) \, dz\\
&\leq 
 \int_{K_s}   \Big[ 2 \Lambda(1+| D  u|^{p-1}) +|\F|^{p-1} \Big]| D h|\, dz+\int_{K_s}   |f|  |  h|\, dz,
\end{align*}
where $K_s := B_4\times (-16, s)$. Hence, applying Young's inequality and collecting like-terms give
\begin{equation}\label{energy-embed}
\int_{B_4} h^2(x,s)\, dx
+ \int_{K_s}  |D h|^p \, dz
\leq 
C\, \int_{K_s} \big(1+  | Du|^p +|\F|^p\big) \, dz +C\int_{K_s}   |f|  |  h|\, dz.
\end{equation}
 But it follows from the same argument as in \eqref{f-sobolev} that 
\begin{align*}
\int_{K_s}   |f|  |  h|\, dz
\leq \e \int_{K_s}  |D h|^p \, dz +\e \sup_{t\in (-16,s)}\int_{B_4} h^2(x,t)\, dx
+ \frac{C(n,p)}{\e^{\frac{p+n}{p(n+1)-n}}} \|f\|_{L^{\bar p'}(K_s)}^{\frac{p(n+2)}{p(n+1)-n}}\qquad \forall \e>0.
\end{align*}
Hence by taking $\e=1/(2C)$ and substituting the resulting expression into \eqref{energy-embed}, we obtain
\begin{align*}
\int_{B_4} h^2(x,s)\, dx
&+\frac12 \int_{K_s}  |D h|^p \, dz\nonumber\\
&\leq \frac12 \sup_{t\in (-16,16)}\int_{B_4} h^2(x,t)\, dx +
C\Bigg[ \int_{Q_4} \big(1+  | Du|^p +|\F|^p\big) \, dz 
+  \|f\|_{L^{\bar p'}(Q_4)}^{\frac{p(n+2)}{p(n+1)-n}}\Bigg]
\end{align*}
for each  $s\in (-16,16)$. 
This implies the conclusion of the lemma.
\end{proof}

\subsection{Higher integrability of  gradients}\label{sub:higher}
We next state the higher integrability in the sense of Elcrat and Meyers \cite{ME}
for spatial gradients  of weak solutions to equation \eqref{eq:uQ4}:
\begin{theorem}
\label{thm:higherint}
Assume that $\alpha>0$ and  $\A$ satisfies \eqref{structural-reference-1}--\eqref{structural-reference-2}.
Let $p>2n/(n+2)$ and suppose that 
$u$  is  a weak solution of \eqref{eq:uQ4}.
Then there exist $\e_0>0$ small and $C>0$  depending only on $\Lambda$, $n$, and $p$  such that
\begin{align}\label{higher-in}
\int_{Q_3} |D u|^{p+\e_0}\, dz\leq C\Bigg\{ 1+ &\Big[\int_{Q_4} (|D u|^p +|\F|^p)\, dz  + \Big( \int_{Q_4}|f|^{\bar p'}dz \Big)^{\hat p}\Big]^{1+ \frac{\e_0 d}{p}}\nonumber\\
&\qquad\qquad + \int_{Q_3} |\F|^{p+\e_0}\, dz +\Big( \int_{Q_3}|f|^{\bar p'(1+\frac{\e_0}{p})}dz \Big)^{\hat p} \Bigg\}.
\end{align}
\end{theorem}
In this theorem we do not impose  any smallness condition on $\A$   and this self improving property  of gradients will be used to perform the perturbation analysis in Section~\ref{sec:AppGra}.
The proof of Theorem~\ref{thm:higherint}  will be given in Section~\ref{sec:proof-higher}.

\section{General arguments without PDEs}\label{sec:general} In this section, we establish some general results which 
are independent of the PDEs under consideration.  
 \subsection{A covering argument} 
\begin{lemma}\label{lm:covering} Let $p>2n/(n+2)$ and $0<R_1< R_2$. 
Assume that $g\in L^p(Q_{R_2})$ and $h_1,\,h_2\in L^1(Q_{R_2})$ are nonnegative functions. Define
 \[
  \bar\lambda^{\frac{p}{d}} := \fint_{Q_{R_2}}\big(g^p + h_1 +1\big)\, dz
  +\frac{1}{|Q_{R_2}|} \Big(\int_{Q_{R_2}} h_2\, dz\Big)^{\hat p}\quad \mbox{and}\quad \bar B^{\frac{p}{d}}:=\Big(\frac{10 R_2}{R_2 - R_1}\Big)^{n+2}.
 \]
Then for any $\lambda\geq \bar B \bar\lambda$, there exists a sequence of disjoint {\it intrinsic cylinders}  $\{Q_{r_i}^\lambda(z_i)\}$ with 
$z_i\in Q_{R_1}$ and $r_i\in (0, \frac{R_2 - R_1}{10}]$
that satisfies the following properties:
\begin{enumerate}
 \item[1)] $\fint_{Q_{r_i}^\lambda(z_i)}\big(g^p + h_1\big)\, dz +\frac{1}{|Q_{r_i}^\lambda(z_i)|} \Big(\int_{Q_{r_i}^\lambda(z_i)} h_2\, dz\Big)^{\hat p}=\lambda^p$ for each $i$.
\item[2)] $\fint_{Q_{r}^\lambda(z_i)}\big(g^p + h_1\big)\, dz 
+\frac{1}{|Q_r^\lambda(z_i)|} \Big(\int_{Q_r^\lambda(z_i)}h_2\, dz\Big)^{\hat p} <\lambda^p$ for every 
$r\in (r_i, R_2 - R_1]$.
\item[3)] 
$ E:=\Big\{z\in Q_{R_1}:\, z\mbox{ is a Lebesgue point of $g$ and } g(z)>\lambda \Big\} \subset \bigcup_{i=1}^\infty Q_{5 r_i}^\lambda(z_i)$.
\end{enumerate}
\end{lemma}
\begin{proof}
The proof of this lemma can be deduced from the arguments in \cite{KiL,AM,Bo}. For the sake of completeness, we reproduce  it here.
Observe  that due to $\lambda\geq 1$ we have $Q_{r}^\lambda(z)\subset Q_{R_2}$ for every $z\in Q_{R_1}$ and every $r\leq R_2 - R_1$.

Let   $z\in E$ be arbitrary. On one hand, 
 we have 
 \begin{align*}
  \fint_{Q_{r}^\lambda(z)}\big(g^p + h_1\big)\, dz
  +\frac{1}{|Q_r^\lambda(z)|} \Big(\int_{Q_r^\lambda(z)} h_2\, dz\Big)^{\hat p}
  &\leq \frac{|Q_{R_2}|}{|Q_{r}^\lambda(z)|}
  \Big[\fint_{Q_{R_2}}\big(g^p + h_1\big) \, dz +\frac{1}{|Q_{R_2}|} \Big(\int_{Q_{R_2}} h_2\, dz\Big)^{\hat p}\Big]\\
  &\leq  \lambda^p \Big(\frac{R_2}{r}\Big)^{n+2} \Big(\frac{\bar \lambda}{\lambda}\Big)^{\frac{p}{d}}<  \lambda^p \Big(\frac{10 R_2}{R_2 - R_1}\Big)^{n+2} \frac{1}{\bar B^{\frac{p}{d}}}=\lambda^p
  \end{align*}
for every $\frac{R_2 - R_1}{10} < r\leq R_2 - R_1$. On the other hand,  the Lebesgue differentiation theorem  gives
\begin{align*}
  \liminf_{r\to 0^+}\Bigg[\fint_{Q_{r}^\lambda(z)}\big(g^p + h_1\big)\, dz
  +\frac{1}{|Q_r^\lambda(z)|} \Big(\int_{Q_r^\lambda(z)} h_2\, dz\Big)^{\hat p} \Bigg]
  \geq  g(z)^p
  >\lambda^p.
  \end{align*}
  Thus by continuity, for each $z\in E$  there must exist $r_z\in (0, \frac{R_2 - R_1}{10}]$ such that
  \[
   \fint_{Q_{r_z}^\lambda(z)}\big(g^p + h_1\big)\, dz+\frac{1}{|Q_{r_z}^\lambda(z)|} 
   \Big(\int_{Q_{r_z}^\lambda(z)} h_2\, dz\Big)^{\hat p} 
   =\lambda^p\]
   and
   \[
     \fint_{Q_{r}^\lambda(z)}\big(g^p + h_1\big)\, dz +\frac{1}{|Q_r^\lambda(z)|} \Big(\int_{Q_r^\lambda(z)} h_2\, dz\Big)^{\hat p} <\lambda^p\quad \forall r\in (r_z, R_2 - R_1].
  \]
  Hence by  the Vitali covering lemma, one can extract a countable subcollection  of disjoint {\it intrinsic cylinders}  $\{Q_{r_i}^\lambda(z_i)\}$ satisfying
\[
 \bigcup_{z\in  E } Q_{r_z}^\lambda(z)\subset \bigcup_{i=1}^\infty Q_{5 r_i}^\lambda(z_i).
 \]
  The lemma  then follows since $ E\subset \bigcup_{z\in  E } Q_{r_z}^\lambda(z)$.
  \end{proof}
 
 \subsection{$L^l$ estimates under a decay assumption}
 For a nonnegative function $h$ on $Q_R$ and a number $\lambda>0$, we define
 \[
  E_h(Q_R, \lambda) := \Big\{z\in Q_R:\, h(z) >\lambda \Big\}.
 \]
In the following result, we derive $L^l$ estimates for a general function under 
 a decay assumption for its distribution function.
\begin{proposition}\label{L^lestimatewithoutPDE} 
Let $R>0$ and $\lambda_0>0$.
Let $f,\, g,\,  \hat g$ be  nonnegative Borel measurable functions on $Q_{2R}$, and let $\mu,\,\nu,\, \hat \nu$ be  nonnegative Borel measures 
on $Q_{2R}$.  Assume that there exist  constants $N\geq 1$ and $\alpha>0$ such that for any $0<R_1<R_2\leq 2R$ we have
\begin{equation}\label{densitydecay}
 \mu\big(E_f(Q_{R_1}, 2N \lambda)\big)  \leq \alpha \Big[\mu\big(E_f(Q_{R_2}, \frac{\lambda}{3})\big) 
 +\nu\big(E_{g}(Q_{R_2},\frac{\lambda}{3})\big) +\hat \nu\big(E_{\hat g}(Q_{R_2},\frac{\lambda}{3})\big)^{\hat p}\Big]
 \end{equation}
 for all $ \lambda\geq 
  \lambda_0 \Big(\frac{10 R_2}{R_2 - R_1}\Big)^{\frac{d(n+2)}{p}}$.
   Then for any $l>0$, we obtain
\begin{align*}
 \frac{1}{M^l}\int_{Q_R} f^l \, d\mu
&\leq  (c_0 \lambda_0 ) ^l \mu (Q_R)+   \Bigg[  (c_0 \lambda_0)^l \mu(Q_{2 R})   
+   \int_{Q_{2R}}g^l  d\nu  +\frac{M^l-1}{(M^{\frac{l}{\hat p}} -1)^{\hat p}} \Big(\int_{Q_{2R}}\hat g^{\frac{l}{\hat p}} 
d\hat \nu \Big)^{\hat p}\Bigg] \sum_{j=1}^\infty (\alpha M^l )^j,
 \end{align*}
 where $M:= \max{\{6N, 2^{\frac{d(n+2)}{p}}\}}$ and $c_0 :=3^{-1} 2^{\frac{6d(n+2)}{p}}$.
\end{proposition}
\begin{proof}
For any $m\in \{0, 1,2,\dots\}$, let $\rho_m := R\Big( 3- \sum_{k=0}^{m}\frac{1}{2^k}\Big)$. Then $\rho_0=2R$, and $\rho_m\downarrow R$ as $m\uparrow \infty$. By
using \eqref{densitydecay} for $R_1\rightsquigarrow \rho_{m+1}$ and  $R_2\rightsquigarrow \rho_m$, we have for any $m\geq 0$ that
\begin{equation*}
 \mu\big(E_f(Q_{\rho_{m+1}}, 2N \lambda)\big)  \leq \alpha \Big[\mu\big(E_f(Q_{\rho_m}, \frac{\lambda}{3})\big) +\nu\big(E_g(Q_{\rho_m},\frac{\lambda}{3})\big) +\hat\nu\big(E_{\hat g}(Q_{\rho_m},\frac{\lambda}{3})\big)^{\hat p} \Big]
\end{equation*}
for every $\lambda \geq 2^{\frac{d(n+2)(m+6)}{p}} \lambda_0$. It follows that
\begin{equation*}
 \mu\Big(E_f(Q_{\rho_{m+1}},  c_0 \lambda_0 6N\lambda')\Big)
 \leq \alpha
\left[ \mu\Big(E_f(Q_{\rho_m}, c_0 \lambda_0 \lambda')\Big)  +
\nu\Big(E_g(Q_{\rho_m}, c_0 \lambda_0 \lambda')\Big) 
+
\hat\nu\Big(E_{\hat g}(Q_{\rho_m}, c_0 \lambda_0 \lambda')\Big)^{\hat p} 
 \right]
\end{equation*}
for $\lambda' \geq  2^{\frac{d(n+2)m}{p}}$.
Since $M^m \geq  2^{\frac{d(n+2)m}{p}}$ and  $M\geq 6N$, by taking $\lambda'=M^m$ we thus obtain
\begin{align}\label{one-step}
 &\mu\Big(E_f(Q_{\rho_{m+1}},c_0 \lambda_0  M^{m+1} )\Big)\nonumber\\
 &\leq \alpha  
\left[ \mu\Big(E_f(Q_{\rho_m}, c_0 \lambda_0 M^m)\Big)  +
\nu\Big(E_g(Q_{\rho_m},  c_0 \lambda_0 M^m)\Big) + 
\hat\nu\Big(E_{\hat g}(Q_{\rho_m},  c_0 \lambda_0 M^m)\Big)^{\hat p}\right]
\quad \forall m=0,1,\dots
\end{align}
By iterating and using \eqref{one-step}, we arrive at:
\begin{align*}
 \mu\Big(E_f(Q_{\rho_k},  c_0 \lambda_0 M^k)\Big)
 \leq \alpha^k \mu\big(E_f(Q_{2 R},  c_0 \lambda_0 )\big) + \sum_{i=0}^{k-1} 
 \alpha^{k - i} 
\Big[\nu\Big(E_g(Q_{\rho_i},  c_0 \lambda_0  M^i)\Big)
+  \hat\nu\Big(E_{\hat g}(Q_{\rho_i},  c_0 \lambda_0  M^i)\Big)^{\hat p}\Big].
\end{align*}
In particular,
\begin{equation}\label{iteration}
 \mu\Big(E_f(Q_R, c_0 \lambda_0 M^k )\Big)
 \leq \alpha^k \mu\big(E_f(Q_{2 R}, c_0 \lambda_0 )\big)+ \sum_{i=0}^{k-1} \alpha^{k - i}  I_i
\quad \forall k\geq 1,
\end{equation} 
where $I_i :=\nu\Big(E_g(Q_{2 R},  c_0 \lambda_0 M^i)\Big)
+  \hat\nu\Big(E_{\hat g}(Q_{2 R},  c_0 \lambda_0 M^i)\Big)^{\hat p}$.

Since  
\begin{align*}
\int_{Q_R} f^l \, d\mu  
&=l \int_0^\infty t^{l-1} \mu\Big(\{Q_R: f>t\}\Big)\, dt\\
&=l \int_0^{c_0 \lambda_0 M} t^{l -1} \mu\Big(\{Q_R: f>t\}\Big)\, dt 
+l  \sum_{k=1}^\infty\int_{c_0 \lambda_0 M^{k}}^{c_0 \lambda_0 M^{k+1}} t^{l -1} \mu\Big(\{Q_R: f>t\}\Big)\, dt\\
&\leq (c_0 \lambda_0 M) ^l \mu (Q_R)  + (M^{l} -1) (c_0 \lambda_0)^l\sum_{k=1}^\infty  M^{l k}\mu\Big(E_f(Q_R, c_0 \lambda_0 M^k )\Big),
\end{align*}
we obtain from \eqref{iteration} that
\begin{align*}
\frac{1}{(M^l -1)(c_0 \lambda_0)^l}\Big[\int_{Q_R} f^l \, d\mu &- (c_0 \lambda_0 M) ^l \mu (Q_R) \Big]
\leq  \mu\big(E_f(Q_{2 R}, c_0 \lambda_0 )\big)   \sum_{k=1}^\infty (\alpha M^l )^k + \sum_{k=1}^\infty\sum_{i=0}^{k-1}    M^{lk} \alpha^{k - i} I_i\\
&= \mu\big(E_f(Q_{2 R}, c_0 \lambda_0 )\big)   \sum_{k=1}^\infty (\alpha M^l)^k
+ \sum_{i=0}^\infty M^{l i} I_i \Big(\sum_{k=i+1}^\infty  (\alpha M^l)^{k - i} \Big)\\
&=  \left[ \mu\big(E_f(Q_{2 R}, c_0 \lambda_0 )\big)   
+ \sum_{i=0}^\infty M^{l i} I_i \right]\sum_{j=1}^\infty (\alpha M^l)^j.
\end{align*}
Moreover, as $\hat p>1$ we have
\begin{align*}
\sum_{i=0}^\infty M^{l i} I_i
&\leq 
\sum_{i=0}^\infty M^{l i} \nu\big(E_g(Q_{2 R},  c_0 \lambda_0 M^i)\big) +\Big[\sum_{i=0}^\infty M^{\frac{l}{\hat p} i} 
\hat\nu\big(E_{\hat g}(Q_{2 R},  c_0 \lambda_0 M^i)\big)\Big]^{\hat p}.
\end{align*}
These together with  Remark~\ref{rm:lower-est-L^p-norm} below imply that
\begin{align*}
&\int_{Q_R} f^l \, d\mu
\leq  (c_0 \lambda_0 M) ^l \mu (Q_R)\\
& +   \left[ (M^l -1) (c_0 \lambda_0)^l \mu\big(E_f(Q_{2 R}, c_0 \lambda_0 )\big)   
+  M^l \int_{Q_{2R}}g^l \, d\nu + M^l\frac{M^l-1}{(M^{\frac{l}{\hat p}} -1)^{\hat p}} \Big(\int_{Q_{2R}}\hat g^{\frac{l}{\hat p}} \, d\hat \nu \Big)^{\hat p} \right]\sum_{j=1}^\infty (\alpha M^l )^j.
\end{align*}
This gives the conclusion of the proposition.
\end{proof}

\begin{remark}\label{rm:lower-est-L^p-norm}
Assume that $V\subset \R^n\times\R$, $\nu$ is a nonnegative Borel measure on $V$, and  $g\in L^l_\nu(V)$ for some $l>0$. 
Then for any $\alpha_0>0$ and $M>1$, we have 
\[
(M^l -1) \Big(\frac{\alpha_0}{M}\Big)^l \sum_{i=0}^\infty  M^{l i} \nu\big(\{ V: |g|> \alpha_0 M^i\} \big)
\leq \int_{V}|g|^l \, d\nu.
\]  
Indeed,
\begin{align*}
\int_{V}|g|^l \, d\nu
&=l \int_0^\infty t^{l-1} \nu\big(\{V: |g| >t\}\big)\, dt
\geq 
l \sum_{i=0}^\infty\int_{\alpha_0 M^{i-1}}^{\alpha_0 M^i} t^{l-1} \nu\big(\{V: |g|>t\}\big)\, dt\\
&\geq \sum_{i=0}^\infty \big[(\alpha_0 M^i)^l - (\alpha_0 M^{i-1})^l\big]\, \nu\big(\{ V: |g|> \alpha_0 M^i\} \big).
\end{align*}
\end{remark}

\section{Conditional $L^q$ estimates for spatial gradients}\label{sec:conditionalLq}
In this section, we formulate  a condition guaranteeing $L^q$ estimates for spatial gradients of
weak solutions to 
equation \eqref{mainPDE}.   The verification of this condition for a large class of vector fields will be carried out in
Section~\ref{sec:AppGra}. For a vector field $\mathbf{G}(x,t,u,\xi)$ and a ball $B\subset \R^n$, we define
\[
\mathbf{G}_B(t,u,\xi) :=\fint_{B} \mathbf{G}(x,t,u,\xi)\, dx.
\]

\begin{definition}[local Lipschitz approximation property]\label{locallip-property} Assume  $\A$ satisfies \eqref{structural-reference-1}--\eqref{structural-reference-2}  and $p>1$.  Given $\bar z = (\bar x, \bar t)$, we define
\begin{equation}
\label{eq:tildeA}
\tilde\A(x,t,u,\xi) := \left \{
\begin{array}{lcll}
 \frac{\A(\bar x + \theta x, \, \bar t + \lambda^{2-p} \theta^2 t, \, u, \, \lambda \xi)}{\lambda^{p-1}} &\text{ if}\quad p\geq 2, \\
\frac{\A(\bar x +  \theta \lambda^{\frac{p-2}{2}}  x, \, \bar t + \theta^2 t, \, u, \, \lambda \xi)}{\lambda^{p-1}} &\qquad  \text{if}\quad 1<p<2.
\end{array}\right.
\end{equation}
 We say that $\A$ satisfies the local Lipschitz approximation property with constant 
 $M_0\in (0, \infty]$ if for any $\e>0$, 
 there exists $\delta =\delta(\e,p,n, M_0,\Lambda,\K)>0$ such that: if $\lambda\geq 1$, $0<\theta<2$, $Q_{4\theta}^\lambda(\bar z)
 \subset Q_6$,
\[
  \fint_{Q_4} \Big[
\sup_{u\in \overline \K}\sup_{ \xi\in\R^n}\frac{|\tilde \A(x,t, u,\xi) - \tilde \A_{B_4}(t, u,\xi)|}{1+ |\xi|^{p-1}}
\Big] \, dxdt + \fint_{Q_4}  |\tilde \F|^p \, dz
+\Big(\fint_{Q_4}  |\tilde f|^{\bar p'} \, dz\Big)^{\hat p}\leq \delta^p, 
\]
 and $\tilde u$ is a weak solution  to 
\begin{equation*}
 \tilde u_t  =  \div \tilde \A(z,\theta \hat\lambda \tilde u,D  \tilde u) 
+\div(|\tilde \F|^{p-2} \tilde \F)+\tilde f \quad 
\text{in}\quad Q_4
\end{equation*}
satisfying $\|\tilde u\|_{L^\infty(Q_4)}\leq M_0/\theta\hat\lambda$ and $\fint_{Q_4} |D  \tilde u|^p\, dz\leq 1$, then we have
\begin{equation}\label{eq:appxgoodgradient}
 \fint_{Q_2} |D\tilde  u - \tilde \Psi|^p\, dz\leq  \e^p
 \end{equation}
 for some  function $ \tilde \Psi\in L^\infty(Q_2;\R^n)$ with
  $  \|\tilde \Psi\|_{L^\infty(Q_2)}\leq N$, where $N\geq 1$ is a constant depending only on $p$, $n$, $M_0$, $\Lambda$, and $\K$. Here 
  $\hat \lambda=\lambda$ if $p\geq 2$ and $\hat \lambda=\lambda^{\frac{p}{2}}$ if $1<p<2$.
\end{definition}

  The following main result of the section  shows that 
 the Lipschitz approximation property for the vector field $\A$ 
 implies $L^q$ estimates for gradients of weak solutions to the corresponding
 equation for any $q>p$.
\begin{theorem}\label{thm:conditionalLq}  Assume $p>2n/(n+2)$ and  $\A$ satisfies  the local Lipschitz
approximation property with constant $M_0\in (0, \infty]$. Then for any $q>1$, there exists $\delta >0$ depending 
only on $p$, $q$, $n$, $M_0$, $\Lambda$, and $\K$ 
such that: if
\begin{equation}\label{BMO-1}
 \sup_{\bar z=(\bar x,\bar t)\in Q_3,\,  Q_{\bar z}( r, \theta)\subset Q_6} \fint_{ Q_{\bar z}( r, \theta)} \Big[
\sup_{u\in \overline\K}\sup_{ \xi\in\R^n}\frac{|\A(x,t, u,\xi) -  \A_{B_r(\bar x)}(t,u,\xi)|}{1+ |\xi|^{p-1}}
\Big] \, dx dt\leq \delta^p  
\end{equation}
and $u$ is a weak solution to \eqref{mainPDE} with $\|u\|_{L^\infty(Q_4)}\leq M_0$, we have
\begin{align*}
\int_{Q_3} |D u|^{pq} \, dz
\leq C \Bigg\{ 1+ \Big[\int_{Q_6} (|Du|^p  &+ |\F|^p)\, dz +\Big(\int_{Q_6}|f|^{\bar p'} \, dz \Big)^{\hat p}\Big]^{1+ d(q-1)} 
\\
& +  \int_{Q_6}|\F|^{pq} \, dz
+\Big(\int_{Q_6}|f|^{\bar p' q} \, dz \Big)^{\hat p}\Bigg\}.
\end{align*}
Here $C$ is a positive constant depending only on  $p$, $q$, $n$, $M_0$, $\Lambda$, and $\K$. 
\end{theorem}
\begin{proof}
Let  $\e>0$ be determined later, and let  $\delta =\delta(\e,p,n, M_0,\Lambda,\K)>0$ be the corresponding constant given by Definition~\ref{locallip-property}.
Let $0<R_1<R_2\leq 6$,
\[
  \lambda_0^{\frac{p}{d}} := \fint_{Q_6}\big(|D u|^p + \frac{1}{\delta^p} |\F|^p+1\big)\, dz
  +\frac{1}{\delta^p}\frac{1}{|Q_6|} \Big(\int_{Q_6} |f|^{\bar p'}\, dz\Big)^{\hat p},\quad \mbox{and}\quad \bar B^{\frac{p}{d}}:=\Big(\frac{10 R_2}{R_2 - R_1}\Big)^{n+2}.
 \]
Also,  let us denote
 $ E(Q_{R_1}, \lambda) := \Big\{z\in Q_{R_1}:\, z\mbox{ is a Lebesgue point of $|Du|$ and } |Du(z)| >\lambda \Big\}$.
Then for any  $\lambda\geq \bar B \lambda_0$, we can apply Lemma~\ref{lm:covering} for  
$g\rightsquigarrow |Du|$, $h_1\rightsquigarrow \delta^{-p} |\F|^p$, and $h_2\rightsquigarrow \delta^{-\frac{p}{\hat p}} |f|^{\bar p'}$
to  obtain:
 there exists a sequence of disjoint {\it intrinsic cylinders}  $\{Q_{r_i}^\lambda(z_i)\}$ with 
$z_i=(x_i, t_i)\in Q_{R_1}$ and $r_i\in (0, \frac{R_2 - R_1}{10}]$
that satisfies the following properties
\begin{enumerate}
 \item[1)] 
$ E(Q_{R_1}, \lambda) \subset \bigcup_{i=1}^\infty Q_{5 r_i}^\lambda(z_i)$.
\item[2)] $\fint_{Q_{r_i}^\lambda(z_i)}\Big(|Du|^p +  
\frac{1}{\delta^p} |\F|^p\Big)\, dz
+\frac{1}{\delta^p}\frac{1}{|Q_{r_i}^\lambda(z_i)|} \Big(\int_{Q_{r_i}^\lambda(z_i)} |f|^{\bar p'}\, dz\Big)^{\hat p}= \lambda^p$ for each $i$.
\item[3)] $\fint_{Q_{r}^\lambda(z_i)}\Big(|Du|^p+  \frac{1}{\delta^p} |\F|^p\Big)\, dz
+\frac{1}{\delta^p}\frac{1}{|Q_{r}^\lambda(z_i)|} \Big(\int_{Q_{r}^\lambda(z_i)} |f|^{\bar p'}\, dz\Big)^{\hat p}<\lambda^p$ for every 
$r\in (r_i, R_2 - R_1]$.
\end{enumerate}

Now let us fix $i$, and note that $Q_{10 r_i}^\lambda(z_i)\subset Q_{R_2}\subset Q_6$ as $z_i\in Q_{R_1}$ and $10 r_i\leq R_2 - R_1$. 
Let $\tilde u$,  $\tilde \F$, and $\tilde \A$ be defined as in 
Lemma~\ref{lm:scaling} with  $\bar x=x_i,\, \bar t=t_i$,  and $\theta = 5 r_i/2$. Then by Lemma~\ref{lm:scaling}, we see that
  $\tilde u$ is a weak solution to the equation
\begin{equation*}
\tilde u_t  = \div \tilde \A(z,\theta \hat\lambda \tilde u,D  \tilde u)  
+\div(|\tilde \F|^{p-2} \tilde \F) +\tilde f \quad 
\text{in}\quad Q_4.
\end{equation*}
Observe that $\|\tilde u\|_{L^\infty(Q_4)}\leq M_0/\theta\hat\lambda$. Using $D \tilde u =\frac{1}{\lambda} D u ()$ and the definitions of $\tilde \F$ and $ \tilde f$, we deduce from property 3)   that
\begin{align*} 
 &\fint_{Q_4} |D \tilde u(x,t)|^p\, dx dt =\frac{1}{\lambda^p} \fint_{Q_{10 r_i}^\lambda(z_i)} |D u(z)|^p\, dz <1,\\
  &\fint_{Q_4} | \tilde \F(x,t)|^p\, dx dt =\frac{1}{\lambda^p} \fint_{Q_{10 r_i}^\lambda(z_i)} |\F(z)|^p\, dz <\frac{1}{\lambda^p} \delta^p \lambda^p = \delta^p,
\end{align*}
and 
\begin{align*}
\fint_{Q_4} | \tilde f(x,t)|^{\bar p'}\, dx dt &=\frac{1}{\lambda^p}\big[2^{-1}|Q_{\frac{5 r_i}{2}}^\lambda(z_i)|\lambda^p\big]^{\frac{\bar p'}{n+2}} \fint_{Q_{10 r_i}^\lambda(z_i)} |f(z)|^{\bar p'}\, dz \\
&\leq  \frac{\lambda^{p(\frac{\bar p'}{n+2}-1)}}{|Q_{10 r_i}^\lambda(z_i)|^{1-\frac{\bar p'}{n+2}}}\int_{Q_{10 r_i}^\lambda(z_i)} |f(z)|^{\bar p'}\, dz
=\frac{\lambda^{-\frac{p}{\hat p}}}{|Q_{10 r_i}^\lambda(z_i)|^{\frac{1}{\hat p}}}\int_{Q_{10 r_i}^\lambda(z_i)} |f(z)|^{\bar p'}\, dz
< \delta^{\frac{p}{\hat p}},
\end{align*}
Moreover, as $\lambda\geq 1$ it is clear that
\begin{align*}
  &\fint_{Q_4} \Big[
\sup_{u\in \overline \K}\sup_{ \xi\in \R^n}\frac{|\tilde \A(x,t,u,\xi) - \tilde \A_{B_4}(t,u,\xi)|}{1+ |\xi|^{p-1}}
\Big] \, dx dt\\
&= \fint_{Q_{10 r_i}^\lambda(z_i)} \Big[
\sup_{u\in \overline \K}\sup_{ \eta\in\R^n}\frac{| \A(y,s, u,\eta) -  \A_{B}(s, u,\eta)|}{\lambda^{p-1} + |\eta|^{p-1}}
\Big] \, dy ds\leq \delta^p
\end{align*}
thanks to condition \eqref{BMO-1}, where the ball $B$ is the projection of $Q_{10 r_i}^\lambda(z_i)\subset\R^n\times \R$ onto $\R^n$.
Since $\A$ satisfies the local Lipschitz approximation property, we conclude that
there exists a function $\tilde \Psi_i\in  L^\infty(Q_2;\R^n)$ such that
\[  \|\tilde \Psi_i\|_{L^\infty(Q_2)}\leq N\quad\mbox{and}\quad 
 \fint_{Q_2} |D\tilde u - \tilde \Psi_i|^p\, dz\leq  \e^p
 \]
with $N\geq 1$ depending only on $p$, $n$, $M_0$, $\Lambda$, and $\K$.
Let us rescale back by defining
\begin{equation*}
\Psi_i(y,s) := \left \{
\begin{array}{lcll}
\lambda\, \tilde \Psi_i (\frac{y-x_i}{\theta}, \frac{s - t_i}{\lambda^{2-p} \theta^2}) &\text{if}\quad p\geq 2, \\
 \lambda\, \tilde \Psi_i \Big(\frac{y-x_i}{\theta \lambda^{\frac{p-2}{2}}}, \frac{s - t_i}{ \theta^2}\Big) &\qquad  \, \text{ if}\quad \frac{2n}{n+2}<p<2.
\end{array}\right.
\end{equation*}
Then we obtain
\begin{equation}
\label{eq:goodappro}
 \|\Psi_i\|_{L^\infty(Q_{ 5 r_i}^\lambda(z_i))}\leq N \lambda
\quad \mbox{and}\quad \fint_{Q_{ 5 r_i}^\lambda(z_i)} |Du - \Psi_i|^p\, dz\leq \lambda^p \e^p. 
 \end{equation}
As a consequence, we get
\begin{align}\label{density-1}
 \big| Q_{5 r_i}^\lambda(z_i) \cap E(Q_{R_1}, 2N \lambda)\big|
 &\leq\big|\{z\in Q_{5  r_i}^\lambda(z_i):\,  |(Du -\Psi_i) (z)|>N \lambda \}\big|\nonumber\\
 &\leq \frac{1}{N^p \lambda^p}   \int_{Q_{5 r_i}^\lambda(z_i)} |Du - \Psi_i|^p\, dz  \leq \Big(\frac{\e}{N}\Big)^p  |Q_{5 r_i}^\lambda(z_i)|.
\end{align}
We next estimate $|Q_{5 r_i}^\lambda(z_i)|=5^{n+2} |Q_{r_i}^\lambda(z_i)|$ on the above right hand side. Setting
 $\hat f(z) :=\delta^{-1} |f(z)|^{\frac{\bar p' \hat p}{p}}$. Then from property 2) and since $\hat p>1$, we have
\begin{align*}
&|Q_{r_i}^\lambda(z_i)| =\frac{1}{\lambda^p} \int_{Q_{r_i}^\lambda(z_i)}\big(|Du|^p +  \frac{1}{\delta^p} |\F|^p  \big)\, dz
+\frac{1}{\lambda^p} \Big(\int_{Q_{r_i}^\lambda(z_i)}
\hat f^{\frac{p}{\hat p}}\, dz\Big)^{\hat p}\\
&\leq \frac{1}{\lambda^p}\left[ \int_{\{Q_{r_i}^\lambda(z_i): |Du|> \frac{\lambda}{3}\}} |Du|^p \, dz  +
\frac{1}{\delta^p} \int_{\{Q_{r_i}^\lambda(z_i): |\F|> \frac{\delta \lambda }{3}\}} |\F|^p\, dz +
 \Big(\int_{\{Q_{r_i}^\lambda(z_i): \hat f> \frac{\lambda }{3}\}}
\hat f^{\frac{p}{\hat p}}\, dz\Big)^{\hat p}\right]
+\frac{1}{3^{p-1}} |Q_{r_i}^\lambda(z_i)|.
\end{align*}
It follows that
\begin{equation}\label{eq:Q-i}
|Q_{r_i}^\lambda(z_i)| 
\leq \frac{C_p}{\lambda^p}\left[ \int_{\{Q_{r_i}^\lambda(z_i): |Du|> \frac{\lambda}{3}\}} |Du|^p \, dz  +
\frac{1}{\delta^p} \int_{\{Q_{r_i}^\lambda(z_i): |\F|> \frac{\delta \lambda }{3}\}} |\F|^p\, dz +
 \Big(\int_{\{Q_{r_i}^\lambda(z_i): \hat f > \frac{\lambda }{3}\}}
\hat f^{\frac{p}{\hat p}}\, dz\Big)^{\hat p}\right].
\end{equation}
Using \eqref{eq:goodappro}--\eqref{eq:Q-i}, we deduce that
\begin{align*}
&\int_{ Q_{5 r_i}^\lambda(z_i) \cap E(Q_{R_1}, 2N \lambda)} |Du|^p dz
\leq 2^{p-1} \Big( \int_{ Q_{5 r_i}^\lambda(z_i) \cap  E(Q_{R_1}, 2N \lambda)} |Du - \Psi_i|^p dz + \int_{ Q_{5 r_i}^\lambda(z_i) \cap  E(Q_{R_1}, 2N \lambda)} 
| \Psi_i|^p dz \Big)\\
&\leq 2^{p-1} \Big(\lambda^p \e^p |Q_{5 r_i}^\lambda(z_i)|  + (N\lambda)^p \, | Q_{5 r_i}^\lambda(z_i) \cap E(Q_{R_1}, 2N\lambda)|\Big)
\leq 2^p \lambda^p \e^p   |Q_{5 r_i}^\lambda(z_i)|\\
&\leq C(n,p) \e^p \left[ \int_{\{Q_{r_i}^\lambda(z_i): |Du|> \frac{\lambda}{3}\}} |Du|^p \, dz  +
\frac{1}{\delta^p} \int_{\{Q_{r_i}^\lambda(z_i): |\F|> \frac{\delta \lambda }{3}\}} |\F|^p\, dz +
 \Big(\int_{\{Q_{r_i}^\lambda(z_i): \hat f> \frac{\lambda }{3}\}}
\hat f^{\frac{p}{\hat p}}\, dz\Big)^{\hat p}\right].
\end{align*}
Since $ E(Q_{R_1}, 2N \lambda)  \subset E(Q_{R_1}, \lambda) \subset \bigcup_{i=1}^\infty Q_{5 r_i}^\lambda(z_i)$ and $\{Q_{r_i}^\lambda(z_i)\}$ is 
disjoint, by taking the sum over $i$ we obtain  
\begin{equation*}
 \int_{E(Q_{R_1}, 2N \lambda)} |Du |^p\, dz
 \leq C(n,p) \e^p 
\left[ \int_{\{Q_{R_2}: |Du|> \frac{\lambda}{3}\}} |Du|^p  dz  +
\frac{1}{\delta^p} \int_{\{Q_{R_2}: |\F|> \frac{\delta \lambda }{3}\}} |\F|^p dz +
 \Big(\int_{\{Q_{R_2}: \hat f> \frac{\lambda }{3}\}}
\hat f^{\frac{p}{\hat p}} dz\Big)^{\hat p}\right]
\end{equation*}
for all $ \lambda \geq \bar B \lambda_0$. Therefore, we can  apply Proposition~\ref{L^lestimatewithoutPDE} with $l:=p(q-1)$,
$ \mu(dz) := |Du(z)|^p \, dz$,  $\nu(dz) := |\frac{\F(z)}{\delta}|^p\, dz$, and $\hat\nu(dz) :=\hat f^{\frac{p}{\hat p}} \, dz$ to  conclude that
 \begin{align*}
 \frac{1}{M^l}\int_{Q_3} |Du|^l \, d\mu
&\leq  (c_0 \lambda_0 ) ^l \mu (Q_3)\\
 &+   \Bigg[  (c_0 \lambda_0)^l \mu(Q_6)   
+   \int_{Q_6}\big|\frac{\F}{\delta}\big|^l \, d\nu  +\frac{M^l-1}{(M^{\frac{l}{\hat p}} -1)^{\hat p}} \Big(\int_{Q_6}
\hat f^{\frac{l}{\hat p}} \, d\hat \nu \Big)^{\hat p}\Bigg]\sum_{j=1}^\infty \big[C(n,p)\e^p M^l \big]^j,
 \end{align*}
 where $M:= \max{\{6N, 2^{\frac{d(n+2)}{p}}\}}$ and $c_0 :=3^{-1} 2^{\frac{6d(n+2)}{p}}$.
  Let us now choose $\e>0$ such that
\[
C(n,p) \,  \e^p M^l = \frac{1}{2}.
\]
Then with the corresponding $\delta$, we obtain
\begin{align*}
\frac{1}{M^l}\int_{Q_3} |D u|^{pq} \, dz
&\leq 2 (c_0 \lambda_0 )^l \int_{Q_6} |Du|^p\, dz   
+  \int_{Q_6}\big|\frac{\F}{\delta}\big|^{pq} \, dz
+\frac{(M^l-1)}{(M^{\frac{l}{\hat p}} -1)^{\hat p}\delta^{pq}} \Big(\int_{Q_6}|f|^{\bar p' q} \, dz \Big)^{\hat p}.
\end{align*}
This together with the definition of $\lambda_0$ yields the conclusion of the theorem.
\end{proof}

\section{Approximating gradients of solutions}\label{sec:AppGra}
The purpose of this section is to verify the local Lipschitz approximation property for a large class of vector fields and 
then employ Theorem~\ref{thm:conditionalLq} to obtain $L^q$ estimates for spatial gradients of weak solutions to the corresponding equations. To achieve this and for the first time, the structural condition \eqref{structural-reference-3} shall be used.
 Throughout this section,
 let $\omega: [0,\infty)\to [0,\infty)$ be the bounded function defined  by
\begin{equation*}\label{modified-Lipschitz}
\omega(r)= \left \{
\begin{array}{lcll}
r \Lambda  \qquad  \text{if}\quad 0\leq r \leq 2, \\
2\Lambda  \qquad\text{if}\quad r>2.
\end{array}\right.
\end{equation*}
Notice that if $\A$ satisfies \eqref{structural-reference-2} and \eqref{structural-reference-3}, then we obtain from the definition
of $\omega$ that
\begin{equation}\label{structural-reference-3-reformulation}
 |\A(z,u_1,\xi)-\A(z,u_2,\xi)|  \leq  \omega(|u_1 -u_2|)\,\, \big(1+ |\xi|^{p-1}\big) \qquad\quad\qquad \,\,\,\,\forall u_1, u_2\in \overline\K. 
\end{equation}
For this reason, \eqref{structural-reference-3} and \eqref{structural-reference-3-reformulation} will be used  
interchangeably.
 Our aim is to approximate  $D u$ 
by a good vector function in $L^p$ norm, and the following lemma is the starting point for that purpose. Let us  define
\begin{equation}\label{oscillating-fn}
\d_{\A,\hat\A}(z) := \sup_{u\in \overline\K}\sup_{ \xi\in \R^n}\frac{|\A(z, u,\xi) - \hat\A(z,u,\xi)|}{1+ |\xi|^{p-1}}.
\end{equation}

\begin{lemma}\label{lm:compare-gradient-part1}
Assume that $\alpha>0$ and $\A$, $\hat \A$ satisfy \eqref{structural-reference-1}--\eqref{structural-reference-2}.
 Assume in addition that $\hat \A$ satisfies \eqref{structural-reference-3}.
Suppose  $u$
is a weak solution  of \eqref{eq:uQ4}
with  
\begin{equation}\label{ass:uF}
\fint_{Q_4}{|D u|^p\, dz}\leq C(\Lambda,p,n)
\quad \mbox{ and }\quad \fint_{Q_4}{|\F|^p\, dz}+\fint_{Q_4}  |f|^{\bar p'} \, dz\leq 1,
\end{equation}
and 
 $h$  is a weak solution of 
\begin{equation*}
\left \{
\begin{array}{lcll}
h_t &=&\div  \hat\A(z,\alpha h,D h)  \quad &\text{in}\quad Q_3, \\
h & =& u\quad &\text{on}\quad \partial_p Q_3.
\end{array}\right.
\end{equation*}
Let $m:=u-h$. Then there exist  positive constants $C,\, \e_0$ depending only on $p$, $n$, and $\Lambda$ such that 
\begin{enumerate}
 \item[(i)] If $p\geq 2$, then 
 \begin{align*}
 \sup_{t\in (-4,4)}\int_{B_2} 
m(x,t)^2\, dx
&+\int_{Q_2}{ |D m|^p dz }\leq C\Big(\|m\|_{L^p(Q_{\frac52})}^p + \|m\|_{L^p(Q_{\frac52})} +\|m\|_{L^2(Q_{\frac52})}^2\Big)\\
&+  C \Bigg\{ \Big(\int_{Q_\frac52}{\big[
 \omega
(\alpha |m|) + \d_{\A,\hat\A} \big]  \, dz}\Big)^{\frac{\e_0}{p+\e_0}} 
+\|\F\|_{L^p(Q_{\frac52})}^p +\Big(\int_{Q_\frac52}  |f|^{\bar p'} \, dz\Big)^{\hat p}\Bigg\}.
 \end{align*}
\item[(ii)] If $1<p<2$, then 
\begin{align*}
\sup_{t\in (-4,4)}&\int_{B_2} 
m(x,t)^2\, dx +\sigma \int_{Q_2} |Dm|^p  dz
\leq  C \sigma^{\frac{2}{2-p}} +C\Big(\|m\|_{L^p(Q_{\frac52})}^p + \|m\|_{L^p(Q_{\frac52})} +\|m\|_{L^2(Q_{\frac52})}^2\Big)\\
&+ C \sigma^{\frac{-1}{p-1}} \Bigg\{ \Big(\int_{Q_\frac52}{\big[
 \omega
(\alpha | m|) + \d_{\A,\hat\A}  \big]  \, dz}\Big)^{\frac{\e_0}{p+\e_0}} +\|\F\|_{L^p(Q_{\frac52})}^p
\Bigg\}
+C \sigma^{\frac{-(p+n)}{p(n+1)-n}}\Big(\int_{Q_\frac52}  |f|^{\bar p'} \, dz\Big)^{\hat p} 
\end{align*}
for every $\sigma>0$ small.
\end{enumerate}
\end{lemma}
\begin{proof}
It follows from Lemma~\ref{gradient-est-II} and  assumption   \eqref{ass:uF}  that
\begin{align}\label{eq:Dv}
 \int_{Q_3}  |D  h|^p \, dz\leq 
C \int_{Q_3} \big(1+  | Du|^p +|\F|^p\big) \, dz
+C \Big(\int_{Q_3}  |f|^{\bar p'} \, dz\Big)^{\hat p}\leq C(\Lambda,p,n).
\end{align}
Therefore, we can employ Theorem~\ref{thm:higherint} for $\F=0$ and $f=0$ to conclude that there exist $\e_0>0$ small and $C>0$
depending only on $\Lambda$, $n$, and $p$ such that
\begin{equation}\label{higher-intergrability-v}
\int_{Q_\frac52} |D h|^{p+\e_0}\, dz\leq C.
\end{equation}
Let $\varphi\in C_0^\infty(Q_\frac52)$ be the standard nonnegative cut-off function which is $1$ on $Q_2$.
Let $K_s := B_\frac52 \times (-25/4, s)$. 
Then by using $\varphi^p m$ as a test function in the equations for $u$ and $h$, we have
for each $s\in (-25/4, 25/4)$ that
\begin{align*}
\int_{K_s}  m_t \varphi^p m\, dz
&=\int_{K_s}{ \langle  \hat\A(z,\alpha h,D h) - \A(z, \alpha u, D u) -|\F|^{p-2} \F, D m\rangle \varphi^p \, dz}\\
&\quad + p\int_{K_s}{ \langle \hat\A(z,\alpha h, D h) - \A(z, \alpha u, D u) -|\F|^{p-2} \F, 
D \varphi\rangle m\varphi^{p-1} \, dz} +\int_{K_s} f m \varphi^p \, dz.
\end{align*}
Since 
\[
 \int_{K_s}  m_t \varphi^p m\, dz =\int_{K_s} \Big[\big(\varphi^p \frac{m^2}{2}\big)_t -  \frac{p}{2} \varphi^{p-1}
 \varphi_t m^2\Big]\, dz
=\frac12 \int_{B_\frac52} (\varphi^p m^2)(x,s)\, dx -  \frac{p}{2}\int_{K_s}  \varphi^{p-1}\varphi_t m^2\, dz,
 \]
the above identity gives
\begin{align*}
 &\frac12 \int_{B_\frac52} (\varphi^p m^2)(x,s)\, dx +I_s
  =\int_{K_s}{ \langle \hat\A(z,\alpha h, D h)  
 -\A(z,\alpha u, D h) -|\F|^{p-2} \F, Dm\rangle
 \varphi^p \, dz}\\
  &+p\int_{K_s}{ \langle  \hat\A(z, \alpha h, D h)- \A(z, \alpha u, D u) -|\F|^{p-2} \F, D \varphi\rangle m
 \varphi^{p-1} 
  dz} +\int_{K_s} f m \varphi^p dz + \frac{p}{2}\int_{K_s}  \varphi^{p-1}\varphi_t m^2 dz,
\end{align*}
where
\[I_s :=\int_{K_s}{ \langle \A(z, \alpha u, D u) - \A(z,\alpha u,D h), Dm  \rangle \varphi^p\, dz }.
 \]

As $\hat\A(z, \alpha h, D h)  
 -\A(z,\alpha u, D h)= [\hat\A( z, \alpha h, D h) - \hat\A(z,\alpha u,D h)] 
 -[\A(z,\alpha u, D h) - \hat\A(z, \alpha u,D h)]$, we deduce   from this,  
 conditions \eqref{structural-reference-2}, 
 \eqref{structural-reference-3-reformulation},  and definition \eqref{oscillating-fn}   that
\begin{align*}
\frac12 \int_{B_\frac52} &(\varphi^p m^2)(x,s)\, dx 
+I_s
\leq \int_{K_s}{\Big\{\big[ \omega(\alpha|m|) + \d_{\A,\hat\A}\big] (1+ |D h|^{p-1})
+|\F|^{p-1}\Big\} | D m|  \varphi^p\, dz}\nonumber\\
&+C \int_{K_s}{ \big(1+ |D h|^{p-1} + |D u|^{p-1} +|\F|^{p-1}\big)  
|m| \, dz}+\int_{K_s} |f| |m \varphi^p| dz
+ C\int_{K_s}   m^2\, dz.
\end{align*}
Using H\"older's inequality and \eqref{ass:uF}--\eqref{eq:Dv}, we can bound the above third integral by 
$C\|m\|_{L^p(K_s)}$. As a consequence, we obtain 
\begin{align}\label{mIs}
\frac12 \int_{B_\frac52} (m^2\varphi^p)(x,s)\, dx +I_s
&\leq \int_{K_s}{\left\{\big[\omega(\alpha |m|)  + \d_{\A,\hat\A}\big] (1+ |D h|^{p-1}) +|\F|^{p-1}\right\} 
|D m|\,  \varphi^p \, dz}\nonumber\\
&\quad +\int_{K_s} |f| |m \varphi^p| dz
+C\Big(  \|m\|_{L^p(K_s)} +\|m\|_{L^2(K_s)}^2\Big),
\end{align}
where $C>0$ depends only on $\Lambda$, $p$, and $n$. We next estimate the two integrals on the
right hand side of \eqref{mIs}. 
Using  Young's inequality, \eqref{higher-intergrability-v}, and the boundedness of
$\d_{\A,\hat\A}$ and $\omega$, it follows  for  any $\sigma>0$  that
\begin{align*}
&\int_{K_s}{\left\{\big[\omega(\alpha |m|)  + \d_{\A,\hat\A}\big] (1+ |D h|^{p-1}) +|\F|^{p-1}\right\} 
|D m|\,  \varphi^p \, dz}-\sigma \int_{K_s}{|D m|^p \varphi^p \, dz}\\
&\leq  \frac{C_p}{\sigma^{\frac{1}{p-1}}}
\left\{\int_{K_s}{\big[
  \omega
(\alpha | m|)+ \d_{\A,\hat\A}  \big]^{\frac{p}{p-1}} (1+ |Dh|^p) \, dz} +\int_{K_s} 
|\F|^p\, dz\right\}\nonumber\\
&\leq   \frac{C_p}{\sigma^{\frac{1}{p-1}}}\Bigg\{\Big(\int_{Q_\frac52} 
(1+|D h|^p)^{\frac{p+\e_0}{p}} 
\Big)^{\frac{p}{p+\e_0}}  \Big(\int_{K_s}{\big[
  \omega
(\alpha | m|)+ \d_{\A,\hat\A} \big]^{\frac{p(p+\e_0)}{(p-1)\e_0}}  \, dz}\Big)^{\frac{\e_0}{p+\e_0}} 
+\|\F\|_{L^p(K_s)}^p\Bigg\}
\\
&\leq   \frac{C(\Lambda,p,n)}{ \sigma^{\frac{1}{p-1}}}  
\left\{ \Big(\int_{K_s}{\big[
 \omega
(\alpha | m|) + \d_{\A,\hat\A}  \big]  \, dz}\Big)^{\frac{\e_0}{p+\e_0}} +\|\F\|_{L^p(K_s)}^p\right\}.
\end{align*}
On the other hand, as in \eqref{f-sobolev} and by the properties of $\varphi$ we have
\begin{align*}
 \int_{K_s}   |f|  | m\varphi^p|\, dz
&\leq \sigma \int_{K_s} [ |D  m|^p \varphi^p + |m|^p ]\, dz +\sigma \sup_{t\in (-\frac{25}{4},s)}\int_{B_\frac52} 
(m^2 \varphi^{p})(x,t)\, dx
+ \frac{C(n,p)}{\sigma^{\frac{p+n}{p(n+1)-n}}} \|f\|_{L^{\bar p'}(K_s)}^{\frac{p(n+2)}{p(n+1)-n}}
\end{align*}
for all $ \sigma>0$. 
Therefore, we deduce from  \eqref{mIs}  that
\begin{align}\label{homo-est}
\frac12 \int_{B_\frac52} (m^2\varphi^p)(x,s)\, dx 
&+I_s
\leq \sigma \int_{K_s}{|D m|^p \varphi^p \, dz}
+\frac{\sigma}{2} \sup_{t\in (-\frac{25}{4},s)}\int_{B_\frac52} 
(m^2 \varphi^{p})(x,t)\, dx
\nonumber\\
&+ C \sigma^{\frac{-1}{p-1}} 
\left\{ \Big(\int_{K_s}{\big[
 \omega
(\alpha | m|) + \d_{\A,\hat\A}  \big]  \, dz}\Big)^{\frac{\e_0}{p+\e_0}} +\|\F\|_{L^p(K_s)}^p\right\}
+ C\sigma^{\frac{-(p+n)}{p(n+1)-n}} \|f\|_{L^{\bar p'}(K_s)}^{\frac{p(n+2)}{p(n+1)-n}}\\
&\quad 
+C\Big( \sigma \|m\|_{L^p(K_s)}^p +  \|m\|_{L^p(K_s)} +\|m\|_{L^2(K_s)}^2\Big)\qquad\qquad \forall \sigma>0.\nonumber
\end{align}
Now if $p\geq 2$, then  \eqref{structural-reference-1}  implies that
$\Lambda^{-1}\int_{K_s}{ |D  m|^p \varphi^p dz}\leq I_s$. Hence by combining with
\eqref{homo-est} and choosing $\sigma>0$ sufficiently small, we obtain
\begin{align*}
&\frac12 \int_{B_\frac52} (m^2\varphi^p)(x,s)\, dx 
+\frac{1}{2\Lambda}\int_{K_s}{ |D m|^p\varphi^p dz }
\leq \frac14 \sup_{t\in (-\frac{25}{4},\frac{25}{4})}\int_{B_\frac52} 
(m^2 \varphi^{p})(x,t)\, dx\\
&+ C \Bigg\{ \Big(\int_{Q_\frac52}{\big[
 \omega
(\alpha | m|) + \d_{\A,\hat\A}  \big]  \, dz}\Big)^{\frac{\e_0}{p+\e_0}} 
+\|\F\|_{L^p(Q_\frac52)}^p+\|f\|_{L^{\bar p'}(Q_\frac52)}^{\frac{p(n+2)}{p(n+1)-n}} +\|m\|_{L^p(Q_\frac52)}^p 
+ \|m\|_{L^p(Q_\frac52)}+\|m\|_{L^2(Q_\frac52)}^2\Bigg\}
\end{align*}
for every $s\in (-\frac{25}{4},\frac{25}{4})$. This implies 
  (i) since $\varphi =1$ on $Q_2$.
On the other hand,  if $1<p<2$ then  Lemma~\ref{simple-est} together with  \eqref{ass:uF}  yields
\[
c \tau^{\frac{2-p}{p}} \int_{Q_\frac52} |D  m|^p \varphi^p dz 
- C\, \tau^{\frac{2}{p}}\leq 
c \tau^{\frac{2-p}{p}} \int_{Q_\frac52} |D  m|^p \varphi^p dz 
- 2c \, \tau^{\frac{2}{p}} \int_{Q_\frac52} (1+|D u|^p)\varphi^p  dz
 \leq I_s
\]
for all $\tau>0$ small. By combining this with \eqref{homo-est} and taking $c \tau^{\frac{2-p}{p}}=2\sigma$, 
 we deduce for  $\sigma>0$ small that 
\begin{align}\label{mI-casep<2}
\frac12 \int_{B_\frac52} (m^2\varphi^p)(x,s)\, dx 
&+\sigma \int_{K_s}{|D m|^p \varphi^p \, dz}
\leq   C\sigma^{\frac{2}{2-p}}   +\frac{1}{4} \sup_{t\in (-\frac{25}{4},\frac{25}{4})}\int_{B_\frac52} 
(m^2 \varphi^{p})(x,t)\, dx\nonumber\\
&+ C \sigma^{\frac{-1}{p-1}}   
\Bigg\{ \Big(\int_{Q_\frac52}{\big[
 \omega
(\alpha | m|) + \d_{\A,\hat\A}  \big]  dz}\Big)^{\frac{\e_0}{p+\e_0}} +\|\F\|_{L^p(Q_{\frac52})}^p\Bigg\}
+ C \sigma^{\frac{-(p+n)}{p(n+1)-n}} \|f\|_{L^{\bar p'}(Q_\frac52)}^{\frac{p(n+2)}{p(n+1)-n}}\\
&+C\Big( \|m\|_{L^p(Q_\frac52)}^p +  \|m\|_{L^p(Q_\frac52)} +\|m\|_{L^2(Q_\frac52)}^2\Big)\nonumber
\end{align}
for every $s\in (-\frac{25}{4},\frac{25}{4})$. In particular, we get
\begin{align*}
\frac{1}{4} \sup_{t\in (-\frac{25}{4},\frac{25}{4})}\int_{B_\frac52} 
(m^2 \varphi^{p})(x,t)\, dx
&\leq   C\sigma^{\frac{2}{2-p}} + C \sigma^{\frac{-1}{p-1}} 
\Bigg\{ \Big(\int_{Q_\frac52}{\big[
 \omega
(\alpha | m|) + \d_{\A,\hat\A}  \big]  \, dz}\Big)^{\frac{\e_0}{p+\e_0}} +\|\F\|_{L^p(Q_{\frac52})}^p\Bigg\}\\
&+C \sigma^{\frac{-(p+n)}{p(n+1)-n}} \|f\|_{L^{\bar p'}(Q_\frac52)}^{\frac{p(n+2)}{p(n+1)-n}}
+C\Big(\|m\|_{L^p(Q_\frac52)}^p +  \|m\|_{L^p(Q_\frac52)} +\|m\|_{L^2(Q_\frac52)}^2\Big).
\end{align*}
 This together with \eqref{mI-casep<2}  gives (ii) as desired.
\end{proof}

\subsection{A compactness argument}\label{sub:compactness}
In order to  verify the local Lipschitz approximation property, we compare gradients of solutions of our equation to those of the corresponding frozen equation. To this end,   we employ a compactness
argument in two steps. In the first step, we reduce the problem to the homogeneous case 
(Lemma~\ref{lm:compare-solution-1}). We then handle the  homogeneous equation in the second step 
(Lemma~\ref{lm:compare-solution-2}) by making use of the higher integrability stated in Theorem~\ref{thm:higherint}. It is crucial  that the constants $\delta$ in these two lemmas can be chosen to be independent of the parameter $\alpha$.
\begin{lemma}[reduction to homogeneous equations]\label{lm:compare-solution-1} Assume that $p> 2n/(n+2)$ and $M_0\in (0,\infty)$. Let $\A$ satisfy 
\eqref{structural-reference-1}--\eqref{structural-reference-3}, and $\A(\cdot,\cdot,0)=0$.
For any $\e>0$, there exists $\delta_1>0$ depending only on $\e$, $\Lambda$, $p$,    $n$, $\K$, and $M_0$  such that:  
if $\alpha>0$,
$\,\fint_{Q_4} |\F|^p\, dz +\Big(\fint_{Q_4}  | f|^{\bar p'} \, dz\Big)^{\hat p} \leq \delta_1^p$,
and $u$ is a weak solution of 
\begin{equation*}\label{eq-for-u}
u_t =\div \A(z, \alpha u, D u) + \div(|\F|^{p-2}\F) +f \quad \text{in}\quad Q_4
\end{equation*}
 satisfying
\begin{equation*}
\|u\|_{L^\infty(Q_4)}\leq \frac{M_0}{\alpha} \quad\mbox{ and }\quad 
\fint_{Q_4}{ |D u|^p\, dz} \leq 1,
\end{equation*}
and $w$  is a weak solution of
\begin{equation*}\label{eq-for-w}
\left \{
\begin{array}{lcll}
w_t &=&\div \A( z,\alpha w,D w)  \quad &\text{in}\quad Q_\frac72, \\
w & =& u\quad &\text{on}\quad \partial_p Q_\frac72,
\end{array}\right.
\end{equation*}
 then 
\begin{equation}\label{u-w-close}
\int_{Q_3}{|Du - D w|^p\, dz}\leq \e^p.
\end{equation}
\end{lemma}
\begin{proof}
We  prove \eqref{u-w-close} by contradiction. Suppose that estimate \eqref{u-w-close} is not true. Then there exist
$\e_0,\, p, \, \Lambda,   \, n,\, \K, \, M_0$, 
a sequence of  positive numbers  $\{\alpha_k\}_{k=1}^\infty$, a sequence 
 $\{\A^k\}_{k=1}^\infty$  satisfying  structural 
 conditions \eqref{structural-reference-1}--\eqref{structural-reference-3} and $\A^k(\cdot,\cdot,0)=0$, 
 and   sequences of functions  $\{\F_k\}_{k=1}^\infty$,  $\{f_k\}_{k=1}^\infty$,  $\{u^k\}_{k=1}^\infty$ 
such that
\begin{equation}\label{F_k-condition}
\fint_{Q_4} |\F_k|^p\, dz +\Big(\fint_{Q_4}  | f_k |^{\bar p'} \, dz\Big)^{\hat p} \leq \frac{1}{k^p}, 
\end{equation}
 $u^k$ is a weak solution of 
\begin{equation*}
u^k_t =\div  \A^k(z,\alpha_k u^k,D u^k) +\div(|\F_k|^{p-2}\F_k) +f_k \quad \text{in}\quad Q_4
\end{equation*}
with 
\begin{equation}\label{gradient-bounded-ass-1}
 \|u^k\|_{L^\infty(Q_4)}\leq \frac{M_0}{\alpha_k} \quad\mbox{ and }\quad 
\fint_{Q_4}{ |D  u^k|^p\, dz} \leq 1,
\end{equation}
\begin{equation}\label{contradiction-conclusion-1}
\int_{Q_3} |D u^k - Dw^k|^p \, dz > \e_0^p  \quad\mbox{for all } k.
\end{equation}
Here  $w^k$  is a weak solution  of
\begin{equation*}
\left \{
\begin{array}{lcll}
w^k_t &=& \div  \A^k(z,\alpha_k w^k,D w^k)  \quad &\text{in}\quad Q_\frac72, \\
w^k & =& u^k\quad &\text{on}\quad \partial_p Q_\frac72.
\end{array}\right.
\end{equation*}
Using Proposition~\ref{pro:compa},
Lemma~\ref{gradient-est-II}, and \eqref{F_k-condition}--\eqref{gradient-bounded-ass-1}, we obtain
\begin{equation}\label{bounds-for-wk}
 \|w^k\|_{L^\infty(Q_\frac72)}\leq \frac{M_0}{\alpha_k} \quad\mbox{ and }\quad 
\fint_{Q_\frac72}{ |D  w^k|^p\, dz} \leq C(\Lambda,p, n).
\end{equation}
If the sequence $\{\alpha_k\}$ has a subsequence converging to $+\infty$, then 
we infer from the fact 
$\|D u^k-D w^k\|_{L^p(Q_3)}
\leq\|D u^k\|_{L^p(Q_3)} +\| D w^k\|_{L^p(Q_3)}$,  Lemma~\ref{energy-estimate},  estimates 
\eqref{F_k-condition}--\eqref{gradient-bounded-ass-1} 
 and \eqref{bounds-for-wk} that
\begin{align*}
\liminf_{k\to \infty}\int_{Q_3} |D u^k-D w^k|^p\, dz=0
\end{align*}
which contradicts \eqref{contradiction-conclusion-1}. 
Thus, we conclude that 
 $\{\alpha_k\}$ is  bounded and hence there exist a subsequence (still labeled  $\{\alpha_k\}$) and  a constant  $\alpha\in [0,\infty)$   such that
 $\alpha_k \to \alpha$. Since $\alpha$ could be zero,  the sequences $\{u^k\}$ and $\{w^k\}$ might be  unbounded in $L^p(Q_\frac72)$.
In spite of that, we claim that: up to a subsequence, there holds
\begin{equation}
\label{claim:u^kw^k->0}
 \lim_{k\to\infty }\Big[\|u^k - w^k\|_{L^p(Q_\frac72)}
 +\|u^k - w^k\|_{L^2(Q_\frac72)}\Big]=0.
\end{equation}
In order to prove \eqref{claim:u^kw^k->0}, we first  note that
by applying  Lemma~\ref{gradient-est-II}  for $u \rightsquigarrow  u^k$, $\F \rightsquigarrow  \F_k$, 
$f \rightsquigarrow  f_k$,
$v 
\rightsquigarrow  w^k$, 
and using \eqref{F_k-condition}--\eqref{gradient-bounded-ass-1}, we get 
\begin{equation*}
 \sup_{t\in (\frac{-49}{4},\frac{49}{4})}\int_{B_\frac72} |u^k - w^k|^2\, dx
+ \int_{Q_\frac72}  |D u^k - D w^k|^p \, dz\leq C. 
\end{equation*}
This  together with  the parabolic embedding  (see \cite[Proposition~3.1, page~7]{D2})   gives 
\begin{align}\label{embedding-1}
\int_{Q_\frac72} |u^k - w^k|^{p\frac{n+2}{n}}\, dz \leq 
C(n,p)\Big(\int_{Q_\frac72}  |D u^k - D w^k|^p \, dz\Big)
\Big(\sup_{t}\int_{B_\frac72} |u^k - w^k|^2\, dx \Big)^{\frac{p}{n}}\leq C.
\end{align}
In particular, $ \|u^k - w^k\|_{L^p(Q_\frac72)} 
\leq  C$.
Thus  there exist  subsequences, still denoted by $\{u^k\}$ and $\{w^k\}$, and   a function $m(z)$ such that $u^k-w^k\to m$ strongly in $L^p(Q_\frac72)$ and $D (u^k -  w^k) \rightharpoonup Dm$ 
weakly in $L^p(Q_\frac72)$. We next show that $m(z)=0$ for a.e. $z\in Q_\frac72$.

Let $m^k := u^k - w^k$. By taking a subsequence if necessary, we can assume that
$m^k(z)\to m(z)$ for a.e. $z\in Q_\frac72$.
For $\e>0$ small, we define the following continuous approximation to the $\sgn^+$  function:
\begin{equation}\label{def:h_e}
h_\e(s) := \left\{\begin{array}{rl}
\!\!& \ 1  \quad \mbox{ for }\quad s\geq \e, \\
\!\!& \ \frac{s}{\e}  \quad \mbox{ for }\quad 0\leq s< \e \\ 
\!\!& \ 0   \quad \mbox{ for }\quad s< 0. 
\end{array}\right.
\end{equation}
 By using $h_\e(m^k)$ as a test function in the equations for $u^k$ and $w^k$ and subtracting the resulting expressions, we obtain:
\begin{align*}
&\int_{B_\frac72}\Big(\int_0^{
m^k_+(x,t)}  h_\e(s)\, ds \Big)dx 
+\int_0^t\int_{B_\frac72}\langle \A^k(z,\alpha_k u^k,D u^k)-\A^k(z,\alpha_k u^k,D w^k) , Dm^k\rangle h_\e'(m^k)\, dz\nonumber\\
&= \int_0^t\int_{B_\frac72} f_k h_\e(m^k)  dz+  \int_0^t\int_{B_\frac72} \langle 
\A^k(z,\alpha_k w^k,D w^k)-\A^k(z,\alpha_k u^k,D w^k) -|\F_k|^{p-2}\F_k, Dm^k\rangle h_\e'(m^k)
dz
\end{align*}
for all $t\in (-49/4,49/4)$, where $m^k_+(x,t) := \max{\{m^k(x,t), 0\}}$. Hence, it follows from \eqref{structural-reference-3}  that
\begin{align}\label{eq:uw-integral-est}
\int_{B_\frac72}\Big(\int_0^{
m^k_+(x,t)}  h_\e(s)\, ds \Big)dx 
&+\int_0^t\int_{B_\frac72}\langle \A^k(z,\alpha_k u^k,D u^k)-\A^k(z,\alpha_k u^k,D w^k) , Dm^k\rangle h_\e'(m^k)\, dz\nonumber\\
&\leq \int_{Q_\frac72} |f_k|\, dz
+ \int_0^t\int_{B_\frac72}\Big[ 
\Lambda \alpha_k m^k(1+|Dw^k|^{p-1})+ |\F_k|^{p-1}\Big] |Dm^k|h_\e'(m^k) \, dz.
\end{align}
Let us consider the following two cases.

{\bf Case 1: $\mathbf{ p\geq 2}$.} Then by applying  structural condition \eqref{structural-reference-1} to the second integral in \eqref{eq:uw-integral-est} and Young's inequality to the last integral, we obtain after canceling like terms that
\[\int_{B_\frac72}\Big(\int_0^{m^k_+(x,t)}  h_\e(s) ds \Big)dx 
\leq \int_{Q_\frac72} |f_k| dz +  C(\Lambda, p)\int_0^t\int_{B_\frac72}\Big [ 
\alpha_k^{\frac{p}{p-1}} (m^k)^{\frac{p}{p-1}} (1+|Dw^k|^p)+ |\F_k|^p\Big] h_\e'(m^k) dz.
 \]
Thanks to the definition of  $h_\e$ in \eqref{def:h_e} and owing to \eqref{bounds-for-wk}, we infer that
\begin{align*}
\int_{B_\frac72}\Big(\int_0^{m^k_+(x,t)}  h_\e(s)\, ds \Big)dx 
&\leq  \int_{Q_\frac72} |f_k|\, dz +  C(\Lambda, p,n)\Bigg[\alpha_k^{\frac{p}{p-1}} \e^{\frac{1}{p-1}} 
+\frac{1}{\e } \int_{Q_\frac72} |\F_k|^p\, dz\Bigg].
\end{align*}
Hence, by letting $k\to\infty$ and using \eqref{F_k-condition} we  obtain 
\[
 \int_{B_\frac72}\Big(\int_0^{m_+(x,t)}  h_\e(s)\, ds \Big)dx\leq  C(\Lambda, p,n)\alpha^{\frac{p}{p-1}} \e^{\frac{1}{p-1}} .
\]
We next let $\e \to 0^+$ to get
$\int_{B_\frac72} m_+(x,t)\,dx \leq 0$ for every $t\in (-49/4,49/4)$.
We then conclude that 
\[
\int_{Q_\frac72} m_+(x,t)\,dx dt=0,
\]
and hence $m(z)\leq 0$ for a.e. $z\in Q_\frac72$.

{\bf Case 2: $\mathbf{\frac{2n}{n+2}< p< 2}$.} Then by applying Lemma~\ref{simple-est} to the second integral in \eqref{eq:uw-integral-est} and Young's inequality to the last integral, we get for any $\tau\in (0,1/2)$ and any $\sigma>0$ that
\begin{align*}
&\int_{B_\frac72}\Big(\int_0^{m^k_+(x,t)}  h_\e(s)\, ds \Big)dx 
+c \tau^{\frac{2-p}{p}}\int_0^t\int_{B_\frac72} |Dm^k|^p h_\e'(m^k) \, dz
-2c \tau^{\frac2p}\int_0^t\int_{B_\frac72} (1+|Du^k|^p) h_\e'(m^k) \, dz\\
&\leq  \sigma \int_0^t\int_{B_\frac72} |Dm^k|^p h_\e'(m^k) \, dz
+\int_{Q_\frac72} |f_k|\, dz\\
&\quad + C(\Lambda,p) \sigma^{\frac{-1}{p-1}} \int_0^t\int_{B_\frac72}\Big [ 
\alpha_k^{\frac{p}{p-1}} (m^k)^{\frac{p}{p-1}} (1+|Dw^k|^p)+ |\F_k|^p\Big] h_\e'(m^k) \, dz.
 \end{align*}
We now take  $\tau>0$  such that 
$c \tau^{\frac{2-p}{p}}=\sigma$, use \eqref{gradient-bounded-ass-1} to bound the above third integral and
use  \eqref{bounds-for-wk} to bound the last integral.  As a consequence, we obtain
\begin{align*}
\int_{B_\frac72}\Big(\int_0^{m^k_+(x,t)}  h_\e(s)\, ds \Big)dx 
&\leq  C(\Lambda,p)
\frac{1}{\e} \sigma^{\frac{2}{2-p}}+\int_{Q_\frac72} |f_k|\, dz\\
&+ 
C(\Lambda,p,n)\sigma^{\frac{-1}{p-1}} \Big[\alpha_k^{\frac{p}{p-1}} \e^{\frac{1}{p-1}}  +\frac{1}{\e}
\int_{Q_\frac72} |\F_k|^p\, dz\Big]\quad\forall \sigma>0 \quad\mbox{small}.
 \end{align*}
Therefore, letting $k\to\infty$ and using \eqref{F_k-condition} yield
\begin{align*}
\int_{B_\frac72}\Big(
\int_0^{m_+(x,t)}  h_\e(s)\, ds \Big)dx 
\leq  C\Big[
\frac{1}{\e} \sigma^{\frac{2}{2-p}}+ \sigma^{\frac{-1}{p-1}} \alpha^{\frac{p}{p-1}} \e^{\frac{1}{p-1}}\Big]\quad\forall \sigma>0 \quad\mbox{small}.
 \end{align*}
Let us minimize the right hand side by choosing $\sigma=\e^{2-p}$ to obtain
\begin{align*}
\int_{B_\frac72}\Big(
\int_0^{m_+(x,t)}  h_\e(s)\, ds \Big)dx 
\leq  C\big(
1 +  \alpha^{\frac{p}{p-1}} \big)\, \e\quad\forall \e>0 \quad\mbox{small}.
 \end{align*}
 We next let $\e \to 0^+$ to get as in {\bf Case 1} that
$\int_{Q_\frac72} m_+(x,t)\,dx dt=0$,
implying $m(z)\leq 0$  a.e. in $Q_\frac72$.

Thus we have shown in both cases that $m(z)\leq 0$  a.e. in $ Q_\frac72$. By interchanging the role of $u^k$ and $w^k$, we also have $m(z)\geq 0$  a.e. in
$ Q_\frac72$. Thus, $m= 0$ in $Q_\frac72$ and so $u^k-w^k\to 0$ in $L^p(Q_\frac72)$. This implies  claim  \eqref{claim:u^kw^k->0} in the case $p\geq 2$. For the case $2n/(n+2)<p<2$, by taking 
$\theta := \frac{n+2}{n} -\frac{n}{p}$ we have 
$\theta\in (0,1)$ and 
\[2= \theta p + (1-\theta) p\frac{n+2}{n}.
 \]
 Therefore, by interpolation and using  estimate \eqref{embedding-1} we obtain
\begin{align*}
\int_{Q_\frac72} |u^k - w^k|^2 \, dz \leq \Big(\int_{Q_\frac72} |u^k - w^k|^p \, dz\Big)^\theta \Big(\int_{Q_\frac72} |u^k - w^k|^{p\frac{n+2}{n}}\, dz\Big)^{1-\theta}\leq C \Big(\int_{Q_\frac72} |u^k - w^k|^p \, dz\Big)^\theta.
\end{align*}
Thus $\|u^k - w^k\|_{L^2(Q_\frac72)}\to 0$ and claim \eqref{claim:u^kw^k->0} follows in this case as well.

 We now use \eqref{claim:u^kw^k->0} to derive a contradiction. 
By applying   Lemma~\ref{lm:compare-gradient-part1} for $\alpha \rightsquigarrow  \alpha_k$, $u \rightsquigarrow  
u^k$, $h \rightsquigarrow  w^k$, $\F \rightsquigarrow  \F_k$,  $f \rightsquigarrow  f_k$ and using \eqref{F_k-condition} together with the 
facts $\omega(r)\leq \Lambda r$ and   $\{\alpha_k\}$ is bounded, we obtain
\begin{align*}
\int_{Q_3}{ |D m^k|^p dz }
&\leq C\Big(\|m^k\|_{L^p(Q_{\frac72})}^p +
\|m^k\|_{L^p(Q_{\frac72})} + \|m^k\|_{L^2(Q_{\frac72})}^2\Big)\\
&\quad +C \Bigg[ \| m^k\|_{L^1(Q_\frac72)}^{\frac{\e_0}{p+\e_0}} 
+\|\F_k\|_{L^p(Q_{\frac72})}^p+\Big(\int_{Q_\frac72}  | f_k|^{\bar p'} \, dz\Big)^{\hat p}\Bigg]
 \end{align*}
for the case $p\geq 2$. Letting $k\to \infty$ and making use of \eqref{F_k-condition} and \eqref{claim:u^kw^k->0}, we conclude that \begin{equation}\label{Du-Dw->0}
\lim_{k\to\infty}\int_{Q_3}{ |D u^k-D w^k|^p dz }= \lim_{k\to\infty}\int_{Q_3}{ |D m^k|^p dz }=0.
\end{equation}
On the other hand, for the case $p<2$ we get from Lemma~\ref{lm:compare-gradient-part1} that
\begin{align*}
\int_{Q_3} |Dm^k|^p  dz
&\leq  C \sigma^{\frac{p}{2-p}} +C\sigma^{-1} \Big(\|m^k\|_{L^p(Q_{\frac72})}^p + \|m^k\|_{L^p(Q_{\frac72})}
+\|m^k\|_{L^2(Q_{\frac72})}^2\Big)\\
&\quad + C \sigma^{\frac{-p}{p-1}} \Big( \| m^k\|_{L^1(Q_\frac72)}^{\frac{\e_0}{p+\e_0}} +
\|\F_k\|_{L^p(Q_{\frac72})}^p\Big)
+C \sigma^{\frac{-p(n+2)}{p(n+1)-n}}\Big(\int_{Q_\frac72}  |f_k|^{\bar p'} \, dz\Big)^{\hat p}
\end{align*}
for all $\sigma>0$ small. By first taking $k\to\infty$ and then taking $\sigma\to 0^+$, we still arrive at \eqref{Du-Dw->0}.  

As \eqref{Du-Dw->0} contradicts  \eqref{contradiction-conclusion-1}, we have produced a contradiction and the lemma is proved.
\end{proof}

In the next result, we  only deal with  the case $\F=0$ and $f=0$, that is, weak solutions of homogeneous equations. 
The spatial gradients of these 
weak solutions enjoy the self improving property (Theorem~\ref{thm:higherint}) which plays an important role in our proof.
\begin{lemma}[gradient approximation for homogeneous equations]\label{lm:compare-solution-2} Assume that $p> 2n/(n+2)$ and $M_0\in (0,\infty)$. Let $\A$ satisfy 
\eqref{structural-reference-1}--\eqref{structural-reference-3}, and $\A(\cdot,\cdot,0)=0$.
For any $\e>0$, there exists $\delta_2>0$ depending only on $\e$, $\Lambda$, $p$,    $n$, $\K$, and $M_0$  such that:  if $\alpha>0$,
\begin{equation*}
\fint_{Q_4} \Big[
\sup_{u\in \overline\K}\sup_{ \xi\in\R^n}\frac{| \A(x,t, u,\xi) -  \A_{B_4}(t,u,\xi)|}{1+|\xi|^{p-1}}
\Big] \, dx dt  \leq \delta_2^p,
\end{equation*}
and $w$ is a weak solution of 
$w_t =\div \A(z, \alpha w, D w)$ in  $ Q_\frac72$
 satisfying
\begin{equation*}
\|w\|_{L^\infty(Q_\frac72)}\leq \frac{M_0}{\alpha} \quad\mbox{ and }\quad 
\fint_{Q_\frac72}{ |D w|^p\, dz} \leq C(\Lambda, p, n),
\end{equation*}
and $v$  is a weak solution of
\begin{equation*}\label{eq-for-v}
\left \{
\begin{array}{lcll}
v_t &=&\div \A_{B_4}(t, \alpha v,D v)  \quad &\text{in}\quad Q_3, \\
v & =& w\quad &\text{on}\quad \partial_p Q_3,
\end{array}\right.
\end{equation*}
 then 
\begin{equation}\label{u-v-close}
\int_{Q_2}{|Dw - D v|^p\, dz}\leq \e^p.
\end{equation}
\end{lemma}
\begin{proof}
 Suppose by contradiction that estimate \eqref{u-v-close} is not true. Then there exist
$\e_0,\, p, \, \Lambda,   \, n,\, \K, \, M_0$, 
 a sequence of  positive numbers  $\{\alpha_k\}_{k=1}^\infty$, a sequence 
 $\{\A^k\}_{k=1}^\infty$  satisfying structural 
 conditions \eqref{structural-reference-1}--\eqref{structural-reference-3} and $\A^k(\cdot,\cdot,0)=0$, 
 and   a sequence of functions   $\{w^k\}_{k=1}^\infty$ 
such that
\begin{equation}\label{Theta_k-condition}
\fint_{Q_4}\Theta_k(z) \, dz \leq \frac{1}{k^p} \qquad \mbox{ with }\quad 
\Theta_k(x,t) 
:= \d_{\A^k, \A^k_{B_4}}(x,t), 
\end{equation}
 $w^k$ is a weak solution of 
$w^k_t =\div \A^k(z,\alpha_k w^k,D w^k)$ in  $Q_\frac72$
with 
\begin{equation}\label{gradient-bounded-ass}
 \|w^k\|_{L^\infty(Q_\frac72)}\leq \frac{M_0}{\alpha_k} \quad\mbox{ and }\quad 
\fint_{Q_\frac72}{ |D  w^k|^p\, dz} \leq C(\Lambda,p, n),
\end{equation}
\begin{equation}\label{contradiction-conclusion}
\int_{Q_2} |D w^k - Dv^k|^p \, dz > \e_0^p  \quad\mbox{for all } k.
\end{equation}
Here  $v^k$  is a weak solution  of
\begin{equation*}
\left \{
\begin{array}{lcll}
v^k_t &=& \div  \A^k_{B_4}(t,\alpha_k v^k,D v^k)  \quad &\text{in}\quad Q_3, \\
v^k & =& w^k\quad &\text{on}\quad \partial_p Q_3.
\end{array}\right.
\end{equation*}
We have from  Theorem~\ref{thm:higherint}   and \eqref{gradient-bounded-ass} that
\begin{equation}\label{higher-w^k}
\int_{Q_3} |D w^k|^{p+\e_0}\, dz \leq C(\Lambda, n,p).
\end{equation}
Also, by using Proposition~\ref{pro:compa},
Lemma~\ref{gradient-est-II} for  $\F\rightsquigarrow 0$, $f\rightsquigarrow 0$, $\A \rightsquigarrow  \A^k$, $\hat\A \rightsquigarrow  \A^k_{B_4}$,  and \eqref{gradient-bounded-ass}, we obtain
\begin{equation}\label{bounds-for-vk}
 \|v^k\|_{L^\infty(Q_3)}\leq \frac{M_0}{\alpha_k} \quad\mbox{ and }\quad 
\fint_{Q_3}{ |D  v^k|^p\, dz} \leq C(\Lambda,p, n).
\end{equation}
If the sequence $\{\alpha_k\}$ has a subsequence converging to $+\infty$, then 
we infer from
 Lemma~\ref{energy-estimate}, \eqref{gradient-bounded-ass}, 
 and \eqref{bounds-for-vk} that
\begin{align*}
\liminf_{k\to \infty}\int_{Q_2} |D w^k-D v^k|^p\, dz=0
\end{align*}
which contradicts \eqref{contradiction-conclusion}. 
Thus, we conclude that 
 $\{\alpha_k\}$ is  bounded and hence there exist a subsequence (still labeled  $\{\alpha_k\}$) and  a constant  $\alpha\in [0,\infty)$   such that
 $\alpha_k \to \alpha$. We claim that: up to a subsequence, there holds
\begin{equation}
\label{claim:w^kv^k->0}
 \lim_{k\to\infty }\Big[\|w^k - v^k\|_{L^p(Q_3)}
 +\|w^k - v^k\|_{L^2(Q_3)}\Big]=0.
\end{equation}
In order to prove \eqref{claim:w^kv^k->0}, we first  note that
by applying Lemma~\ref{gradient-est-II}  for $u \rightsquigarrow  w^k$, $\F \rightsquigarrow  0$,
$f \rightsquigarrow  0$, $v \rightsquigarrow  v^k$, 
and using  \eqref{gradient-bounded-ass}, we get 
\[
 \sup_{t\in (-9,9)}\int_{B_3} |w^k -  v^k|^2\, dx
+ \int_{Q_3}  |D w^k - D v^k|^p \, dz\leq C. 
\]
This  together with  the parabolic embedding gives 
\begin{equation}\label{difference-control}
\int_{Q_3} |w^k - v^k|^{p\frac{n+2}{n}} \, dz 
\leq  C.
\end{equation}
Thus  there exist  subsequences, still denoted by $\{w^k\}$ and $\{v^k\}$, and   a function $m(z)$ such that $w^k-v^k\to m$ strongly in $L^p(Q_3)$ and $D (w^k -  v^k) \rightharpoonup Dm$ 
weakly in $L^p(Q_3)$. We next show that $m(z)=0$ in $Q_3$.

Let $m^k := w^k - v^k$. By taking a further subsequence, we can assume that
$m^k(z)\to m(z)$ for a.e. $z\in Q_3$. Let $h_\e(s)$ be given by \eqref{def:h_e}.
By using $h_\e(m^k)$ as a test function in the equations for $w^k$ and $v^k$ and subtracting the
resulting expressions, we obtain:
\begin{align*}
\int_{B_3}\Big(\int_0^{m^k_+(x,t)}  h_\e(s)\, ds \Big)dx 
&+\int_0^t\int_{B_3}\langle \A^k_{B_4}(t,\alpha_k v^k,D w^k)-\A^k_{B_4}(t,\alpha_k v^k,D v^k) , Dm^k\rangle h_\e'(m^k)\, dz\\
&=  \int_0^t\int_{B_3} \langle \A^k_{B_4}(t,\alpha_k v^k,D w^k)-\A^k(z,\alpha_k w^k,D w^k) , Dm^k
\rangle h_\e'(m^k)\,
dz
\end{align*}
for all $t\in (-9,9)$.
Let us first assume that $p\geq 2$.
 Writing $\A^k_{B_4}(t,\alpha_k v^k,D w^k)-\A^k(z,\alpha_k w^k,D w^k)=
[\A^k_{B_4}(t,\alpha_k v^k,D w^k)-\A^k(z,\alpha_k v^k,D w^k)]
+[\A^k(z,\alpha_k v^k,D w^k)-\A^k(z,\alpha_k w^k,D w^k)]$ and
using  structural conditions \eqref{structural-reference-1} and \eqref{structural-reference-3}, we then deduce that
\begin{align*}
\int_{B_3}\Big(\int_0^{m^k_+(x,t)}  h_\e(s)\, ds \Big)dx +
&\Lambda^{-1}\int_0^t\int_{B_3} |Dm^k|^p h_\e'(m^k) \, dz\\
&\leq    \int_0^t\int_{B_3} \big(\Theta_k + 
\Lambda\alpha_k m^k\big)(1+|Dw^k|^{p-1}) |Dm^k|h_\e'(m^k) \, dz
 \nonumber. 
\end{align*}
It follows by applying Young's inequality to the last integral and canceling  like terms that
\begin{align*}
&\int_{B_3}\Big(\int_0^{m^k_+(x,t)}  h_\e(s)\, ds \Big)dx
\leq   C(\Lambda,p)\int_0^t\int_{B_3} \Big[\Theta_k^{\frac{p}{p-1}} + 
\alpha_k^{\frac{p}{p-1}} (m^k)^{\frac{p}{p-1}}\Big](1+|Dw^k|^p)  h_\e'(m^k) \, dz\\
 &\leq  C(\Lambda, p)\Bigg[\frac1\e\int_{Q_3}\Theta_k^{\frac{p}{p-1}} (1+|Dw^k|^p)\, dz
 +\alpha_k^{\frac{p}{p-1}} \e^{\frac{1}{p-1}}\int_{Q_3} (1+|Dw^k|^p)\, dz\Bigg]\\
 &\leq  C(\Lambda, p)\Bigg[\frac1\e
 \Big(\int_{Q_3}\Theta_k^{\frac{p(p+\e_0)}{(p-1)\e_0}}\, dz\Big)^{\frac{\e_0}{p+\e_0}}
\Big(\int_{Q_3} 
 (1+|Dw^k|^{p+\e_0})\, dz\Big)^{\frac{p}{p+\e_0}}
 +\alpha_k^{\frac{p}{p-1}} \e^{\frac{1}{p-1}}\int_{Q_3} (1+|Dw^k|^p)\, dz\Bigg].
\end{align*}
Therefore, we can use  \eqref{higher-w^k}  and the fact $\{\Theta_k\}$ is bounded to deduce that
\begin{align*}
\int_{B_3}\Big(\int_0^{m^k_+(x,t)}  h_\e(s)\, ds \Big)dx 
\leq  C\Big[\frac1\e
 \Big(\int_{Q_3} \Theta_k\, dz\Big)^{\frac{\e_0}{p+\e_0}}
 +\alpha_k^{\frac{p}{p-1}} \e^{\frac{1}{p-1}}\Big].
\end{align*}
Hence, by letting $k\to\infty$ and making use of \eqref{Theta_k-condition} we  obtain
\[
 \int_{B_3}\Big(\int_0^{m_+(x,t)}  h_\e(s)\, ds \Big)dx\leq C \alpha^{\frac{p}{p-1}} \e^{\frac{1}{p-1}}.
\]
We next let $\e \to 0^+$ to get
$\int_{B_3} m_+(x,t)\,dx \leq 0$ for every $ t\in (-9,9)$.
We then conclude that 
\[
\int_{Q_3} m_+(x,t)\,dx dt=0,
\]
and hence $m(z)\leq 0$ for a.e. $z\in Q_3$. The above arguments can be modified as done in the proof of Lemma~\ref{lm:compare-solution-1}
to get the same conclusion for the case $2n/(n+2)<p<2$ as well. Now by interchanging the role of $w^k$ and $v^k$, we also have $m(z)\geq 0$ for a.e.
$z\in Q_3$. Thus, $m= 0$ in $Q_3$ and so $w^k-v^k\to 0$ in $L^p(Q_3)$. In the case $p<2$, we can interpolate  
this $L^p$ convergence  with \eqref{difference-control}  
 as in the proof of Lemma~\ref{lm:compare-solution-1} 
 to infer further 
that $w^k-v^k\to 0$ in $L^2(Q_3)$. Thus,  claim \eqref{claim:w^kv^k->0} is proved.

We now use \eqref{claim:w^kv^k->0} to derive a contradiction. 
By applying   Lemma~\ref{lm:compare-gradient-part1} for $\alpha \rightsquigarrow 
\alpha_k$, $u \rightsquigarrow  w^k$, $h \rightsquigarrow  v^k$, $\F \rightsquigarrow  0$, $f \rightsquigarrow  0$, $\A \rightsquigarrow  
\A^k$, $\hat\A \rightsquigarrow  \A^k_{B_4}$, and using  the facts $\omega(r)\leq \Lambda r$ and $\{\alpha_k\}$ is bounded, we obtain
\begin{align*}
\int_{Q_2}{ |D m^k|^p dz }\leq C \Bigg[ \|m^k\|_{L^p(Q_3)}^p+ \|m^k\|_{L^p(Q_3)} + \|m^k\|_{L^2(Q_3)}^2
+ \Big(\| m^k\|_{L^1(Q_3)}+ \| \Theta_k\|_{L^1(Q_3)}\Big)^{\frac{\e_0}{p+\e_0}} 
\Bigg]
 \end{align*}
for the case $p\geq 2$. Letting $k\to \infty$ and making use of \eqref{Theta_k-condition} and \eqref{claim:w^kv^k->0}, we conclude that \begin{equation}\label{Dw-Dv->0}
\lim_{k\to\infty}\int_{Q_2}{ |D w^k-D v^k|^p dz }= \lim_{k\to\infty}\int_{Q_2}{ |D m^k|^p dz }=0.
\end{equation}
On the other hand, for the case $p<2$ we get
\begin{align*}
\int_{Q_2} |Dm^k|^p  dz
\leq  C \sigma^{\frac{p}{2-p}}
  &+C\sigma^{-1}\Big( \|m^k\|_{L^p(Q_3)}^p +  \|m^k\|_{L^p(Q_3)} +\|m^k\|_{L^2(Q_3)}^2\Big)\\
  &+ C \sigma^{\frac{-p}{p-1}} \Big(\| m^k\|_{L^1(Q_3)}+ \| \Theta_k\|_{L^1(Q_3)}\Big)^{\frac{\e_0}{p+\e_0}}
  \end{align*}
  for all $\sigma>0$ small. By first taking $k\to\infty$ and then taking $\sigma\to 0^+$, we still arrive at \eqref{Dw-Dv->0}.  
  
As \eqref{Dw-Dv->0} contradicts  \eqref{contradiction-conclusion}, we have produced a contradiction and the lemma is proved.
\end{proof}

\subsection{Proof of the main gradient estimate}
We need  a key $L^\infty$ gradient estimate from \cite{KuM1} to prove Theorem~~\ref{thm:main2}. This estimate is a generalization of the fundamental gradient estimate by
DiBenedetto and Friedman \cite{DF1} for the parabolic $p$-Laplace system (see also \cite[Chapter~8]{D2}). The statement below  can be deduced from 
\cite[Theorem~1.1]{KuM1} and the discussion therein.
\begin{theorem}[interior Lipschitz estimate, \cite{KuM1}]\label{thm:Lip-est}
Assume that $p>\frac{2 n}{n+2}$. Let $h$ be a weak solution of 
 \[
 h_t -\div \, \ba(x,t,Dh) =0 \quad \mbox{in}\quad  Q_3,
 \]
 with the vector field $\ba: Q_3\times \R^n\to \R^n$ satisfying the assumptions
\begin{equation*}
\left \{
\begin{array}{lcll}
   \langle \partial_\xi \ba(x,t,\xi) \eta,\eta\rangle & \geq & \Lambda^{-1} (s^2+ |\xi|^2)^{\frac{p-2}{2}}|\eta|^2,\\
|\ba(x,t,\xi)| + |\partial_\xi \ba(x,t,\xi)| \, (s^2+ |\xi|^2)^{\frac12}  &\leq &   \Lambda (s^2+ |\xi|^2)^{\frac{p-1}{2}}, \\
 |\ba(x,t,\xi)-\ba(\bar x,t,\xi)| &\leq &  \Lambda \hat\omega(|x - \bar x|) (s^2+ |\xi|^2)^{\frac{p-1}{2}}
\end{array}\right.
\end{equation*}
whenever $\xi,\, \eta\in \R^n$ and $(x,t),\, (\bar x,t)\in Q_3$. Here $s\in [0,1]$ and $\hat\omega:[0,\infty)\to [0,1]$ is a  nondecreasing function satisfying the Dini condition
\[
 \int_0 \hat\omega(\rho)\frac{d\rho}{\rho} <\infty.
\]
Then $Dh\in L^\infty_{loc}(Q_3)$ and  there exists a constant $c$ depending only on $n$, $p$, $\Lambda$, and  $\hat\omega(\cdot)$ such that 
\[
 |Dh(x_0,t_0)| \leq c \Big[\fint_{Q_r(x_0,t_0)}  (|Dh|^p +1)\, dz\Big]^{\frac{d}{p}}
\]
holds whenever $Q_r(x_0,t_0)\subset Q_3$ with $(x_0,t_0)$ is a Lebesgue point for $Dh$. The constant $d\geq 1$ is the  number given by \eqref{de:d}.
\end{theorem}
We are now ready to prove our main result.
\begin{proof}[\bf PROOF OF THEOREM~\ref{thm:main2}]
Thanks to Theorem~\ref{thm:conditionalLq}, it is enough to prove that $\A$ admits the local Lipschitz
 approximation property with constant $M_0$. We first observe that $\A$ satisfies structural conditions
 \eqref{structural-reference-1} and \eqref{structural-reference-2}.
 Indeed, the first condition in \eqref{strengthen-stru}
  implies \eqref{structural-reference-1} (see for example \cite[Lemma~1]{T}) and 
  \eqref{structural-reference-2} follows from the facts
  \[
 |\A(z,u,\xi)|\leq \Lambda  (\mu^2+ | \xi|^2)^{\frac{p-1}{2}}\mbox{ if } p\geq 2 \quad \mbox{and}\quad   |\A(z,u,\xi)|\leq \frac{\Lambda}{p-1}  (\mu^2+ | \xi|^2)^{\frac{p-1}{2}} \mbox{ if } p<2,
\]
which are consequences of the second condition in \eqref{strengthen-stru} and the assumption $\A(\cdot,\cdot,0)=0$.
 
We next verify the Lipschitz
 approximation property for $\A$. Let $\e>0$, and let $\delta_1$ and  $\delta_2$ be the  corresponding 
constants  given by Lemma~\ref{lm:compare-solution-1} and Lemma~\ref{lm:compare-solution-2} respectively.
Let $\delta :=\min{\{\delta_1, \delta_2\}}$.
Now assume that  $\lambda\geq 1$, $0<\theta<2$, $Q_{4\theta}^\lambda(\bar z)
 \subset Q_6$, $\tilde \A$ given by \eqref{eq:tildeA}, $\tilde \F$, and $\tilde f$  satisfy
\[
  \fint_{Q_4} \Big[
\sup_{u\in \overline \K}\sup_{ \xi\in\R^n}\frac{|\tilde \A(x,t, u,\xi) - \tilde \A_{B_4}(t, u,\xi)|}{1+ |\xi|^{p-1}}
\Big] \, dxdt + \fint_{Q_4}  |\tilde \F|^p \, dz
+\Big(\fint_{Q_4}  |\tilde f|^{\bar p'} \, dz\Big)^{\hat p}\leq \delta^p, 
\]
 and $\tilde u$ is a weak solution  to 
\begin{equation*}
 \tilde u_t  =  \div \tilde \A(z,\theta \hat\lambda  \tilde u,D  \tilde u) 
+\div(|\tilde \F|^{p-2} \tilde \F) +\tilde f\quad 
\text{in}\quad Q_4
\end{equation*}
with $\|\tilde u\|_{L^\infty(Q_4)}\leq M_0/\theta\hat\lambda$ and $\fint_{Q_4} |D  \tilde u|^p\, dz\leq 1$.
We want to show that \eqref{eq:appxgoodgradient} holds true. For this, we  note that $\tilde\A(\cdot,\cdot,0)=0$ and 
$\tilde\A$ satisfies conditions
\eqref{structural-reference-1}--\eqref{structural-reference-3} with the same constant $\Lambda$ as $\A$. 
 Now let $\tilde w$  be a weak solution of
\begin{equation*}\label{eq-for-w}
\left \{
\begin{array}{lcll}
\tilde w_t &=&\div \tilde\A( z,\theta\hat\lambda \tilde w,D \tilde w)  \quad &\text{in}\quad Q_\frac72, \\
\tilde w & =& \tilde u\quad &\text{on}\quad \partial_p Q_\frac72
\end{array}\right.
\end{equation*}
and $\tilde v$ be a weak solution of the frozen equation
\begin{equation}\label{tildev-eq}
\left\{
\begin{array}{lcll}
\tilde v_t &=&\div \tilde\A_{B_4}(t, \theta\hat\lambda \tilde v,D \tilde v)  \quad &\text{in}\quad Q_3, \\
\tilde v & =&\tilde  w\quad &\text{on}\quad \partial_p Q_3.
\end{array}\right.
\end{equation}
The existence of these weak solutions   is guaranteed  by Remark~\ref{existence+uniqueness}.
Also, from Proposition~\ref{pro:compa} we have
 \begin{equation}\label{vwu}
 \|\tilde v\|_{L^\infty(Q_3)}\leq \|\tilde w\|_{L^\infty(Q_\frac72)}\leq \|\tilde u\|_{L^\infty(Q_4)}\leq M_0/\theta\hat\lambda.
 \end{equation}
In addition, we  infer from  
  Lemma~\ref{gradient-est-II}  that
 \begin{align}\label{DvDwDu}
 \|D\tilde v\|_{L^p(Q_3)}
 &\leq C(\Lambda,p,n)\Big( 1+ \|D\tilde w\|_{L^p(Q_\frac72)}\Big)\nonumber\\
 &\leq C(\Lambda,p,n)\Bigg[ 1+ \|D\tilde u\|_{L^p(Q_4)} +\|\tilde \F\|_{L^p(Q_4)}
 +\Big(\int_{Q_4}  |\tilde f|^{\bar p'} \, dz\Big)^{\frac{\hat p}{p}}\Bigg]
 \leq C(\Lambda,p,n).
 \end{align}
 Therefore,  by applying Lemma~\ref{lm:compare-solution-1} and Lemma~\ref{lm:compare-solution-2}
 for $\A \rightsquigarrow  \tilde\A$ and $\alpha \rightsquigarrow \theta\hat\lambda$ we obtain
\begin{equation*}
\int_{Q_3}{|D\tilde u - D \tilde w|^p\, dz}\leq \e^p \quad\mbox{ and }\quad \int_{Q_2}{|D\tilde w - D \tilde v|^p\, dz}\leq \e^p.
\end{equation*}
Consequently, if we take $\tilde \Psi := D\tilde v$ then it follows from the triangle inequality that
\[
 \|D\tilde u-\tilde \Psi\|_{L^p(Q_2)}\leq 2\e.
 \]
Thus  it remains to show that there exists  $N>0$ depending only on $p$, $n$, $\Lambda$,  $M_0$, and $\K$ such that
 \begin{equation}\label{Dv-est}
  \|D\tilde v\|_{L^\infty(Q_2)}\leq N.
 \end{equation}
 To this end, we use   the interior H\"older regularity theory (see \cite[Theorem~1.1,~pages 41]{D2} for the case $p\geq 2$ and \cite[Theorem~1.1,~pages 77]{D2}
for the case $p<2$)  to infer that 
 there exist $\bar \alpha\in (0,1)$ and $\gamma>0$ depending only on $n$, $p$,  and $\Lambda$ such that
\begin{align}\label{tildev-0holder}
 |\tilde v(x_1, t)- \tilde v(x_2,t)|
 &\leq 
 \gamma \| \tilde v\|_{L^\infty(Q_3)} |x_1 - x_2|^{\bar \alpha}
 \leq  
 \frac{\gamma \, M_0}{\theta \hat\lambda} |x_1 - x_2|^{\bar \alpha}\quad \mbox{for all}\quad (x_1,t),\, (x_2,t)\in Q_{\frac52}.
\end{align}
Notice that the presence of  $\theta \hat \lambda$ in  \eqref{tildev-eq} does not prevent us from applying the H\"older theory as the equation satisfy all required conditions  in  \cite[Chapter~2,~page~16]{D2} with structural constants independent of   $\theta$ and $\hat \lambda$.

Now let $\ba(x,t,\xi):= \tilde\A_{B_4}(t,  \theta \hat\lambda\tilde v(x,t),\xi)$. Then  \eqref{tildev-eq} implies that $\tilde v$ satisfies
 \begin{equation*}\label{eq:v-Di}
 \tilde v_t =\div \ba(x,t,D  \tilde v)  \quad\text{in}\quad Q_3
\end{equation*}
in the weak sense. 
Let $s:=\mu/\lambda$. Observe that $\tilde \A(\cdot, \cdot,\xi)= \frac{1}{\lambda^{p-1}} \A(\cdot,\cdot,\lambda\xi)$
and $\partial_\xi \tilde \A(\cdot, \cdot,\xi)= \frac{1}{\lambda^{p-2}} (\partial_\xi\A)(\cdot,\cdot,\lambda\xi)$.  Thanks to the third condition in \eqref{strengthen-stru} and H\"older estimate \eqref{tildev-0holder}, the coefficient $\ba$ satisfies 
\begin{align*}
|\ba(x_1, t,\xi) -\ba(x_2,t,\xi)| 
 &\leq \fint_{B_4} | \tilde\A(x,t,  \theta \hat\lambda\tilde v(x_1,t),\xi) -\tilde\A(x,t, \theta \hat\lambda\tilde v(x_2,t),\xi)|\, dx\\
 &\leq \Lambda  \, \theta \hat\lambda |\tilde v(x_1, t)- \tilde v(x_2,t) |\,  \frac{1}{\lambda^{p-1}}(\mu^2+|\lambda\xi|^2)^{\frac{p-1}{2}}\\
 &\leq \Lambda  \gamma M_0 |x_1 - x_2|^{\bar \alpha} \, (s^2+|\xi|^2)^{\frac{p-1}{2}}
\end{align*}
for any $(x_1,t),\, (x_2,t)\in Q_{\frac52}$ and any $\xi\in \R^n$. Moreover, we have 
\begin{align*}
  \langle \partial_\xi \ba(x,t,\xi) \eta,\eta\rangle 
 &=\fint_{B_4} \langle \partial_\xi\tilde\A(y,t,  \theta \hat\lambda\tilde v(x,t),\xi) \eta,\eta\rangle \, dy\\
 &\geq \frac{\Lambda^{-1}}{\lambda^{p-2}}  (\mu^2+ |\lambda \xi|^2)^{\frac{p-2}{2}}|\eta|^2
  =\Lambda^{-1} (s^2+ |\xi|^2)^{\frac{p-2}{2}}|\eta|^2,\\
  |\partial_\xi \ba(x,t,\xi)| 
 &\leq \fint_{B_4} |\partial_\xi\tilde\A(y,t,  \theta \hat\lambda\tilde v(x,t),\xi) | dy 
 \leq   \frac{\Lambda}{\lambda^{p-2}}  
 (\mu^2+ |\lambda \xi|^2)^{\frac{p-2}{2}}
 =\Lambda (s^2+ |\xi|^2)^{\frac{p-2}{2}},\\
 |\ba(x,t,\xi)|
 &\leq \fint_{B_4} |\tilde\A(y,t,  \theta \hat\lambda\tilde v(x,t),\xi) | dy 
 \leq \frac{\Lambda}{\lambda^{p-1}}  
 (\mu^2+ | \lambda \xi|^2)^{\frac{p-1}{2}}=\Lambda (s^2+ |\xi|^2)^{\frac{p-1}{2}}.
\end{align*}
Since  $s=\mu/\lambda\in [0,1]$, we therefore can conclude from Theorem~\ref{thm:Lip-est} and estimate \eqref{DvDwDu}  that 
\[
\|D \tilde v\|_{L^\infty(Q_2)}\leq  C(\Lambda,p,n, M_0),
\]
which gives desired  estimate \eqref{Dv-est}. Thus the proof of Theorem~\ref{thm:main2} is complete.
\end{proof}

\begin{remark}
An alternative way of proving estimate   \eqref{Dv-est} and working directly with $\A$ instead of $\tilde \A$ is to transform equation \eqref{tildev-eq} to its original  setting and then employ the intrinsic gradient bound in \cite{KuM1}. Precisely,  let us rescale $\tilde v$  by defining
\begin{equation*}
 v (x,t) := \left \{
\begin{array}{lcll}
\theta \hat\lambda\,  \tilde v (\frac{x-\bar x}{\theta}, \frac{t - \bar t}{\lambda^{2-p} \theta^2}) &\text{if}\quad p\geq 2, \\
\theta \hat \lambda\, \tilde  v \Big(\frac{x-\bar x}{\theta \lambda^{\frac{p-2}{2}}}, \frac{t - \bar t}{ \theta^2}\Big) &\qquad  \, \text{ if}\quad \frac{2n}{n+2}<p<2.
\end{array}\right.
\end{equation*}
 Then $ v$ is a weak solution of 
 \begin{equation*}
 v_t =\div \A_{B}(t,  v,D  v)  \quad\text{in}\quad Q^\lambda_{3\theta}(\bar z),
\end{equation*}
 where $B$ is the projection of $Q^\lambda_{4\theta}(\bar z)$ onto $\R^n$, i.e. $B= B_{4\theta}(\bar x)$ if $p\geq 2$ and $B= B_{\lambda^{\frac{p-2}{2}}4\theta}(\bar x)$ if $p<2$.
 Moreover, we deduce from \eqref{vwu}--\eqref{DvDwDu} that
\begin{equation}\label{Lp-lambda}
\| v\|_{L^\infty(Q^\lambda_{3\theta}(\bar z))}\leq M_0 \quad \mbox{and}\quad 
\Big( \fint_{Q^\lambda_{3\theta}(\bar z)} |Dv|^p\, dz \Big)^{\frac1p}  \leq C(\Lambda,p,n) \, \lambda.
\end{equation}
Hence  if $p\geq 2$, then it follows from the H\"older regularity theory (see \cite[Theorem~1.1,~pages 41]{D2} )  that
 there exists $\bar \alpha\in (0,1)$ and $\gamma>0$ depending only on $n$, $p$,  and $\Lambda$ such that
\[
 |v(x_1, t)- v(x_2,t)|\leq  \gamma \|v\|_{L^\infty(Q^\lambda_{3\theta}(\bar z))} \, \Big(\frac{|x_1 - x_2|}{\theta }
 \Big)^{\bar \alpha}
 \leq \gamma M_0 \Big(\frac{|x_1 - x_2|}{\theta }
 \Big)^{\bar \alpha}
  \quad \mbox{for all}\quad (x_1,t),\, (x_2,t)\in Q^\lambda_{\frac52 \theta}(\bar z).
\]
In the case $p<2$,  we can use 
 \cite[Theorem~1.1,~pages 77]{D2} to obtain:
 \[
 |v(x_1, t)- v(x_2,t)|\leq  \gamma \|v\|_{L^\infty(Q^\lambda_{3\theta}(\bar z))} \, \Big(\frac{|x_1 - x_2|}{\lambda^{\frac{p-2}{2}}\theta }
 \Big)^{\bar \alpha}
\leq  \gamma M_0  \Big(\frac{|x_1 - x_2|}{\theta \lambda^{\frac{p-2}{2}} }
 \Big)^{\bar \alpha}
  \quad \mbox{for all}\quad (x_1,t),\, (x_2,t)\in Q^\lambda_{\frac52 \theta}(\bar z).
\]
Therefore, if we let $\ba(x,t,\xi):= \A_{B}(t,  v(x,t),\xi)$ then for any $(x_1,t),\, (x_2,t)\in Q_{\frac52\theta}^\lambda(\bar z)$ and any $\xi\in \R^n$ we have
\begin{align*}
|\ba(x_1, t,\xi) -\ba(x_2,t,\xi)| 
 &\leq \fint_{B} | \A(x,t,   v(x_1,t),\xi) -\A(x,t,  v(x_2,t),\xi)|\, dx\\
 &\leq \Lambda  | v(x_1, t)-  v(x_2,t) | \,  (\mu^2+|\xi|^2)^{\frac{p-1}{2}}\leq \Lambda  \hat\omega ( |x_1 - x_2|) \,  (\mu^2+|\xi|^2)^{\frac{p-1}{2}},
\end{align*}
where
\begin{equation*}
 \hat\omega(r)  := \left \{
\begin{array}{lcll}
\gamma M_0 \Big(\frac{r}{\theta}\Big)^{\bar{\alpha}} &\text{if}\quad p\geq 2, \\
\gamma M_0\Big(\frac{r}{\theta \lambda^{\frac{p-2}{2}}}\Big)^{\bar \alpha} &\qquad  \, \text{ if}\quad \frac{2n}{n+2}<p<2.
\end{array}\right.
\end{equation*}
Thus we  conclude from \cite[Theorem~4.1 and (4.34)]{KuM1}
that there exist $C_0=C_0(n,p,\Lambda)>1$ and $\sigma_0=\sigma_0(n,p,\Lambda,  M_0)>0$ such that: if $Q_{ \sigma \theta}^\lambda(z_0)\subset Q_{\frac52 \theta}^\lambda (\bar z)$ with $\sigma\leq \sigma_0$ and $z_0$ is a Lebesgue point of $Dv$, then
\begin{equation*}
|Dv(z_0)|\leq C_0 \Big[\fint_{Q_{\frac{\sigma}{2} \theta}^\lambda(z_0)} (|D v| +1)^p\, dz\Big]^{\frac1p}.
\end{equation*}
By replacing $\sigma_0$ by  $\min{\{\sigma_0, \frac12\}}$ if necessary, we can assume that $\sigma_0\leq \frac12$. Then for any point $z_0\in Q_{2\theta}^\lambda(\bar z)$ which is  a Lebesgue point of $Dv$, we have $Q_{ \sigma_0 \theta}^\lambda(z_0)\subset Q_{\frac52 \theta}^\lambda (\bar z)$ and hence 
\begin{equation*}
|Dv(z_0)|\leq C_0 \Big[\fint_{Q_{\frac{\sigma_0}{2} \theta}^\lambda(z_0)} (|D v| +1)^p\, dz\Big]^{\frac1p}.
\end{equation*}
Using the second estimate in \eqref{Lp-lambda} to estimate the above right hand side, we deduce that
\[
\|D  v\|_{L^\infty(Q_{2\theta}^\lambda(\bar z))}\leq  C(\Lambda,p,n,M_0)\, \lambda. 
\]
By rescaling back, we obtain 
$
\|D \tilde v\|_{L^\infty(Q_2)}\leq  C(\Lambda,p,n,M_0)$
which gives  \eqref{Dv-est}. 
\end{remark}

\section{Higher integrability of gradients}\label{sec:proof-higher} 
In this section we  prove Theorem~\ref{thm:higherint}  about  the higher integrability  of weak solutions to equation \eqref{eq:uQ4}.
The proof of this  will be given in Subsection~\ref{sub:higherint}
and  is based on the arguments  in \cite{KiL,Mi1} (see also \cite[Section~8.2]{DMS} and \cite[Lemma~12]{LMV}). 
The key ingredient is a Caccioppoli type estimate.
\subsection{Caccioppoli type estimates}
Let $\eta\in C_0^\infty(B_2(0))$ be such that $0\leq \eta\leq 1$, $\eta=1$ in $B_1(0)$, and $|D\eta|\leq 2$ in $B_2(0)$. For $\bar x\in \R^n$ and $\rho>0$, set $\eta_{\bar x, \rho}(x) = \eta\big(\frac{2(x-\bar x)}{\rho}\big)$.
We then define 
the weighted mean
\[
u^\eta_{\bar x, \rho}= u^\eta_{\bar x, \rho}(t) =\frac{\int_{B_\rho(\bar x)} \eta_{\bar x, \rho}(x)^p u(x,t)\, dx }{\int_{B_\rho(\bar x)} \eta_{\bar x, \rho}^p\, dx}.
\]
The next lemma implies the absolute continuity of $t\mapsto  u^\eta_{\bar x, \rho}(t)$.
\begin{lemma}\label{lm:AC}
 Let $p>1$ and suppose that $u$ is a weak solution of \eqref{eq:uQ4}. 
 Then for any $t_1, t_2\in (-16,16)$ with $t_1<t_2$, we have
 \[
|u^\eta_{\bar x, \rho}(t_1) -   u^\eta_{\bar x, \rho}(t_2)|\leq C(\Lambda, n, p)\Bigg[ \frac{1}{\rho^{n+1}}\int_{t_1}^{t_2}\int_{B_\rho(\bar x)}
\Big(1+
|D u|^{p-1} + |\F|^{p-1}\Big)
  dx dt + \frac{1}{\rho^n} \int_{t_1}^{t_2}\int_{B_\rho(\bar x)} |f|  dx dt\Bigg].
 \]
\end{lemma}
\begin{proof} For a function $g(x,t)$ and $h>0$, let $[g]_h (x,t)$  denote its Steklov average
defined as in \eqref{Steklov-ave}.
By using  $\phi=\eta_{\bar x, \rho}(x)^p$ in the Steklov formulation \eqref{weak-2}, we get 
 for  $t_1, t_2\in (-16,16)$ that
 \begin{align*}
 & \int_{B_\rho(\bar x)} [u \eta_{\bar x, \rho}^p]_h(\cdot,t_2) \,dx-\int_{B_\rho(\bar x)} [u \eta_{\bar x, \rho}^p]_h(\cdot,t_1) \,dx
 =\int_{t_1}^{t_2}\int_{B_\rho(\bar x)}\partial_t [u \eta_{\bar x, \rho}^p]_h(\cdot,t) \,dx dt\\
 &=   -p\int_{t_1}^{t_2}\int_{B_\rho(\bar x)} \langle [ \A(\cdot , \alpha u,D u)]_h(\cdot,t) +  [ |\F|^{p-2} \F]_h(\cdot,t), D\eta_{\bar x, \rho}\rangle
  \eta_{\bar x, \rho}^{p-1} dx dt +\int_{t_1}^{t_2}\int_{B_\rho(\bar x)} [f]_h(\cdot,t) \eta_{\bar x, \rho}^p dx dt.
    \end{align*}
 Using the growth condition \eqref{structural-reference-2} for $\A$ and the choice of  the function $\eta(x)$,
we find after passing to the limit $h\to 0^+$ that
\begin{align*}
  \Big| \int_{B_\rho(\bar x)} u(\cdot,t_2) \eta_{\bar x, \rho}^p \,dx
  &-
 \int_{B_\rho(\bar x)} u(\cdot, t_1) \eta_{\bar x, \rho}^p \,dx\Big|\\
 &\leq 
 \frac{C(\Lambda, p)}{\rho}\int_{t_1}^{t_2}\int_{B_\rho(\bar x)} \Big(1+ |D u|^{p-1} + |\F|^{p-1}\Big)
 \, dx dt +\int_{t_1}^{t_2}\int_{B_\rho(\bar x)} |f| \, dx dt
 \end{align*}
for every $t_1, t_2\in (-16,16)$ with $t_1<t_2$. This gives the lemma as desired.
\end{proof}
Lemma~\ref{lm:AC} is only used to prove the following Caccioppoli type estimate.
\begin{lemma}[Caccioppoli type estimate]\label{lm:Caccio}
 Let $p>1$ and suppose that $u$ is  a weak solution of \eqref{eq:uQ4}. Then there exists a constant $C>0$ 
 depending
 only on $\Lambda$, $n$, and $p$ such that
 \begin{align*}\label{eq:Cacci}
 \sup_{\bar t - \tau_1<t<\bar t +\tau_1}\int_{B_\rho(\bar x)} |u(x,t) &- u^\eta_{\bar x, 2\rho}(t)|^2 dx 
 +\int_{Q_{\bar z}(\rho, \tau_1)} |Du|^p dz
 \leq C \Bigg[\frac{1}{\tau_2 -\tau_1}\int_{Q_{\bar z}(2\rho, \tau_2)} |u - u^\eta_{\bar x, 2\rho}(t)|^2 dz\\ 
 &+ \frac{1}{\rho^p} \int_{Q_{\bar z}(2\rho, \tau_2)} |u - u^\eta_{\bar x, 2\rho}(t)|^p dz + 
 \int_{Q_{\bar z}(2\rho, \tau_2)} \big(1 + |\F|^p\big)\, dz +\Big( \int_{Q_{\bar z}(2\rho, \tau_2)}|f|^{\bar p'}dz \Big)^{\hat p}\Bigg]
 \end{align*}
 for all $\rho>0$ and $0<\tau_1<\tau_2$ satisfying $Q_{\bar z}(2\rho, \tau_2)\subset Q_4$.
\end{lemma}
\begin{proof}
Let $\sigma\in C_0^\infty((\bar t-\tau_2, \bar t + \tau_2))$ be such that $0\leq \sigma\leq 1$, $\sigma=1$ in $[\bar t-\tau_1, \bar t + \tau_1]$, and $|\partial_t \sigma|\leq 2/(\tau_2 -\tau_1)$. By using
$\sigma(t)  \eta_{\bar x, 2\rho}^p (u - u^\eta_{\bar x, 2\rho})$
as a test function in \eqref{eq:uQ4}, we obtain
\begin{align*}
&\frac12 \int_{\bar t -\tau_2}^\tau \sigma(t)  \frac{d}{dt}\Big[\int_{B_{2\rho}(\bar x)} \eta_{\bar x, 2\rho}^p
|u - u^\eta_{\bar x, 2\rho}|^2 dx \Big]dt
+\int_{\bar t -\tau_2}^\tau \int_{B_{2\rho}(\bar x)}\partial_t u^\eta_{\bar x, 2\rho} \sigma(t)  \eta_{\bar x, 2\rho}^p (u - u^\eta_{\bar x, 2\rho})dx dt\\
&=-\int_{\bar t -\tau_2}^\tau \int_{B_{2\rho}(\bar x)} \langle  \A(z, \alpha u,D u) +|\F|^{p-2}\F , 
D  [\eta_{\bar x, 2\rho}^p (u - u^\eta_{\bar x, 2\rho})]\rangle \sigma(t) dz\\
&\quad +\int_{\bar t -\tau_2}^\tau \int_{B_{2\rho}(\bar x)} f \eta_{\bar x, 2\rho}^p (u - u^\eta_{\bar x, 2\rho}) \sigma(t) dz.
\end{align*}
Now since $t\mapsto  u^\eta_{\bar x, 2\rho}(t)$ is absolutely continuous by Lemma~\ref{lm:AC} and thus $\partial_t u^\eta_{\bar x, 2\rho}$ is integrable on 
$(\bar t - \tau_2, \bar t +\tau_2)$, we see that the second term on the left hand side is zero. 
Furthermore, we can 
apply integration by parts for the first term and 
decompose $\A(z, \alpha u,D u)$ as $[\A(z, \alpha u,D u)-\A(z, \alpha u,0)] + \A(z, \alpha u,0)$.
Then, by using \eqref{structural-reference-1}--\eqref{structural-reference-2} and   Lemma~\ref{simple-est} with
$\xi=0$ and $\tau =1/2$ we obtain for every $\tau\in (\bar t-\tau_2, \bar t + \tau_2)$ that
\begin{align*}
&\frac12   \int_{B_{2\rho}(\bar x)} \eta_{\bar x, 2\rho}(x)^p |u(x,\tau) - u^\eta_{\bar x, 2\rho}(\tau)|^2\sigma(\tau) dx +c(\Lambda, p)\int_{K_\tau}  |D  u|^p  \eta_{\bar x, 2\rho}^p\sigma(t) dz\\
&\leq c(\Lambda, p)\int_{K_\tau}    \eta_{\bar x, 2\rho}^p\sigma(t) dz
+  \frac12 \int_{K_\tau} \sigma'(t)   \eta_{\bar x, 2\rho}^p |u - u^\eta_{\bar x, 2\rho}|^2 dz\\
&\quad + p \int_{K_\tau} \Big[ \Lambda (1+|D u|^{p-1}) +|\F|^{p-1}\Big] 
\eta_{\bar x, 2\rho}^{p-1}  | D  \eta_{\bar x, 2\rho}| \sigma(t) |u - u^\eta_{\bar x, 2\rho}| dz\\
&\quad  + \int_{K_\tau} \big(\Lambda + |\F|^{p-1}\big)
  \eta_{\bar x, 2\rho}^p \sigma(t) |D u| dz
   +\int_{K_\tau} |f| \, \big|\eta_{\bar x, 2\rho}^p (u - u^\eta_{\bar x, 2\rho}) \sigma(t) \big | \, dz,
\end{align*}
where $K_\tau := B_{2\rho}(\bar x) \times (\bar t - \tau_2, \tau)$.
Using Young's inequality, it follows that
\begin{align*}
&  \int_{B_{2 \rho}(\bar x)} \eta_{\bar x, 2\rho}(x)^p |u(x,\tau) - u^\eta_{\bar x, 2\rho}(\tau)|^2\sigma(\tau) dx +\int_{K_\tau}  |D  u|^p  \eta_{\bar x, 2\rho}^p\sigma(t) dz
\leq \frac{C}{\tau_2 - \tau_1} \int_{K_\tau} |u - u^\eta_{\bar x, 2\rho}|^2 dz\\
&+\frac{C}{\rho^p}  \int_{K_\tau}   |u - u^\eta_{\bar x, 2\rho}|^p dz
+C\int_{K_\tau}  \big(1 + |\F|^p\big) dz + C \int_{K_\tau} |f| \, \big|\eta_{\bar x, 2\rho}^p (u - u^\eta_{\bar x, 2\rho}) \sigma(t) \big| \, dz.
\end{align*}
But as in \eqref{f-sobolev}, the last integral can be estimated as follows
\begin{align*}
&\int_{K_\tau}   |f| \, \big |\eta_{\bar x, 2\rho}^p (u - u^\eta_{\bar x, 2\rho}) \sigma(t) \big|\, dz
\leq \e \int_{K_\tau}  \big|D [\eta_{\bar x, 2\rho}^p (u - u^\eta_{\bar x, 2\rho})  ]\big|^p  \sigma(t)\, dz\\
& +\e \sup_{t\in (\bar t -\tau_2,\tau)}\int_{B_{2\rho}(\bar x)} \eta_{\bar x, 2\rho}^p |u(x,t) - u^\eta_{\bar x, 2\rho}(t)|^2\sigma(t)\, dx + \frac{C(n,p)}{\e^{\frac{p+n}{p(n+1)-n}}} \|f\|_{L^{\bar p'}(K_\tau)}^{\frac{p(n+2)}{p(n+1)-n}}\quad \forall \e>0.
\end{align*}
Therefore, by choosing $\e$ suitably and using the definition of $\hat p$ in \eqref{de:d}   we  deduce that
\begin{align*}
&  \int_{B_{2 \rho}(\bar x)} \eta_{\bar x, 2\rho}(x)^p |u(x,\tau) - u^\eta_{\bar x, 2\rho}(\tau)|^2\sigma(\tau) dx +\frac12 \int_{K_\tau}  |D  u|^p  \eta_{\bar x, 2\rho}^p\sigma(t) dz\\
&\leq\frac12 \sup_{t\in (\bar t -\tau_2,\bar t +\tau_2)}\int_{B_{2\rho}(\bar x)} \eta_{\bar x, 2\rho}^p |u(x,t) - u^\eta_{\bar x, 2\rho}(t)|^2\sigma(t)\, dx+ \frac{C}{\tau_2 - \tau_1} \int_{Q_{\bar z}(2\rho, \tau_2)} |u - u^\eta_{\bar x, 2\rho}|^2 dz\\
 &\quad +\frac{C}{\rho^p}  \int_{Q_{\bar z}(2\rho, \tau_2)}   |u - u^\eta_{\bar x, 2\rho}|^p dz
+C\int_{Q_{\bar z}(2\rho, \tau_2)}  \big(1 + |\F|^p\big) dz 
+C \Big( \int_{Q_{\bar z}(2\rho, \tau_2)}|f|^{\bar p'}dz \Big)^{\hat p}.
\end{align*}
for every $\tau\in (\bar t-\tau_2, \bar t + \tau_2)$.
From this we infer the conclusion of the lemma.
\end{proof}
\begin{remark}\label{rm:Steklov}
We note that the  use of the test function in the proof of Lemma~\ref{lm:Caccio}
is justified by making use of the alternate weak formulation  \eqref{weak-2}. For simplicity,  we will 
not, however, display these technicalities.
\end{remark}

\subsection{Proof of the higher integrability}\label{sub:higherint}
We are  ready to prove  the higher integrability of gradients.
\begin{proof}[\bf PROOF OF THEOREM~\ref{thm:higherint}]
Given the Caccioppoli type estimate in Lemma~\ref{lm:Caccio}, the proof of Theorem~\ref{thm:higherint}
can be deduced  from the arguments  in \cite{KiL,Mi1}. Notice that the only use of  equation 
\eqref{eq:uQ4} is  to obtain Lemma~\ref{lm:Caccio}. For this reason and for completeness, 
 we choose to present the proof for only the case $p\geq 2$. Let 
$h(z) := 1 +|\F(z)|\,$ and $\,
 \hat f(z) := |f(z)|^{\frac{\bar p' \hat p}{p}}$.
  For $\lambda>0$, we denote 
 \begin{align*}
&E(\lambda) := \Big\{z\in Q_3:\, z\mbox{ is a Lebesgue point of $|Du|$ and }
|Du(z)| >\lambda \Big\},\\
&E_h(\lambda)  := \big\{z\in Q_3: h >\lambda \big\},\quad \mbox{ and }\quad  
E_{\hat f}(\lambda)  := \big\{z\in Q_3: \hat f >\lambda \big\}.
\end{align*}
Also, define 
\[
  \lambda_0^{\frac{p}{d}} := \fint_{Q_4} (|D u|^p +h^p)\, dz+\frac{1}{|Q_4|}\Big( \int_{Q_4}|f|^{\bar p'}dz \Big)^{\hat p}
  \quad \mbox{and}\quad \bar B^{\frac{p}{d}} := 40^{n+2}.  
 \]
Then for any  $\lambda\geq \bar B \lambda_0$, by applying a modification of Lemma~\ref{lm:covering} 
 we obtain:
 there exists a sequence of  {\it intrinsic cylinders}  $\{Q_{r_i}^\lambda(z_i)\}$ with 
$z_i\in Q_3$ and $r_i\in (0, \frac{1}{20}]$
that satisfies the following properties
\begin{enumerate}
 \item[a)] $\{Q_{4 r_i}^\lambda(z_i)\}$ is disjoint and
$ E(\lambda) \subset \bigcup_{i=1}^\infty Q_{20 r_i}^\lambda(z_i)$.
\item[b)] $\fint_{Q_{r_i}^\lambda(z_i)}\big(|Du|^p + h^p\big)\, dz +\frac{1}{|Q_{r_i}^\lambda(z_i)|}\Big( \int_{Q_{r_i}^\lambda(z_i)}|f|^{\bar p'}dz \Big)^{\hat p}= \lambda^p\,$ for each $i$.
\item[c)] $\fint_{Q_{r}^\lambda(z_i)}\big(|Du|^p + h^p\big)\, dz +\frac{1}{|Q_r^\lambda(z_i)|}
\Big( \int_{Q_r^\lambda(z_i)}|f|^{\bar p'}dz\Big)^{\hat p}  <\lambda^p\,$ for every $r\in (r_i,
 1]$.
\end{enumerate}
 Note that $Q_{20 r_i}^\lambda(z_i)\subset Q_4$ as $z_i\in Q_3$ and $20  r_i\leq 1$.\\
{\bf Claim:}
There exists  $C>0$ depending only on $\Lambda$, $n$, and $p$ such that: for each $i$, we have
 \begin{equation}\label{eq:claim}
  \fint_{Q_{20 r_i}^\lambda(z_i)} |D u|^p\, dz \leq C \Bigg[\Big(\fint_{Q_{4 r_i}^\lambda(z_i)}|Du|^{\frac{np}{n+2}} \, 
  dz\Big)^{\frac{n+2}{n}} + \fint_{Q_{4 r_i}^\lambda(z_i)} h^p\, dz
  +\frac{1}{|Q_{4 r_i}^\lambda(z_i)|}\Big( \int_{Q_{4 r_i}^\lambda(z_i)}|f|^{\bar p'}dz \Big)^{\hat p}\Bigg].
 \end{equation}
Assume the claim for the moment.  Set $q=np/(n+2)$.  Since it follows from property b) that
$\lambda^p \leq C_n \Big[
\fint_{Q_{20 r_i}^\lambda(z_i)} |D u|^p\, dz +\fint_{Q_{4 r_i}^\lambda(z_i)} h^p\, dz+\frac{1}{|Q_{4 r_i}^\lambda(z_i)|}\Big( \int_{Q_{4 r_i}^\lambda(z_i)}|f|^{\bar p'}dz \Big)^{\hat p}
\Big]$, we infer from \eqref{eq:claim} that
\begin{align*}
  \fint_{Q_{20 r_i}^\lambda(z_i)} |D u|^p\, dz +\lambda^p
  &\leq C \Bigg[\Big(\fint_{Q_{4 r_i}^\lambda(z_i)}|Du|^q \, dz\Big)^{\frac{p}{q}} +
  \fint_{Q_{4 r_i}^\lambda(z_i)} h^p\, dz+\frac{1}{|Q_{4 r_i}^\lambda(z_i)|}\Big( \int_{Q_{4 r_i}^\lambda(z_i)}\hat f^{\frac{p}{\hat p}}dz \Big)^{\hat p}\Bigg]\\
  &\leq 3C\eta^p \lambda^p + 
C\Bigg[ \Big(\frac{1}{|Q_{4 r_i}^\lambda(z_i)|}\int_{Q_{4 r_i}^\lambda(z_i) \cap E(\eta \lambda) }|Du|^q \, dz\Big)^{\frac{p}{q}}\\
 &\qquad \qquad + \frac{1}{|Q_{4 r_i}^\lambda(z_i)|}\int_{Q_{4 r_i}^\lambda(z_i) \cap E_h(\eta \lambda) }h^p \, dz +\frac{1}{|Q_{4 r_i}^\lambda(z_i)|}\Big( \int_{Q_{4 r_i}^\lambda(z_i) \cap E_{\hat f}(\eta \lambda)}\hat f^{\frac{p}{\hat p}}dz \Big)^{\hat p}\Bigg]
  \end{align*}
  for any $\eta>0$.  By choosing $\eta>0$ small and applying  H\"older inequality to the first integral on the right hand side, we then deduce that
  \begin{align*}
  \int_{Q_{20 r_i}^\lambda(z_i)} |D u|^p dz 
 &\leq C\Bigg[ 
 \Big(\fint_{Q_{4 r_i}^\lambda(z_i)  }|Du|^p \, dz\Big)^{\frac{p-q}{p}} 
 \int_{Q_{4 r_i}^\lambda(z_i) \cap E(\eta \lambda) }|Du|^q \, dz  + \int_{Q_{4 r_i}^\lambda(z_i) 
 \cap E_h(\eta \lambda) }h^p \, dz\\ &\qquad \qquad\qquad \qquad
 \qquad\qquad \qquad 
 \qquad\qquad \qquad 
 +\Big( \int_{Q_{4 r_i}^\lambda(z_i) \cap E_{\hat f}(\eta \lambda)}\hat f^{\frac{p}{\hat p}}dz \Big)^{\hat p}\Bigg]\\
 &\leq 
 C \Bigg[ \lambda^{p-q}\int_{Q_{4 r_i}^\lambda(z_i) \cap E( \eta \lambda) }|Du|^q \, dz 
  + \int_{Q_{4 r_i}^\lambda(z_i) \cap E_h(\eta \lambda) }h^p \, dz + \Big( \int_{Q_{4 r_i}^\lambda(z_i) \cap E_{\hat f}(\eta \lambda)}\hat f^{\frac{p}{\hat p}}dz \Big)^{\hat p}\Bigg],
\end{align*}
where we have used  property c) to obtain the last inequality.
This together with property a) gives
\begin{align*}
\int_{E( \lambda)} |Du|^p \, dz
&\leq C   \sum_i \Bigg[ 
\lambda^{p-q} \int_{Q_{4 r_i}^\lambda(z_i) \cap E( \eta \lambda) }|Du|^q \, dz +
\int_{Q_{4 r_i}^\lambda(z_i) \cap E_h(\eta \lambda) }h^p \, dz+\Big( \int_{Q_{4 r_i}^\lambda(z_i) \cap E_{\hat f}(\eta \lambda)}\hat f^{\frac{p}{\hat p}}dz \Big)^{\hat p}\Bigg]\\
&\leq C \Bigg[ \lambda^{p-q} \int_{E( \eta \lambda) }|Du|^q\, dz +\int_{ E_h(\eta \lambda) }h^p \, dz
+\Big( \int_{ E_{\hat f}(\eta \lambda)}\hat f^{\frac{p}{\hat p}}dz \Big)^{\hat p}\Bigg]
\qquad\qquad \forall \lambda \geq   \bar B \lambda_0.
\end{align*}
Therefore if we let $\lambda_1 := \eta^{-1} \bar B \lambda_0\geq   \bar B \lambda_0$, $\,\mu(dz) =|Du|^pdz$, $\,\hat \mu(dz) =|Du|^q dz$, $\,\nu(dz) =h^p dz$, and $\sigma(dz) =\hat f^{\frac{p}{\hat p}} dz$, then
\begin{align}\label{f-epsilon}
&\int_{E(\lambda_1)} |Du|^\e \, d\mu 
- \lambda_1^\e \mu\big(E(\lambda_0')\big) 
=\e \int_{\lambda_1}^\infty \lambda^{\e-1} \mu\big(E(\lambda)\big)\, d\lambda
\\
&\leq C\e \int_{\lambda_1}^\infty \lambda^{p-q+\e-1} \hat\mu\big(E(\eta \lambda)\big)\, d\lambda
+ C\e \int_{\lambda_1}^\infty \lambda^{\e-1} \nu\big(E_h(\eta \lambda)\big)\, d\lambda
+ C\e \int_{\lambda_1}^\infty \lambda^{\e-1} \sigma\big(E_{\hat f}(\eta \lambda)\big)^{\hat p}\, d\lambda\nonumber
\\
&= \frac{C\e}{\eta^{p-q+\e}} \int_{\eta \lambda_1}^\infty t^{p-q+\e-1} \hat\mu\big(E(t)\big)\, dt
+ \frac{C}{\eta^\e} \Bigg[\e\int_{\eta \lambda_1}^\infty t^{\e-1} \nu\big(E_h(t)\big)\, dt
+ \e\int_{\eta \lambda_1}^\infty t^{\e-1} 
\sigma\big(E_{\hat f}(t)\big)^{\hat p}\, dt\Bigg].\nonumber
\end{align}
Observe that
\begin{align*}
&(p-q+\e) \int_{\eta\lambda_1}^\infty t^{p-q+\e-1} \hat\mu\big(E(t)\big)\, dt
\leq\int_{E(\eta \lambda_1)} |Du|^{p-q+\e} \, d\hat\mu\\
&\leq \int_{E( \lambda_1)} |Du|^{p-q+\e} \, d\hat\mu 
+ \lambda_1^{\e}\int_{E(\eta\lambda_1)
\setminus E( \lambda_1)}  |Du|^{p-q}\, d\hat\mu
\leq \int_{E( \lambda_1)} |Du|^{p-q+\e} \, d\hat\mu +\lambda_1^{\e}\int_{E(\eta\lambda_1)
}  |Du|^{p-q}\, d\hat\mu
\end{align*}
and as $\lambda_1\geq 2$ we have
\begin{align*}
&\e\int_{\eta \lambda_1}^\infty t^{\e-1} \sigma\big(E_{\hat f}(t)\big)^{\hat p}\, dt
\leq \sum_{k=1}^\infty \e \int_{2^k}^{2^{k+1}} t^{\e-1} \sigma\big(E_{\hat f}(t)\big)^{\hat p}\, dt
\leq (2^\e - 1)\sum_{k=1}^\infty 2^{\e k}\sigma\big(E_{\hat f}(2^k)\big)^{\hat p}\\
&\leq (2^\e - 1) \Big[\sum_{k=1}^\infty 2^{\frac{\e}{\hat p} k}\sigma\big(E_{\hat f}(2^k)\big)\Big]^{\hat p}
\leq  \frac{ (2^\e - 1)2^\e}{(2^{\frac{\e}{\hat p}}-1)^{\hat p}} \Big(\int_{Q_3} \hat f^{\frac{\e}{\hat p}}\, d\sigma\Big)^{\hat p}
\leq  \frac{ 4^\e}{(2^{\frac{\e}{\hat p}}-1)^{\hat p}}\, \Big(\int_{Q_3}  |f|^{\bar p' (1+\frac{\e}{p})}\, dz\Big)^{\hat p},
\end{align*}
where we have used 
Remark~\ref{rm:lower-est-L^p-norm} to obtain the inequality 
right before the last one.
Therefore, it follows from \eqref{f-epsilon} that
\begin{align}\label{key-high}
\int_{E(\lambda_1)} |Du|^{p+\e} \, dz
&\leq \frac{C\e}{(p-q+\e)\eta^{p-q+\e}}  \Big[\int_{E(\lambda_1)} |Du|^{p +\e}\, dz
+\lambda_1^{\e}\int_{E(\eta\lambda_1)
}  |Du|^p\, dz \Big]\nonumber\\
&\quad +C\eta^{-\e} \Bigg[\int_{Q_3} h^{p+\e}\, dz
+\frac{4^\e}{(2^{\frac{\e}{\hat p}}-1)^{\hat p}}\, \Big(\int_{Q_3} 
|f|^{\bar p' (1+\frac{\e}{p})}\, dz\Big)^{\hat p}\Bigg] + 
\lambda_1^\e \int_{E(\lambda_1)} |Du|^p\, dz.
\end{align}
Choosing $\e>0$ small we can absorb the integral involving $|Du|^{p+\e}$ into the left hand side and get:
\begin{equation}\label{ini-higher}
\int_{E(\lambda_1)} |Du|^{p+\e} \, dz
\leq 2\lambda_1^\e \int_{E(\eta\lambda_1)} |Du|^p\, dz + C  \Big[\int_{Q_3} h^{p+\e}\, dz
+ \Big(\int_{Q_3}  |f|^{\bar p' (1+\frac{\e}{p})}\, dz\Big)^{\hat p}\Big].
\end{equation}
Notice that there is a difficulty in moving the term to the left side since it may be infinite. However,
this can be  
handled by using truncation as in the proof of \cite[Proposition~4.1]{KiL} and \cite[page~591--594]{Bo}.
The main observation there is that \eqref{key-high} still holds true if $E(\lambda_1)$ is replaced by
$E^k(\lambda_1)$ and $E(\eta\lambda_1)$ is replaced by
$E^k(\eta\lambda_1)$, where $E^k(\lambda) :=\{z\in Q_3: \, |Du|_k(z)>\lambda\}$ with 
$|Du|_k :=\min{\{|Du|, k\}}$.
As a consequence of \eqref{ini-higher}, we obtain
\begin{align*}
\int_{Q_3} |Du|^{p+\e} \, dz
&\leq \int_{E(\lambda_1)} |Du|^{p+\e} \, dz
+ \lambda_1^\e \int_{Q_3\setminus E(\lambda_1)} |Du|^p  \, dz\\
&\leq 3\lambda_1^\e \int_{Q_3} |Du|^p\, dz + C  \Big[\int_{Q_3} h^{p+\e}\, dz
+\Big(\int_{Q_3}  |f|^{\bar p' (1+\frac{\e}{p})}\, dz\Big)^{\hat p}\Big].
\end{align*}
This together with the definitions of $\lambda_1$ and $h(z)$ gives the desired estimate \eqref{higher-in}.

It remains to prove the claim.
Thanks to properties b) and c), estimate \eqref{eq:claim} will follow if we can show that
\begin{align}\label{eq:reverseHolder}
 \fint_{Q_{r_i}^\lambda(z_i)} |D u|^p\, dz
 \leq 
 C \Bigg[&\Big(\fint_{Q_{4 r_i}^\lambda(z_i)}|Du|^{\frac{n p}{n+2}} \, dz\Big)^{\frac{n+2}{n}}
 +\fint_{Q_{4 r_i}^\lambda(z_i)} h^p\,dz
 + \frac{1}{|Q_{4 r_i}^\lambda(z_i)|}
 \Big( \int_{Q_{4 r_i}^\lambda(z_i)}|f|^{\bar p'}dz \Big)^{\hat p}\Bigg].
 \end{align}
 Let us set $\bar z =z_i$ and $r:=r_i$. Then by applying Lemma~\ref{lm:Caccio} for $\rho=r$, 
 $\tau_1 = \lambda^{2-p} r^2$ and $\tau_2 = 2 \lambda^{2-p} r^2$, we obtain
  \begin{align*}
 \int_{Q_{\bar z}(r,\lambda^{2-p} r^2)} |Du|^p dz 
 &\leq 
 \frac{C}{\lambda^{2-p} r^2} \int_{Q_{\bar z}(2r, 2\lambda^{2-p} r^2)} |u - u^\eta_{\bar x, 2r}(t)|^2 dz + 
 \frac{C}{r^p} \int_{Q_{\bar z}(2r,2\lambda^{2-p} r^2)} |u - u^\eta_{\bar x, 2r}(t)|^p dz\\
 &\quad + C \Bigg[\int_{Q_{\bar z}(2r,2\lambda^{2-p} r^2)} h^p dz
 +  \Big(\int_{Q_{\bar z}(2r,2\lambda^{2-p} r^2)} |f|^{\bar p'} dz\Big)^{\hat p}\Bigg].
 \end{align*}
 Moreover, Young's inequality and property b) give for any $\e>0$ that
 \begin{align*}
 &\frac{C}{\lambda^{2-p} r^2} \int_{Q_{\bar z}(2r, 2\lambda^{2-p} r^2)} |u - u^\eta_{\bar x, 2r}(t)|^2 dz
 \leq  \e \lambda^p |Q_{\bar z}(r, \lambda^{2-p} r^2)| +\e^{\frac{2-p}{2}}\frac{C}{r^p} \int_{Q_{\bar z}(2r,2\lambda^{2-p} r^2)} |u - u^\eta_{\bar x, 2r}(t)|^p dz\\
&=\e\Bigg[
\int_{Q_{\bar z}(r, \lambda^{2-p} r^2)} (|Du|^p + h^p) dz
+\Big(\int_{Q_{\bar z}(r,\lambda^{2-p} r^2)} |f|^{\bar p'} dz\Big)^{\hat p}
\Bigg]
+\e^{\frac{2-p}{2}}\frac{C}{r^p} \int_{Q_{\bar z}(2r,2\lambda^{2-p} r^2)} |u - u^\eta_{\bar x, 2r}(t)|^p dz.
\end{align*}
 Therefore, we deduce that
\begin{align}
\label{eq:estimate-1}
 \int_{Q_{\bar z}(r, \lambda^{2-p} r^2)} |Du|^p dz
 &\leq  \frac{C}{r^p} \int_{Q_{\bar z}(2r,2\lambda^{2-p} r^2)}
 |u - u^\eta_{\bar x, 2r}(t)|^p dz\nonumber\\ 
 &\quad +  C\Bigg[ \int_{Q_{\bar z}(2r,2\lambda^{2-p} r^2)} h^p dz
 + \Big(\int_{Q_{\bar z}(2r,2\lambda^{2-p} r^2)} |f|^{\bar p'} dz\Big)^{\hat p}\Bigg]\nonumber\\
 &\leq  \frac{C}{r^p}
 \Big(\int_{Q_{\bar z}(4 r, 4 \lambda^{2-p} r^2)} |Du|^{\frac{np }{n+2}} dz\Big)
  \left(\sup_{|t-\bar t|<4 \lambda^{2-p} r^2 }\int_{B_{4 r}(\bar x)} 
  |u - u^\eta_{\bar x, 2r}(t)|^2 dx\right)^{\frac{p}{n+2}}\\
  &\quad + C \Bigg[\int_{Q_{\bar z}(2r,2\lambda^{2-p} r^2)} h^p dz 
  + \Big(\int_{Q_{\bar z}(2r,2\lambda^{2-p} r^2)} |f|^{\bar p'} dz\Big)^{\hat p}\Bigg],\nonumber
 \end{align}
where the last inequality follows from \cite[Lemma~3.3]{KiL}. We next estimate the supremum on the  right hand side. For this, we apply Lemma~\ref{lm:Caccio} for $\rho=4r $, $\tau_1 = 4 \lambda^{2-p} r^2$ and $\tau_2 =8 \lambda^{2-p} r^2$ to get
 \begin{align*}
 \sup_{|t-\bar t|<4 \lambda^{2-p} r^2}&\int_{B_{4 r}(\bar x)} |u(x,t) - u^\eta_{\bar x, 4r}(t)|^2 dx
 \leq \frac{C}{ \lambda^{2-p} r^2}
 \int_{Q_{\bar z}(8r, 8 \lambda^{2-p} r^2)} |u - u^\eta_{\bar x, 8r}(t)|^2 dz\\
 &+ \frac{C}{r^p} 
 \int_{Q_{\bar z}(8r, 8 \lambda^{2-p} r^2)} |u - u^\eta_{\bar x, 8 r}(t)|^p dz
 + C \Bigg[\int_{Q_{\bar z}(8r, 8 \lambda^{2-p} r^2)} h^p dz +
 \Big(\int_{Q_{\bar z}(8r, 8 \lambda^{2-p} r^2)} |f|^{\bar p'} dz\Big)^{\hat p}
 \Bigg].
 \end{align*}
 Using Poincar\'e inequality for functions in Sobolev spaces $W^{1,q}(B_{8r}(\bar x))$ ($q=2,p$) together with the fact
 \begin{align*}
 \Big(\int_{B_{4 r}(\bar x)} |u - u^\eta_{\bar x, 2r}(t)|^2 dx&\Big)^{\frac12}
 \leq \Big(\int_{B_{4 r}(\bar x)} |u - u^\eta_{\bar x, 4r}(t)|^2 dx\Big)^{\frac12}
 + |B_{4 r}(\bar x)|^\frac12 |u^\eta_{\bar x, 2r}(t)- u^\eta_{\bar x, 4r}(t)|\\
 &= \Big(\int_{B_{4 r}(\bar x)} |u - u^\eta_{\bar x, 4r}(t)|^2 dx\Big)^{\frac12}
 + 
 \frac{|B_{4 r}(\bar x)|^\frac12
 }{\int_{B_{2 r}(\bar x)} \eta_{\bar x, 2r}^p\, dx} \Big|
 \int_{B_{2 r}(\bar x)}[u(x,t)- u^\eta_{\bar x, 4r}(t)] \eta_{\bar x, 2r}^p dx \Big|\\
   &\leq C  \Big(\int_{B_{4 r}(\bar x)} |u - u^\eta_{\bar x, 4r}(t)|^2 dx\Big)^{\frac12},
 \end{align*}
 we infer that
 \begin{align*}
 &\sup_{|t-\bar t|<4 \lambda^{2-p} r^2}\int_{B_{4 r}(\bar x)} 
 |u(x,t) - u^\eta_{\bar x, 2r}(t)|^2 dx\\
 &\leq  C \lambda^{p-2}  \int_{Q_{\bar z}(8r, 8 \lambda^{2-p} r^2)} |Du |^2 dz 
 + C \Bigg[\int_{Q_{\bar z}(8r, 8 \lambda^{2-p} r^2)} (|Du|^p + h^p) dz
 +
 \Big(\int_{Q_{\bar z}(8r, 8 \lambda^{2-p} r^2)} |f|^{\bar p'} dz\Big)^{\hat p}\Bigg]\\
 &\leq  C \lambda^{p-2} |Q_{8r}^\lambda(\bar z)|^{\frac{p-2}{p}} \Big(\int_{Q_{8r}^\lambda(\bar z)}
 |Du |^p dz\Big)^{\frac{2}{p}} + C \Bigg[\int_{Q_{8r}^\lambda(\bar z)} (|Du|^p + h^p) dz
 +\Big(\int_{Q_{8r}^\lambda(\bar z)} |f|^{\bar p'} dz\Big)^{\hat p}\Bigg]\\
 &\leq C \lambda^p |Q_{8r}^\lambda(\bar z)|= C \lambda^2 r^{n+2}.
 \end{align*}
 This together with  \eqref{eq:estimate-1} yields
 \begin{align*}
 \fint_{Q^\lambda_r (\bar z)} |Du|^p dz 
 &\leq  C \lambda^{\frac{2p}{n+2}} 
\fint_{Q^\lambda_{4 r} (\bar z)} |Du|^{\frac{np }{n+2}} dz + C \Bigg[\fint_{Q^\lambda_{2 r} (\bar z)} 
h^p dz+ \frac{1}{|Q^\lambda_{2 r} (\bar z)|}
 \Big( \int_{Q^\lambda_{2 r} (\bar z)}|f|^{\bar p'}dz \Big)^{\hat p}\Bigg]\\
&\leq \frac12 \lambda^p + C \Big(\fint_{Q^\lambda_{4 r} (\bar z)} 
|Du|^{\frac{np }{n+2}} dz\Big)^{\frac{n+2}{n}}
+ C \Bigg[\fint_{Q^\lambda_{2 r} (\bar z)} 
h^p dz+ \frac{1}{|Q^\lambda_{2 r} (\bar z)|}
 \Big( \int_{Q^\lambda_{2 r} (\bar z)}|f|^{\bar p'}dz \Big)^{\hat p}\Bigg].
   \end{align*}
   It follows by using property b) that
   \begin{align*}
 \fint_{Q^\lambda_r (\bar z)} |Du|^p dz \leq   C
 \Bigg[ \Big(\fint_{Q^\lambda_{4 r} (\bar z)} |Du|^{\frac{np }{n+2}} dz\Big)^{\frac{n+2}{n}}
 + \fint_{Q^\lambda_{4 r} (\bar z)} h^p\, dz + \frac{1}{|Q^\lambda_{4 r} (\bar z)|}
 \Big( \int_{Q^\lambda_{4 r} (\bar z)}|f|^{\bar p'}dz \Big)^{\hat p}\Bigg].
   \end{align*}
Hence \eqref{eq:reverseHolder} is proved 
and  the proof of Theorem~\ref{thm:higherint} is complete.
\end{proof}

\begin{appendix}
  \section{A comparison principle}
 Let $\Omega_T :=\Omega\times (0,T)$ with  $T>0$ and  $\Omega$ being a bounded domain in $\R^n$, $n\geq 2$.
Let  $\K\subset \R$ be an open interval,  $
\A : \Omega_T \times \overline\K\times \R^n \longrightarrow \R^n$,
 $\Phi: \Omega_T \times \overline\K \to \R^n$ and $g,\, b: \Omega_T \times \overline\K \to \R$ be  Carath\'eodory maps.  We assume that 
 there exist constants
$\Lambda>0$ and  $1< p<\infty$ such that the following   conditions are satisfied  for a.e. $z\in \Omega_T$ and all $\xi,\eta\in\R^n$:
\begin{align}
&\big\langle  \A(z,u,\xi) -\A(z,u,\eta), \xi-\eta\big\rangle \geq 0\quad \, \, \forall u\in \overline\K,\label{a-1}\\ 
&u\in \R\mapsto g(z,u) \mbox{ is monotone nondecreasing}.
\label{a-2}
\end{align}
Also, there exist functions $K\in L^{p'}(\Omega_T)$ and $k\in L^1(0,T)$ such that
 \begin{align}
 &|\A(z,u_1,\xi)-\A(z,u_2,\xi)|  
 \leq   |u_1 - u_2| \Big(\Lambda |\xi|^{p-1} + K(z)\Big), \label{a-3}\\
 &|\Phi(z,u_1)-\Phi(z,u_2)|  \leq  |u_1 - u_2| K(z),\quad \mbox{and}
 \quad |b(z,u_1)-b(z,u_2)|  \leq  |u_1 - u_2| k(t)\label{a-4}
\end{align}
for all $u_1, u_2\in \overline\K$  with $|u_1-u_2|$ sufficiently small.

\begin{definition}\label{solution-apen}
 A map
 $
  u\in C((0,T); L^2(\Omega))\cap L^p(0,T; W^{1,p}(\Omega)) $
  is called a weak solution to 
 \begin{equation}
\label{mainPDE2}
u_t  =  \div \A(z, u,D u) +\div \Phi(z,u) -g(z,u) +b(z,u)  \quad \text{in}\quad \Omega_T.
\end{equation} 
 if $u(z)\in \overline{\K}\,\,$ for a.e. $z\in\Omega_T$ and
 \[
  \int_{\Omega_T} u \varphi_t\,dz = \int_{\Omega_T} \langle \A(z, u,D u), D \varphi\rangle\,dz
  +\int_{\Omega_T} \langle \Phi(z, u), D \varphi\rangle\,dz
  + \int_{\Omega_T} [g(z,u)-b(z,u)] \varphi\,dz
  \]
for every $\varphi\in C_0^\infty(\Omega_T)$.
\end{definition}

The following  result shows that  equation \eqref{mainPDE2} admits a comparison principle.

\begin{proposition}[comparison principle]\label{pro:compa}
Assume that $\A$, $\Phi$, $g$ and $b$ satisfy conditions \eqref{a-1}--\eqref{a-4}. 
 Let $u$ and $v$ be weak solutions to \eqref{mainPDE2} such that $u\leq v$ on $\partial_p \Omega_T$. Then
\[
u \leq v \quad \mbox{a.e. in}\quad \Omega_T.
\]
\end{proposition}
\begin{remark}
By inspecting the arguments below, ones see that we in fact only need to assume that $u$ is a weak subsolution and $v$ is a weak supersolution.
\end{remark}
\begin{proof}
 For $\e>0$ small, we define $h_\e(s)$ as in  \eqref{def:h_e}.
Let us denote $\Omega_t =\Omega\times (0,t)$. By using $h_\e(u-v)$ as a test function in the equations for $u$ and $v$ and subtracting the resulting expressions, we obtain:
\begin{align*}
&\int_{\Omega}\Big(\int_0^{(u-v)_+(x,t)}  h_\e(s)\, ds \Big)dx 
+\int_{\Omega_t}h_\e'(u-v)\langle \A(z, u,D u)-\A(z, u,D v) , Du -Dv\rangle\, dz\\
&+\int_{\Omega_t}\Big[g(z,u)-g(z,v)\Big] h_\e(u-v)
\, dz\\
&=  \int_{\Omega_t}h_\e'(u-v)\langle \A(z, v,D v)-\A(z, u,D v)  , Du -Dv\rangle\, dz\\ &-\int_{\Omega_t}h_\e'(u-v)\langle \Phi(z, u)-\Phi(z, v) , Du -Dv\rangle\, dz 
+\int_{\Omega_t}\Big[b(z,u)-b(z,v)\Big] h_\e(u-v)
\, dz
\end{align*}
for all $t\in (0,T)$. Since the second and third terms on the left-hand side are nonnegative thanks to \eqref{a-1}--\eqref{a-2}, we deduce that
\begin{align}\label{unique-ineq}
&\int_{\Omega}\Big(\int_0^{(u-v)_+(x,t)}  h_\e(s)\, ds \Big)dx \nonumber\\
&\leq   \frac{1}{\e} \int_0^t \int_{\Omega\cap \{0<u-v<\e\}}(u-v)\Big(\Lambda |Dv|^{p-1} + 2 K\Big) |Du -Dv|\, dz
+\int_{\Omega_t}k(\tau) |u-v| h_\e(u-v)
\, dx d\tau \nonumber\\
&\leq   \int_0^t \int_{\Omega\cap \{0<u-v<\e\}}\Big(\Lambda |Dv|^{p-1} + 2K\Big) |Du -Dv|\, dz
+\int_{\Omega_t}k(\tau) |u-v| h_\e(u-v)
\, dx d\tau. 
\end{align}
As $\e \to 0^+$, we have 
\[
\int_{\Omega}\Big(\int_0^{(u-v)_+(x,t)}  h_\e(s)\, ds \Big)dx
\longrightarrow \int_{\Omega} (u-v)_+(x,t)\,dx.
\]
Moreover, the first term on the right-hand side tends to zero and the last term tends to
\[
\int_{\Omega_t}k(\tau) |u-v| \sgn^+(u-v)
\, dx d\tau =\int_0^t\int_{\Omega} k(\tau) (u-v)_+
\, dx d\tau.
\]
Thus if we denote $m(\tau) :=\int_{\Omega} (u-v)_+(x,\tau)\,dx$, then by letting $\e\to 0^+$ in \eqref{unique-ineq} we obtain
\[
m(t) \leq \int_0^t k(\tau) m(\tau)\, d\tau\quad \forall t\in (0,T).
\]
Therefore, it follows from the Gr\"onwall’s inequality that $m(t)\leq 0$ for every $t\in (0,T)$. We then conclude that 
$\int_{\Omega_T} (u-v)_+(x,t)\,dx dt=0$,
and hence $u\leq v$ for a.e. in $\Omega_T$. The proof is complete.
\end{proof}
\begin{remark}\label{existence+uniqueness} 
We note that
 the comparison principle in Proposition~\ref{pro:compa} 
 together with the standard method for proving existence using Galerkin  approximation  
(see \cite[pages~466--475]{La} and \cite{L,ZL}) ensures  that: for any
$u\in  L^\infty(Q_3)\cap L^p(-9,9; W^{1,p}(B_3)) $ satisfying $u(z)\in \overline
\K\,$  for a.e. $z\in Q_3$, the Dirichlet problem 
\begin{equation*}
\left \{
\begin{array}{lcll}
v_t &=&\div \A(z, v,D v)  \quad &\text{in}\quad Q_3, \\
v & =& u\quad &\text{on}\quad \partial_p Q_3
\end{array}\right.
\end{equation*}
has a  weak solution in the sense of Definition~\ref{solution-apen}.
\end{remark}
\end{appendix}


\begin{thebibliography}{10}

\bibitem{AM} E.~Acerbi and G.~Mingione, {\it Gradient estimates for a class of parabolic systems.}
 Duke Math. J.  136  (2007),  no. 2, 285--320.

 
\bibitem{Bo}  V.~B\"ogelein, {\it Global Calder\'on-Zygmund theory for nonlinear parabolic systems.}
 Calc. Var. Partial Differential Equations  51  (2014),  no. 3-4, 555--596.

 
 \bibitem{BOR} S.-S.~Byun, J.~Ok, and S.~Ryu, {\it Global gradient estimates for general nonlinear parabolic equations in
 nonsmooth domains.}
 J. Differential Equations  254  (2013),  no. 11, 4290--4326.
 
\bibitem{BW1}  S.-S.~Byun and L. Wang, {\it Parabolic equations in Reifenberg domains.}
Arch. Ration. Mech. Anal. 176 (2005), no. 2, 271--301.


\bibitem{CP} L.~Caffarelli and I.~Peral, 
{\it    On $W^{1,p}$ estimates for elliptic equations in divergence  form.}
 Comm. Pure Appl. Math.  51  (1998),  no. 1, 1--21.
 
 
 
 
 
\bibitem{D2} E.~DiBenedetto, {\it Degenerate parabolic equations.} Universitext. Springer-Verlag, New York, 1993.

\bibitem{DF1} E.~DiBenedetto and A.~Friedman, {\it Regularity of solutions of nonlinear degenerate parabolic 
systems.}
 J. Reine Angew. Math.  349  (1984), 83--128.

\bibitem{DF2} E.~DiBenedetto and A.~Friedman, {\it H\"older estimates for nonlinear degenerate parabolic systems.}
 J. Reine Angew. Math.  357  (1985), 1--22.
 
\bibitem{DM} E. DiBenedetto and  J.~Manfredi, {\it On the higher integrability of the gradient of weak solutions of
 certain degenerate elliptic systems.}
 Amer. J. Math.  115  (1993),  no. 5, 1107--1134.
 

\bibitem{DMS}
F.~Duzaar, M.~Mingione, and K.~Steffen,  {\it Parabolic systems with polynomial growth and regularity.}
 Mem. Amer. Math. Soc.  214  (2011),  no. 1005, x+118 pp.
 
 
\bibitem{I} T.~Iwaniec, {\it Projections onto gradient fields and $L^p$-estimates for
 degenerated elliptic operators.}
 Studia Math.  75  (1983),  no. 3, 293--312.
 
 \bibitem{KiL} J.~Kinnunen and J.~Lewis, {\it Higher integrability for parabolic systems of $p$-Laplacian type.}
 Duke Math. J.  102  (2000),  no. 2, 253--271.
 
 \bibitem{KZ} J.~Kinnunen and S.~Zhou, {\it A local estimate for nonlinear equations with discontinuous
 coefficients.}
 Comm. Partial Differential Equations  24  (1999),  no. 11-12, 2043--2068.
 

 \bibitem{KuM1}
T.~Kuusi and G.~Mingione, {\it New perturbation methods for nonlinear parabolic problems.}
 J. Math. Pures Appl. (9)  98  (2012),  no. 4, 390--427.
 
 \bibitem{KuM2}
 T.~Kuusi and G.~Mingione, {\it Gradient regularity for nonlinear parabolic equations.}
 Ann. Sc. Norm. Super. Pisa Cl. Sci. (5)  12  (2013),  no. 4, 755--822.
 
 
 
 \bibitem{La} O.~ Ladyzenskaja, V.~ Solonnikov, and N.~ Ural\'tceva,
{\it Linear and quasilinear equations of parabolic type.}  Translations of Mathematical Monographs, Vol. 23,  
American Mathematical Society, 1968.

\bibitem{LMV}
C.~Leone, M.~Misawa, and A.~Verde, {\it The regularity for nonlinear parabolic systems of p-Laplacian type
 with critical growth.}  J. Differential Equations  256  (2014),  no. 8, 2807--2845.
 
 \bibitem{HNP1} L.~Hoang, T.~Nguyen, and T.~Phan,
 {\it Gradient estimates and global existence of smooth solutions to a cross-diffusion system.}
    SIAM J. Math. Anal.  47  (2015),  no. 3, 2122--2177. 
   


 \bibitem{L} J.-L. Lions.
 {\it  Quelques m\'ethodes de r\'esolution des probl\`emes aux limites non
 lin\'eaires.}
 Dunod,  1969.

 \bibitem{ME}
 N.~Meyers and A.~Elcrat, {\it Some results on regularity for solutions of non-linear elliptic systems
 and quasi-regular functions.}
 Duke Math. J.  42  (1975), 121--136.
 
 \bibitem{Mi1} M.~Misawa, {\it Partial regularity results for evolutional p-Laplacian systems with
 natural growth.}
 Manuscripta Math.  109  (2002),  no. 4, 419--454.

 
\bibitem{Mi2} M.~Misawa, {\it $L^q$ estimates of gradients for evolutional $p$-Laplacian
 systems.}
 J. Differential Equations  219  (2005),  no. 2, 390--420.

 
 \bibitem{NP} T. ~ Nguyen and T. ~Phan, {\it
 Interior gradient estimates for quasilinear elliptic equations.} 
 Calc. Var. (2016) 55: 59. doi:10.1007/s00526-016-0996-5.
 
 
 
 
 \bibitem{T} P.~Tolksdorf, {\it Regularity for a more general class of quasilinear elliptic
 equations.}
 J. Differential Equations  51  (1984),  no. 1, 126--150.
 
 
 


 \bibitem{ZL} W.~Zou and J.~Li, {\it Existence and uniqueness of bounded weak solutions for some nonlinear parabolic problems.}
 Bound. Value Probl.  2015, 2015:69.
 
 \end{thebibliography}
\end{document}